\documentclass[10pt,reqno]{amsart}
\usepackage{bbm}
\usepackage{amsmath,amsthm,amsfonts,amssymb,amscd,mathrsfs}
\usepackage{amssymb}
\usepackage{psfrag,graphicx}
\usepackage{subfigure}
\usepackage{epic,eepic}
\usepackage{color}
\usepackage{ebezier}
\usepackage{tikz-cd}

\usepackage{hyperref}
\hypersetup{
    colorlinks=true,
    linkcolor=blue,
    filecolor=magenta,
    urlcolor=cyan,
    pdftitle={Sharelatex Example},
}

\newcommand{\RN}[1]{%
  \textup{\uppercase\expandafter{\romannumeral#1}}%
}

\newtheorem{thm}{Theorem}[section]
\newtheorem{thmx}{Theorem}[section]

\newtheorem{defn}{Definition}[section]
\newtheorem{lem}{Lemma}[section]
\newtheorem{cor}{Corollary}[section]
\newtheorem{corx}{Corollary}[section]

\newtheorem{prop}{Proposition}[section]

\newtheorem{rem}{Remark}[section]

\numberwithin{equation}{section}

\def\B{\mathbb{B}}

\def\E{\mathbb{E}}
\def\N{\mathbb{N}}
\def\P{\mathbb{P}}

\def\R{\mathbb{R}}

\def\AAa{\mathcal{A}}
\def\BB{\mathcal{B}}

\def\EE{\mathcal{E}}

\def\FF{\mathcal{F}}

\def\VV{\mathcal{V}}
\def\NN{\mathcal{N}}

\def\UU{\mathcal{U}}
\def\OO{\mathcal{O}}

\def\RR{\mathcal{R}}
\def\MM{\mathcal{M}}
\def\KK{\mathcal{K}}
\def\WW{\mathcal{W}}
\def\ZZ{\mathcal{Z}}

\def\andd{\,\,\text{and}\,\,}

\def\range{{\rm ran}\,}

\def\dist{{\rm dist}}
\def\graph{{\rm graph}\,}
\def\dim{{\rm dim}\,}

\def\lan{\langle}
\def\ran{{\rangle}}

\def\ol{\overline}

\def\pa{\partial}
\def\stst{\subset\subset}
\def\sm{\setminus}

\def\al{\alpha}
\def\be{\beta}
\def\ep{\epsilon}
\def\ka{\kappa}

\def\la{\lambda}
\def\La{\Lambda}
\def\si{\sigma}

\def\om{\omega}
\def\Om{\Omega}
\def\de{\delta}

\def\ga{\gamma}
\def\Ga{\Gamma}
\def\vp{\varphi}

\begin{document}

\title[Quasi-potential of stationary and quasi-stationary densities]{Quasi-potential of stationary and quasi-stationary densities: existence, regularity, and applications}

\author{Chenchen Mou}
\address{Department of Mathematics, City University of Hong Kong, China - Hong Kong SAR}
\email{chencmou@cityu.edu.hk}
\thanks{C.M. was partially supported by port of the Hong Kong Research Grants Council (RGC) under Grants No. GRF 11311422 and GRF 11303223. W.Q. was partially supported by a start-up grant from Academy of Mathematics and Systems Science. Z.S. was partially supported by a start-up grant from the University of Alberta, NSERC RGPIN-2018-04371, and NSERC RGPIN-2024-04938. Y.Y. was partially supported by NSERC RGPIN-2020-04451 and a Scholarship from Jilin University.}

\author{Weiwei Qi}
\address{State Key Laboratory of Mathematical Sciences (SKLMS), Academy of Mathematics and Systems Science, Chinese Academy of Sciences, Beijing 100190, China}
\email{wwqi@amss.ac.cn}

\author{Zhongwei Shen}
\address{Department of Mathematical and Statistical Sciences, University of Alberta, Edmonton, AB T6G 2G1, Canada}
\email{zhongwei@ualberta.ca}

\author{Yingfei Yi}
\address{Department of Mathematical and Statistical Sciences, University of Alberta, Edmonton, AB T6G 2G1, Canada, and School of Mathematics, Jilin University, Changchun 130012, PRC}
\email{yingfei@ualberta.ca}


\begin{abstract}
The present paper is devoted to the large deviation principle (LDP), with particular emphasis on the regularity of the quasi-potential for densities of stationary and quasi-stationary distributions of randomly perturbed dynamical systems. Our framework is set up within a positively invariant set contained in the basin of attraction of a maximal attractor of the unperturbed system. Such a setting with a general maximal attractor is anticipated in many applications. 

We begin by establishing the LDP, under the almost necessary condition that the maximal attractor is an equivalence class in the sense of Freidlin-Wentzell. The LDP gives rise to the quasi-potential that is non-negative, vanishes only on the maximal attractor, has a variational representation, and is a locally Lipschitz viscosity solution of a first-order Hamilton-Jacobi equation (HJE). 

Our main purpose is to investigate the regularity of the quasi-potential. Assuming the maximal attractor is a normally contracting invariant manifold, we achieve the regularity of the quasi-potential, along with some interesting properties such as positive definiteness of its Hessian in the normal direction to the maximal attractor and uniqueness of minimizers in its variational representation. This geometric and dynamical condition on the maximal attractor is a natural but highly nontrivial generalization of a linearly stable equilibrium -- the only case in which the regularity of the quasi-potential is previously known -- and is believed to be a sharp sufficient condition for regularity. Despite introducing complex dynamical structures of the Hamiltonian system associated with the HJE, it enables us to construct a regular solution of the HJE that is nonnegative and vanishes only on the maximal attractor by means of the invariant manifold theory of this Hamiltonian system. Even without the equivalence class assumption on the maximal attractor, this regular solution exists; under the equivalence condition, it coincides with the quasi-potential, thereby establishing its regularity.

Our results have many applications including a rigorous justification of the macroscopic fluctuation theory of non-equilibrium thermodynamic systems, a theoretical foundation for the recent potential landscape and flux framework for describing emergent behaviors, and a solid foundation for further investigating the asymptotic of the prefactor, paving the way for more refined mathematical frameworks and advanced applications.

\end{abstract}





\maketitle

\tableofcontents


\section{\bf Introduction}

Randomly perturbed dynamical systems of the form
\begin{equation}\label{SDE-introduction}
    \dot{x}=b(x)+\ep\si(x)\dot{W}_{t},\quad x\in\R^{d}
\end{equation}
have been widely used in stochastic thermodynamics to study especially nonequilibrium systems arsing from many scientific areas \cite{Graham87,Risken1989,Gardiner2009}. Such a system \eqref{SDE-introduction} describes the evolution of macroscopic variables $x_i$, $i=1,\dots,d$ fluctuating around their means due to thermodynamic fluctuations and are often derived as a mesoscopic or diffusive limit of microscopic systems consisting of a large number of particles represented by $\frac{1}{\ep^{2}}$. Governing the large time asymptotic dynamics, the stationary distribution $\mu_\ep$, assumed to be unique and have a nice and positive density $u_{\ep}$, is of particular interest. For an equilibrium system that satisfies detailed balance, the stationary density follows a Gibbs density given by $u_{\ep}=\frac{1}{K_{\ep}}e^{-\frac{2}{\ep^{2}}V_{\rm eq}}$ for some potential function $V_{\rm eq}$ and normalization constant $K_{\ep}$. In this scenario, extensive theories regarding entropy production, fluctuation, phase transition, etc. have been established in terms of $V_{\rm eq}$ and its characteristics. In contrast, nonequilibrium systems, exhibiting more complex and intriguing behaviours, are much less known, due partly to the absence of a natural potential function. Establishing the large deviation principle (LDP) of the stationary distribution $\mu_{\ep}$ or its density $u_{\ep}$ to define a potential-like function for a nonequilibrium system is one of the foremost problems in stochastic thermodynamics. Whenever it is determined by
\begin{equation}\label{LDP-stationary-distrbution}
    -\inf_{B^{o}}V\leq\liminf_{\ep\to0}\frac{\ep^2}{2}\ln\mu_{\ep}(B)\leq\limsup_{\ep\to0}\frac{\ep^2}{2}\ln\mu_{\ep}(B)\leq -\inf_{\ol{B}}V,\quad\forall B\in\BB(\R^{d}),
\end{equation}
or
\begin{equation}\label{LDP-introduction}
V:=-\lim_{\ep\to0}\frac{\ep^2}{2}\ln u_{\ep},    
\end{equation}
the quasi-potential (or macroscopic potential, nonequilibrium potential) $V$ and its characteristics lay the solid foundation for understanding nonequilibrium systems (see e.g. \cite{Graham87,Graham95,HQ20,HQ21}). It also plays a crucial role in studying the noise-induced exit problem (see e.g. \cite{Day87,Day99,FW12,HS18}), provides insight for the issue of stochastic stability and bifurcations (see e.g. \cite{Graham87,Zeeman88,Graham95}), and is the building block of the recent potential landscape and flux framework for describing emergent behaviors that applies to nonequilibrium networks, ecological systems, cell biology, and so on (see e.g. \cite{WXW08,Wang15,FKLW2019,XPSLW21}).

However, establishing the LDP \eqref{LDP-stationary-distrbution} or \eqref{LDP-introduction} and investigating key properties of the quasi-potential $V$ such as regularity and variational principle are highly nontrivial, as they depend sensitively on dynamical properties of the flow $\varphi^{t}$ generated by the unperturbed system $\dot{x}=b(x)$, and when it exists, $V$ necessarily satisfies the first-order Hamilton-Jacobi equation $\nabla V^{\top} A\nabla V+b\cdot\nabla V=0$ in the sense of a viscosity solution, where $A=\si\si^{\top}$. Despite extensive studies and discussions in physics literature (see e.g. \cite{Graham87,Graham95}) regarding these issues, rigorous mathematical results are only known in a few cases. The LDP \eqref{LDP-stationary-distrbution} has been established in the case when $\varphi^{t}$ admits finitely many ``indecomposable" compact invariant sets containing all the $\omega$-limit sets (see \cite[Theorem 4.3, Chapter 6]{FW12}). The LDP \eqref{LDP-introduction} and the regularity of $V$ are unfortunately only known in the case when the flow $\varphi^{t}$ admits a singleton set global attractor $\AAa=\{x_{*}\}$, that is, $x_{*}$ is the globally asymptotically stable equilibrium \cite{Sheu86,DD85,Day87,BB09}. More precisely, the existence of the limit \eqref{LDP-introduction} in this case is established in \cite{Sheu86,Day87,BB09}. The regularity of $V$ and many other properties are obtained in \cite{DD85} under the additional non-degeneracy assumption on the equilibrium $x_{*}$.

We note that the LDP \eqref{LDP-introduction} for the stationary density $u_{\ep}$ is typically much stronger than the LDP \eqref{LDP-stationary-distrbution} for the stationary distribution $\mu_{\ep}$ in compact subsets. Such a strong result is necessary in many situations especially those involving the consideration of the noise-vanishing limit $\ep\to0$, and therefore, lays the solid foundation for some further developments that we discuss in Subsection \ref{sub-application}. 

The main purpose of the present paper is to extend previous work on the LDP \eqref{LDP-introduction} and regularity of the quasi-potential $V$ from the setting of a singleton set global attractor to that of a more general global attractor. Such a far-reaching generalization turns to be anticipated in many applications. For example, while a limit cycle is commonly used to model simple biological oscillations, more complex oscillatory behaviors require quasi-periodic or even chaotic dynamics. Our framework also encompasses quasi-stationary distributions and restricted stationary distributions, making it applicable to systems with multiple local attractors. To the best of our knowledge, this is the first time that the LDP \eqref{LDP-introduction}, the regularity of $V$ and related results are established at this level of generality.

In Subsection \ref{problem-description}, we introduce the mathematical framework and state our main results. Subsection \ref{sub-comment-approach} provides insights into our approach, highlighting key ideas and novel contributions. Finally, in Subsection \ref{sub-application}, we briefly discuss potential applications of our findings.


\subsection{Setup and main results}\label{problem-description}

Let $\Om\subset\R^{d}$ be open and connected, and $b\in C^{k}(\Om;\R^{d})$ for $k>2$. Denote by $\vp^{t}$ the local flow generated by solutions of the following ODE
    \begin{equation}\label{main-ode}
    \dot{x}=b(x),\quad x\in\Om.
\end{equation}

\medskip

\begin{itemize}
    \item[\bf(H1)] Assume that $\Om$ is positively invariant under $\vp^{t}$, namely, $\vp^{t}(\Om)\subset\Om$ for all $t\geq0$, and that $\vp^{t}$ admits the maximal attractor $\AAa$ in $\Om$, that is,  $\AAa\subset\Om$ is compact, $\vp^{t}$-invariant, and satisfies $\lim_{t\to\infty}\dist_{H}(\vp^{t}(\OO),\AAa)=0$ for all $\OO\stst\Om$.
\end{itemize}

\medskip

In {\bf(H1)}, $\dist_{H}$ denotes the Hausdorff semi-distance. Some basics of maximal attractors are collected in Appendix \ref{appendix-attractor}.

\medskip

Let $A=(a^{ij})\in C^{k+1}(\Om;\mathbf{S}^{d}_{+})$, where $\mathbf{S}^{d}_{+}$ is the set of symmetric and positive definite $d\times d$ matrices, and consider for each $0<\ep\ll1$ the following eigenvalue problem for $(u,\la)$:
\begin{equation}\label{main-problem}
    \begin{cases}
        L_{\ep}^{*}u:=\frac{\ep^{2}}{2}\nabla\cdot(\nabla\cdot(Au))-\nabla\cdot(bu)=-\la u\quad\text{in}\quad \Om,\\
        u>0\quad\text{in}\quad\Om,\quad\displaystyle\int_{\Om}u=1,\\
        \la\geq0,
    \end{cases}
\end{equation}
where $\nabla\cdot(Au)$ is the row-wise divergence. Note that $L^{*}$ is just the Fokker-Planck operator associated with the SDE \eqref{SDE-introduction} when $A=\si\si^{\top}$, and no properties of $u$ on the boundary $\partial\Om$ is required.  

\medskip

\begin{itemize}
    \item[\bf(H2)] 
Assume that for each $\ep$, \eqref{main-problem} admits a classical solution pair $(u_{\ep},\la_{\ep})$, and that the family of probability measures $\{\mu_{\ep}\}_{\ep}$ on $\Om$ with densities $\{u_{\ep}\}_{\ep}$ is tight. 
\end{itemize} 

\medskip

Under {\bf(H1)} and {\bf(H2)}, there must hold $\limsup_{\ep\to0}\frac{\ep^{2}}{2}\ln\la_{\ep}<0$ with the convention $\ln0=-\infty$ (see e.g. \cite{Friedman72/73}). Moreover, limiting points of $\{\mu_{\ep}\}_{\ep}$ as $\ep\to 0$ under the topology of weak convergence exist, are invariant measures of $\vp^{t}$, and must be supported on $\AAa$ (see e.g. \cite{Kifer88}). Note that in {\bf(H2)} we do not require the problem \eqref{main-problem} to have a unique classical solution pair. Such a setting is helpful in the application to quasi-stationary distributions, whose non-uniqueness is known in many situations (see e.g. \cite{MS94,LS00,CMS13,Yamato22}) due to singularities of $b(x)$ and $A(x)$ as $x\to\partial\Om$. We also point out that no restrictions on behaviors of $u_{\ep}$ near $\partial\Om$ is imposed except the tightness of $\{u_{\ep}\}_{\ep}$.

For $x,y\in\Omega$, we set
\begin{equation*}\label{def-V(x,y)}
\VV(y,x):=\inf_{t<0}\inf_{\phi\in \Phi_{t,y,x}}\frac{1}{4}\int_t^0 \left[b(\phi)-\dot{\phi}\right]^{\top}A^{-1}(\phi) \left[b(\phi)-\dot{\phi}\right],
\end{equation*}
where $\Phi_{t,y,x}:=\left\{\phi\in AC([t,0];\Om):\dot{\phi}\in L^2([t,0];\R^{d}), \,\, \phi(t)=y,\,\,\phi(0)=x\right\}$. It is known that $\VV:\Om\times\Om:\to[0,\infty)$ is locally Lipschitz continuous \cite{FW12}. We follow Freidlin and Wentzell \cite{FW12} to define \emph{equivalence class} and \emph{quasi-potential function}.

\begin{defn}\label{def-equiv-class}
    A set $\La\subset\Om$ is called an \emph{equivalence class} if $\VV(x,y)=\VV(y,x)=0$ for all $x,y\in\La$. 
\end{defn}

\medskip

\begin{itemize}
    \item[\bf(H3)] The maximal attractor $\AAa$ is an equivalence class.
\end{itemize}

\medskip

\begin{rem}\label{rem-indecomposibility}
Note that {\bf(H3)} is independent of $A\in C^{k+1}(\Om;\mathbf{S}^{d}_{+})$, and even, $A\in C(\Om;\mathbf{S}^{d}_{+})$. Moreover, {\bf(H3)} implies that $\AAa$ is chain-transitive, equivalent to that $\AAa$ does not contain a smaller maximal attractor (see Proposition \ref{app-prop-indecomposable-attractor}). The converse does not hold in general (it does not seem easy to construct a counterexample), but can be established under additional conditions (see Remark \ref{app-rem-indecomposable-attractor}). There exist substantial chain-transitive maximal attractors that are also equivalence classes (see e.g. \cite[page 147]{FW12}). 
\end{rem}

\medskip

The \emph{quasi-potential function} $V$ is defined by
\begin{equation}\label{def-quasi-potential-V-FW-sense}
    V(x):=\min_{y\in \AAa}\VV(y,x)=\inf_{\phi\in\Phi_{x}}I(\phi)=\inf_{\phi\in\Phi_{x}}\frac{1}{4}\int_{-\infty}^0\left[b(\phi)-\dot{\phi}\right]^{\top}A^{-1}(\phi)\left[b(\phi)-\dot{\phi}\right], \quad \forall x\in \Om,
\end{equation}
where 
$$
\Phi_{x}:=\left\{\phi\in AC((-\infty,0];\Om):\dot{\phi}\in L^2_{loc}((-\infty,0];\R^{d}), \,\,\phi(0)=x,\,\,\lim_{t\to -\infty}\dist(\phi(t),\AAa)=0\right\}.
$$
Set  $\rho_V:=\liminf_{x\to \pa\Om} V(x)$ and $\Om^V:=\left\{x\in\Om: V(x)<\rho_V\right\}$, where the exact meaning of $x\to \pa\Om$ is explained at the beginning of Subsection \ref{subsec-Variational-representation}. Denote by $\Om_{\rho}^{V}$ the $\rho$-sublevel set of $V$ in $\Om$ for each $\rho\in (0,\rho_V]$, namely, $\Om_{\rho}^{V}:=\left\{x\in\Om: V(x)<\rho\right\}$.

Let the Hamiltonian $H:\Om\times\R^{d}\to\R$ be defined by $H(x,p)=p^{\top}A(x)p+b(x)\cdot p$ and consider the first-order Hamilton-Jacobi equation (HJE) 
    \begin{equation}\label{eqn-HJE}
    H(x,\nabla W(x))=0.    
    \end{equation}

\begin{lem}\label{lem-quasi-potential-basic-results}
    Assume {\bf(H1)} and {\bf(H3)}. Then, the following hold.
    \begin{itemize}
        \item $V$ is locally Lipschitz continuous, non-negative, and vanishes only on $\AAa$.
    
\item For each $\rho\in (0,\rho_V]$, $\Om^{V}_{\rho}$ is connected. 

\item For each $x\in\Om^{V}$, $V(x)=\min_{\phi\in\Phi_{x}}I(\phi)$.

        \item $V$ is a viscosity solution of the HJE \eqref{eqn-HJE} in $\Om^{V}$, and satisfies \eqref{eqn-HJE} whenever $V$ is differentiable at $x\in\Om^{V}$.
    \end{itemize}
\end{lem}

The conclusions in Lemma \ref{lem-quasi-potential-basic-results} are more or less well-known. Their proofs are scattered in literature (see e.g. \cite{VF70,DD85}) and are also contained in the proof of Theorem \ref{thm-main-result}.

\medskip

Our first result establishes the large deviation principle (LDP) for the densities $\{u_{\ep}\}_{\ep}$ given in {\bf(H2)}.

\begin{thmx}[Large deviation principle]\label{thm-main-result}
Assume {\bf(H1)}-{\bf(H3)}. Then, for any $\al\in (0,1)$, 
$$
\lim_{\ep\to0}\frac{\ep^{2}}{2}\ln u_{\ep}=-V\quad \text{in}\quad  C^{\al}(\Om^{V}).
$$
\end{thmx}

The assumption {\bf(H3)} guarantees $\lim_{\ep\to0}\frac{\ep^{2}}{2}\ln u_{\ep}=0$ on $\AAa$. The next result almost gives an inverse to this. See Remark \ref{rem-indecomposibility} for the tricky implication relationships between {\bf(H3)} and $\AAa$ being chain-transitive.

\begin{thmx}
Assume {\bf(H1)} and {\bf(H2)}. If $\lim_{\ep\to0}\frac{\ep^{2}}{2}\ln u_{\ep}=0$ on $\AAa$, then $\AAa$ is chain-transitive.  
\end{thmx}
\begin{proof}
Suppose on the contrary that $\AAa$ is not chain-transitive. Then, it must contain a smaller maximal attractor $\tilde{\AAa}$ (see Proposition \ref{app-prop-indecomposable-attractor}), and since $\AAa$ is connected, there holds $\left(B(\tilde{\AAa})\setminus\tilde{\AAa}\right)\cap\AAa\neq\emptyset$, where $B(\tilde{\AAa})$ is the basin of attraction of $\tilde{\AAa}$. Since $\limsup_{\ep\to0}\frac{\ep^2}{2}\ln u_{\ep}(x)<0$ for $x\in B(\tilde{\AAa})\setminus\tilde{\AAa}$ (see e.g. \cite[Theorem A (2)]{JSY19}), we arrive at a contradiction by examining such an $x$ belonging to $\AAa$.
\end{proof}

\begin{rem}\label{rem-importance-of-equivalence-class}
Some comments on Theorem \ref{thm-main-result} follow.
\begin{itemize}
\item[(1)] {\bf(H3)} does not ensure that limiting measures of $\{\mu_{\ep}\}_{\ep}$ have the whole attractor $\AAa$ as their support. For example, if $\AAa$ consists of one fixed point and one homoclinic orbit, then $\AAa$ satisfies {\bf(H3)} but limiting measures of $\{\mu_{\ep}\}_{\ep}$ are supported on the fixed point. The asymptotic of the prefactor $R_{\ep}:=K_{\ep}u_{\ep}e^{\frac{\ep^2}{2}V}$, where $K_{\ep}$ is the normalization constant, needs to be investigated to determine exact concentrations in general.

\item[(2)] Even if $\lim_{\ep\to0}\frac{\ep^{2}}{2}\ln u_{\ep}$ exists on $\Om\setminus\Om^{V}$, it is unlike that the limit is given by $-V$. The behaviors of $b$ and $A$ on the boundary $\partial\Om$ play a role in determining the limit.
\end{itemize}
\end{rem}


For stationary distributions, Theorem \ref{thm-main-result} is well-known when the maximal attractor $\AAa$ is a singleton set \cite{Sheu86,Day87,BB09}, that is, $\AAa=\{x_{*}\}$ is an asymptotically stable equilibrium. Assuming further that the equilibrium $x_{*}$ is linearly stable, it is shown in \cite{DD85} that there is a connected, open, and dense subset $G\subset\Om^{V}$ containing a neighborhood of $x_{*}$ such that $V\in C^{k}(G)$. The regularity of $V$ in the whole domain $\Om^{V}$ fails in general \cite{GT85,Day87}. Such a regularity result of $V$ has many applications that we mention a few in Subsection \ref{sub-application}.  

The main purpose of the present paper is to establish the regularity of the quasi-potential function $V$. While a counterexample is not known yet, the regularity as in the case of a non-degenerate singleton set attractor is believed to fail in general when $\AAa$ has a strange structure given that $V$ must be non-negative, vanish only on $\AAa$, and solve the HJE \eqref{eqn-HJE}. Therefore, we impose proper geometric and dynamical conditions on $\AAa$ without excluding possible complex dynamics. More precisely, we require $\AAa$ to be a normally contracting invariant manifold of $\vp^{t}$, which is a natural but highly non-trivial generalization of a linearly stable equilibrium. A typical example is a linearly stable limit cycle. 

\medskip

\begin{itemize}
    \item[\bf(H4)] The maximal attractor $\AAa$ is a $C^{k}$-submanifold with $\dim\AAa<d$ and there are constants $C>0$, $0<\be_0<\be_*$ satisfying $\frac{\be_0}{\be_*-\be_0}<\frac{1}{k'}$ for some $2\leq k'\leq k-1$ such that for each $x\in\AAa$,
\begin{itemize}
\item $T_x\R^d=S(x)\oplus T_x \MM$,
\item $D\varphi^{t}(x)S(x)=S(\varphi^{t}(x))$, $t\in\R$,
\item $\|D\varphi^{t}(x)|_{S(x)}\|\leq Ce^{-\be_* t}$, $t\geq 0$, 
\item $\|D\varphi^{t}(x)|_{T_x\MM}\|\leq  Ce^{\be_0 |t|}$, $t\in \R$.
\end{itemize}
\end{itemize}

\medskip

Consider the following Hamiltonian system associated with the HJE \eqref{eqn-HJE}:
\begin{equation}\label{hamiltonian-system-intro}
    \begin{cases}
        \dot x=2 A(x)p+b(x),\\
        \dot p=-p^{\top}\nabla A(x) p-\nabla b^{\top}(x)p
    \end{cases}
    \quad\text{in}\quad\Om\times\R^{d},
\end{equation}
where $p^{\top}\nabla A(x)p=\left(p^{\top}\partial_{1}A(x)p,\dots,p^{\top}\partial_{d}A(x)p^{\top}\right)^{\top}$, $\nabla b$ denotes the Jacobian of $b$, and $\nabla b^{\top}$ is its transpose. Denote by $\Phi^{t}_{H}$ the local flow generated by \eqref{hamiltonian-system-intro}. Obviously, $\AAa\times\{0\}$ is an invariant set of $\Phi^{t}_{H}$. The following result establishes in particular the regularity of $V$ in a neighborhood of $\AAa$ and its connection to the local dynamics of $\Phi^{t}_{H}$ near $\AAa\times\{0\}$.

\begin{thmx}\label{thm-regularity-V}
Assume {\bf(H1)}-{\bf(H4)}. Then, there are open sets $\AAa\subset\OO_{2}\subset\OO_1\stst\OO_{0}\stst\Om$ such that the following hold.
\begin{itemize}
    \item (Local unstable manifold of $\AAa\times\{0\}$) There are a $d$-dimensional $C^{k'}$ Lagrangian submanifold $W^{u}\subset\Om\times\R^{d}$ and a vector field $F\in C^{k'}(\OO_{0};\R^{d})$ such that
\begin{itemize}
    \item $\Phi_{H}^{t}\WW^{u}\subset \WW^{u}$ for all $t\leq0$, and for any $\be\in(0,\be_*)$, there is $C>0$ such that 
    \begin{equation*}
    \sup_{(x,p)\in \WW^{u}}\text{\rm dist}\left(\Phi_{H}^t(x,p),\AAa\times\{0\}\right)\leq C e^{\be t},\quad \forall t\leq 0;
    \end{equation*}
    \item $\graph F=\WW^{u}\cap (\OO_0\times\R^{d})$;
    \item $F$ satisfies $H(x,F(x))=0$ for all $x\in \OO_0$, and is conservative on $\OO_1$.
\end{itemize}

    \item (Local regularity) $V\in C^{k'+1}(\OO_2)$ satisfies $\nabla V=F$ on $\OO_2$, and for each $x\in \AAa$, $\text{\rm Hess}(V)|_x:T_x\R^d \otimes T_x\R^d\to \R$ is semi-positive definite and satisfies
    \begin{equation*}
        \text{\rm Hess}(V)|_x(y,z)
        \begin{cases}
            =0,&\quad \text{if \quad $y$ or $z\in T_x\AAa$},\\
            >0,&\quad \text{if \quad $y=z\in T^{\bot}_x\AAa$}, 
        \end{cases}
    \end{equation*}
where $T^{\bot}_x\AAa\subset T_x\R^d$ denotes the orthogonal complement of $T_x\AAa$.

\item (Uniqueness of minimizers) For each $x\in\OO_{2}$, there is a unique $\phi\in\Phi_{x}$ such that $V(x)=I(\phi)$, where $I$ is given in \eqref{def-quasi-potential-V-FW-sense}. Moreover, $\phi\in C^{1}((-\infty,0];\OO_2)$ satisfies $\dot{\phi}=2A(\phi)\nabla V(\phi)+b(\phi)$, and $(\phi,F(\phi))$ solves the Hamiltonian system \eqref{hamiltonian-system-intro}.

\end{itemize}
\end{thmx}

The reader is referred to Section \ref{sec-regular-sol-HJE} for more and finer properties of $\WW^{u}$ and $F$. In particular, $\WW^{u}$ is foliated by submanifolds $\WW^{u}_{x}$ with asymptotic phase $(x,0)\in\AAa\times\{0\}$.

The next result addresses the global regularity of $V$. 

\begin{thmx}[Global regularity]\label{thm-global-regularity}
Assume {\bf(H1)}-{\bf(H4)}. Let $\WW^{u}$ and $\OO_{2}$ be as in Theorem \ref{thm-regularity-V}. Then, there exists a connected, open, and dense subset $G\subset\Om^{V}$ containing $\OO_2$ such that $V\in C^{k'+1}(G)$. Moreover, for each $x\in G$, the following hold.
\begin{itemize}
    \item There is a unique $\phi\in\Phi_{x}$ such that $V(x)=I(\phi)$.

\item $\phi\in C^{1}((-\infty,0];G)$ satisfies $\dot{\phi}=2A(\phi)\nabla V(\phi)+b(\phi)$.

\item  $(\phi,\nabla V(\phi))$ satisfies the Hamiltonian system \eqref{hamiltonian-system-intro} and there is $T_{x}>0$ such that $(\phi(t),\nabla V(\phi(t)))\in \WW^{u}$ for all $t\leq-T_{x}$.
\end{itemize}
\end{thmx}

The following result contains in particular a Helmholtz-type decomposition of the vector field $b$ that has fruitful consequences.

\begin{corx}\label{cor-macro-potential-flux}
    Assume {\bf(H1)}-{\bf(H4)}. Set $V_{\ep}:=-\frac{\ep^2}{2}\ln u_{\ep}$ and $\ga_{\ep}:=b-\frac{\ep^2}{2}\frac{\nabla\cdot (Au_{\ep})}{u_{\ep}}$. Let $G$ be as in Theorem \ref{thm-global-regularity}. The following hold.
    \begin{itemize}
        \item[(1)] $\lim_{\ep\to0}V_{\ep}=V$ in $C^{\al}(\Om^{V})$ for any $\al\in (0,1)$.
        \item[(2)] $\lim_{\ep\to0}\ga_{\ep}=\ga$ in the weak sense, where $\ga=b+A\nabla V$.
        \item[(3)] $\nabla V\cdot\ga=0$ in $G$.
        \item[(4)] $\|b\|^{2}_{A^{-1}}=\|A\nabla V\|^{2}_{A^{-1}}+\|\ga\|^{2}_{A^{-1}}$ in $G$, where $\|\ell\|_{A^{-1}}=\sqrt{\ell\cdot A^{-1}\ell}$ for a vector field $\ell$ on $\Om$.
        \item[(5)] $b\cdot\nabla V=-\|A\nabla V\|^{2}_{A^{-1}}$ in $G$. 
    \end{itemize}
\end{corx}
\begin{proof}
(1) is exactly Theorem \ref{thm-main-result}. Note that $\ga_{\ep}=b-\frac{\ep^2}{2}\nabla \cdot A +A\nabla V_{\ep}$. By (1), it is easy to see that $\lim_{\ep\to 0}\nabla V_{\ep}=\nabla V$ in the weak sense. This yields (2). The orthogonality in (3) is just the HJE $H(x,V(x))=0$, which holds for $x\in G$ thanks to Lemma \ref{lem-quasi-potential-basic-results} and Theorems \ref{thm-main-result} and \ref{thm-regularity-V}. For (4), we note that (3) implies $A\nabla V$ and $\ga$ are orthogonal with respect to the inner product $\lan\ell,\tilde{\ell}\ran_{A^{-1}}:=\ell\cdot A^{-1}\tilde{\ell}$ for vector fields $\ell,\tilde{\ell}$ on $\Om$. Since $b=-A\nabla V+\ga$, the conclusion follows from the Pythagorean theorem. The conclusion in (5) is an immediate consequence of (3).
\end{proof}




\subsection{Comments on approaches}\label{sub-comment-approach}

Given Theorem \ref{thm-regularity-V}, the proof of Theorem \ref{thm-global-regularity} follows directly from the strategy laid out in \cite[Section 5]{DD85}. Below, we briefly comment on strategies for proving Theorems \ref{thm-main-result} and \ref{thm-regularity-V} as well as related results and approaches.

\medskip

\paragraph{\bf Comments on the proof of Theorem \ref{thm-main-result}} Previously, the LDP as in Theorem \ref{thm-main-result} was only established for stationary densities in the case of $\Om=\R^{d}$ and a singleton set maximal attractor $\AAa$ \cite{Sheu86,Day87,BB09} and for quasi-stationary densities of a class of one-dimensional singular diffusions \cite{QSY2024}. In \cite{Sheu86}, the author considered a special class of systems whose vector field $b$ admitting the decomposition $b=-A\nabla I+\ell$ with $\ell\cdot\nabla I=0$. An extension was made in \cite{Day87}, in which the key idea is to pass the LDP for transition probability densities $p_{\ep}(t,x,y)$ to that for stationary densities as $t\to\infty$, relying on the LDP for stationary distribution as in \eqref{LDP-stationary-distrbution}. The required LDP for $p_{\ep}(t,x,y)$ reads that for each $t>0$ the limit $\lim_{\ep\to0}\frac{\ep^2}{2}\ln p_{\ep}(t,x,y)$ exists locally uniformly in $x,y\in\R^{d}$. In \cite{BB09}, the authors took a control theoretic approach, formulating an optimization problem for $V_{\ep}:=-\frac{\ep^2}{2}\ln u_{\ep}$ and defining a control problem to capture its limit, to re-establish the result in \cite{Day87}, albeit under stronger dissipative conditions. These conditions ensure that the resulting quasi-potential $V$ must satisfy $\lim_{|x|\to\infty}V(x)=\infty$, playing a technical role in arguments. The approach in \cite{QSY2024} for treating densities of quasi-stationary distributions is elementary thanks to the one-dimensional nature, and therefore, can not be generalized to higher-dimensional problems. 

Both the probabilistic approach in \cite{Day87} and the control theoretic one in \cite{BB09} are by no means restrictive to the singleton set attractor case. In consideration of the absence of the LDP for transition probability densities and possible technical troubles to establish them due to factors including a general domain $\Omega$, a possible degenerate and singular diffusion matrix $A$ on $\partial\Omega$, and no restrictions on behaviors of $\{u_{\ep}\}_{\ep}$ near $\partial\Omega$ except the tightness, we choose to leverage the control theoretic approach in \cite{BB09} that fits better into our setting. Given weak conditions near $\partial\Omega$, we must adapt the global technique in \cite{BB09} to develop a localization argument that is much more technical. In particular, we have to work on bounded domains, deal with exit events, and pass to the domain limit.

\medskip

\paragraph{\bf Comments on the proof of Theorem \ref{thm-regularity-V}}

Only in the case of a linearly stable equilibrium attractor $\AAa=\{x_*\}$ is the regularity of the quasi-potential $V$ previously known \cite{DD85}. In this case, $(x_*,0)$ is a hyperbolic fixed point of the Hamiltonian system \eqref{hamiltonian-system-intro}. Then, Hamiltonian trajectories corresponding to curves minimizing $V$ (see Lemma \ref{thm-extremal-orbit}) must eventually (for negative time) stay on the local unstable manifold of $(x_*,0)$, giving rise to a natural connection between the local unstable manifold and the quasi-potential $V$, from which the regularity of $V$ follows readily. 

The situation in our setting is entirely different. The invariant manifold $\AAa\times\{0\}$ is not normally hyperbolic but admits a center direction with unclear dynamics on the associated center manifold, which is non-unique in general. As a result, Hamiltonian trajectories corresponding to curves minimizing $V$ may eventually stay on the center manifold instead of transitioning to the local unstable manifold of $\AAa\times\{0\}$, and hence, the connection between the local unstable manifold and $V$ remains unsolved.  

To circumvent this issue, we take an approach that builds on a local uniqueness result of the HJE \eqref{eqn-HJE} (see Lemma \ref{lem-unique-minimizer}) asserting that if $\OO\subset\Om$ is an open set containing $\AAa$ and $W\in C^{1}(\OO)$ satisfies
\begin{equation}\label{eqn-HJE-BVP}
    \begin{cases}
        H(x,\nabla W(x))=0,\,\, x\in\OO,\\
        W=0\,\,\text{on}\,\,\AAa,\,\,\text{and}\,\, W>0\,\,\text{on}\,\,\OO\setminus\AAa,
    \end{cases}
\end{equation}
then $W=V$ in a smaller open set $\AAa\subset\OO'\subset\OO$, reducing to the construction of such a $W$ that is done within two steps. Note that if $\AAa$ is a non-characteristic hypersurface, then the method of characteristics can be applied. However, this is never the case since $\AAa$ seldom has dimension $d-1$, and even if it does, it must be characteristic because $\AAa$ is an invariant manifold of $\vp^{t}$ in this case and the diffusion matrix $A$ is non-degenerate near $\AAa$. 

First, we consider the Hamiltonian system \eqref{hamiltonian-system-intro} and construct an $d$-dimensional local unstable manifold of the invariant manifold $\AAa\times\{0\}$, which is actually a Lagrangian submanifold and given by the graph of a $C^{k'}$ vector field $F$ obeying $H(x,F(x))=0$. A tricky ingredient is to identify the unstable direction of $\AAa\times\{0\}$ and determine that it is not perpendicular to the $x$-space so that (i) the invariant splitting of the linearization of \eqref{hamiltonian-system-intro} along $\AAa\times\{0\}$ can be established, (ii) the classical invariant manifold theory (see Appendix \ref{app-Fenichel-theory}) can be applied to yield the existence of the local unstable manifold of $\AAa\times\{0\}$, and (iii) the local unstable manifold, naturally constructed as the graph of some function of the unstable direction, can be represented as the graph of a function $F$ of $x$.

It is then shown that $F$ is conservative, and hence, is a gradient field given by a $C^{k'+1}$ function $\hat{V}$ defined by $\hat{V}(x)=\int_{-\infty}^0 L(X_{x,F(x)}, \dot X_{x,F(x)})$, where $L(x,v)=\frac{1}{4}\left(b(x)-v\right)^{\top}A^{-1}(x)\left(b(x)-v\right)$ is the Lagrangian associated with the Hamiltonian $H$, and $X_{x,F(x)}$ is the $x$-component of the unique solution of \eqref{hamiltonian-system-intro} on $(-\infty,0]$ with initial condition $(x,F(x))$. Hence, $W:=\hat{V}$ satisfies required conditions. 

The reader is referred to Remarks \ref{rem-general} and \ref{rem-key-ideas} for more technical comments on our approach. It is worthwhile to point out that constructing such an $W$ does not require $\AAa$ to be an equivalence class, which however is needed for the local uniqueness result of \eqref{eqn-HJE-BVP}. Actually, we directly work with a normally contracting invariant manifold $\MM$ when constructing $W$ that solves \eqref{eqn-HJE-BVP} with $\AAa$ replaced by $\MM$. Such a setting is of independent interest. 

After proving Theorem \ref{thm-regularity-V}, we manage to show that Hamiltonian trajectories corresponding to curves minimizing $V$ transition to the local unstable manifold of $\AAa\times\{0\}$ (see Lemma \ref{lem-minimizer-on-unstable-manifold}). This result plays a crucial role in the proof of the global regularity of $V$, namely, Theorem \ref{thm-global-regularity}.


\subsection{Applications}\label{sub-application}

Our framework applies directly to stationary distributions (and its restriction to subdomains) and quasi-stationary distributions (QSDs) of randomly perturbed dynamical systems. The existence and uniqueness of stationary distributions have been extensively studied. In contrast, QSDs on unbounded domains often exist without uniqueness unless strong dissipative conditions are imposed. However, the tightness could be a serious issue especially for stationary distributions restricted to subdomains and QSDs. They are discussed in Subsections \ref{subsec-application-sd} and \ref{subsec-application-qsd}. 

These results for stationary and quasi-stationary distributions rigorously justify the macroscopic fluctuation theory of nonequilibrium thermodynamic systems described by randomly perturbed dynamical systems. This theory gives rise to the theoretical foundation for the recent potential landscape and flux framework for describing emergent behaviors in for instance living systems. Details are presented in Subsection \ref{subsec-macro-potential-fluctuation}. 

The LDP in Theorem \ref{thm-main-result} determines the leading exponential asymptotic of $u_{\ep}$.  Examining the asymptotic of the prefactor $R_{\ep}:=K_{\ep}u_{\ep}e^{\frac{\ep^2}{2}V}$, where $K_{\ep}$ is the normalization constant, turns out to be necessary in many applications. The higher regularity of $V$ ensures the well-behavedness of coefficients appearing in the equation satisfied by $R_{\ep}$, and therefore, lays the solid foundation for further investigating the asymptotic of $R_{\ep}$. Details are discussed in Subsection \ref{subsec-subexponential-LDP}.


\subsection{Organization of the rest of the paper} 

In Section \ref{sec-LDP}, we study the LDP for $\{u_{\ep}\}_{\ep}$ and prove Theorem \ref{thm-main-result}. Section \ref{sec-regular-sol-HJE} is devoted to the construction of a regular solution of the HJE \eqref{eqn-HJE} in a neighborhood of a normally contracting invariant manifold $\MM$ of $\vp^{t}$ that is non-negative and vanishes only on $\MM$. Theorems \ref{thm-regularity-V} and \ref{thm-global-regularity} are proven in Section \ref{sec-regularity-quasi-potential}. Applications are discussed in Section \ref{sec-application}. Appendices \ref{appendix-attractor}-\ref{sec-app-energy-balance-equation} are included as supplements to the main context.


\section{\bf Large deviation principle}\label{sec-LDP}

This section is devoted to studying the large deviation principle (LDP) of the densities $\{u_{\ep}\}_{\ep}$. In particular, we prove Theorem \ref{thm-main-result}. Through this section, we assume {\bf(H1)}-{\bf(H3)}.

In Subsection \ref{subsec-Logarithmic transformation}, we study the limit of $V_\epsilon:=-\frac{\epsilon^{2}}{2}\ln u_{\epsilon}$ as $\ep\to0$ along subsequences, resulting in a particular limiting function $V^{x_{0}}$. Subsection \ref{subsec-Variational-representation} is devoted to establishing a variational representation for $V^{x_{0}}$. In Subsection \ref{subsec-proof-LDP}, we identify $V^{x_{0}}$ with the quasi-potential function $V$ defined in \eqref{def-quasi-potential-V-FW-sense}, and hence, complete the proof of Theorem \ref{thm-main-result}.

\subsection{Logarithmic transformation}\label{subsec-Logarithmic transformation}

Recall that $u_{\ep}$ and $\la_{\ep}$ solve \eqref{main-problem}. The equation $L_{\ep}^{*}u_{\ep}=-\la_{\ep}u_{\ep}$ can be rewritten as 
\begin{equation}\label{eq:HJBuepsilon0}
    \frac{\epsilon^2}{2}\sum_{i,j=1}^d a^{ij}\partial_{ij}^2u_{\epsilon}- b_{\epsilon}\cdot\nabla u_{\epsilon}+c_\epsilon u_{\epsilon}=0,
\end{equation}
where $\partial_{i}:=\partial_{x_{i}}$, $\partial_{ij}^{2}:=\partial^{2}_{x_{i}x_{j}}$, $b_{\epsilon}:=-\ep^{2}\nabla\cdot A+b$
and $c_\epsilon:=\frac{\epsilon^2}{2}\sum_{i,j=1}^d\partial_{ij}^2a^{ij}+\lambda_\epsilon-\nabla\cdot b$. It is straightforward to verify that the logarithmically transformed function $V_\epsilon:=-\frac{\epsilon^{2}}{2}\ln u_{\epsilon}$ satisfies
\begin{equation}\label{eq:Aepsilon0}
    \frac{\epsilon^2}{2}\sum_{i,j=1}^{d}a^{ij}\partial_{ij}^{2}V_{\epsilon}-b_{\epsilon}\cdot\nabla V_{\epsilon}-\nabla V_{\epsilon}^{\top}A\nabla V_{\epsilon}-\frac{\epsilon^2}{2}c_{\epsilon}=0.
\end{equation}

The local uniform boundedness of $\{\nabla V_{\ep}\}_{\ep}$ is established in the following lemma.

\begin{lem}\label{prop-uniform-gradient-V_ep}
For each open $\OO\stst \Om$, there exists $C=C(\OO,\Om)>0$ such that $\sup_{\ep}\|\nabla V_{\epsilon}\|_{L^\infty(\OO)}\leq C$.
\end{lem}
\begin{proof}
Given the regularity assumptions on $b$ and $a^{ij}$, the classical elliptic theory ensures that $u_{\ep}\in C^3(\Om)$. Since $u_{\ep}>0$ in $\Om$, it follows that $w_{\ep}:=\ln u_{\ep}\in C^3(\Om)$. Noting that $w_{\ep}=-\frac{2}{\ep^2}V_{\ep}$, we see from \eqref{eq:Aepsilon0} that $w_{\ep}$ satisfies
\begin{equation*}
    \sum_{i,j=1}^d a^{ij}\pa^2_{ij}w_{\ep}+\nabla w_{\ep}^{\top} A\nabla w_{\ep}-H_{\ep}\cdot \nabla w_{\ep}=G_{\ep},
\end{equation*}
where $H_{\ep}=\frac{2}{\ep^2}b_{\ep}$ and $G_{\ep}=-\frac{2}{\ep^2}c_{\ep}$. For each open $\OO\subset\subset\Om$, we apply \cite[Lemma 5.1]{MPR05} to find a domain $\OO'$ with $\OO\stst \OO'\stst\Om$ and a positive constant $C_1$, depending only on $\la_{\OO'}$ -- the minimum eigenvalue of $(a^{ij})$ in $\ol{\OO'}$ -- and $\|a^{ij}\|_{C_b^2(\OO')}$ such that  
\begin{equation*}
    \|\nabla w_{\ep}\|_{L^{\infty}(\OO)}\leq C_1\sup_{\OO'}\left( 1+|H_{\ep}|+|D^2 H_{\ep}|+|G_{\ep}|+|DG_{\ep}|\right).
\end{equation*}
We remark that the original statement of  \cite[Lemma 5.1]{MPR05} requires $(a^{ij})$ to be uniformly elliptic in $\R^d$, $a^{ij}\in C^2_b(\R^d)$, $H_{\ep}\in C^2(\R^d)$ and $G_{\ep}\in C^1(\R^d)$, but the proof only depends on local properties of the coefficients. So a straightforward adaptation of the proof leads to the above inequality.

Thus, there exists $C_2>0$ depending on $\la_{\OO'}$, $\|a^{ij}\|_{C^3_b(\OO')}$ and $\|b^i\|_{C^2_b(\OO')}$ such that $\|\nabla w_{\ep}\|_{L^{\infty}(\OO)}\leq \frac{C_2}{\ep^2}$. The desired result follows readily from the fact that $V_{\ep}=-\frac{\ep^2}{2}w_{\ep}$. 
\end{proof}

In the next result, we study the limiting behaviors of $V_{\ep}$ as $\ep\to0$.

\begin{prop}\label{prop-Convergent-subsequence-V^x_0}
Each sequence $\{\ep_n\}_n$ with $\lim_{n\to \infty}\ep_n=0$ has a subsequence, still denoted by $\{\ep_n\}_n$, such that the following hold: there is a point $x_0\in \AAa$ and a locally Lipschitz continuous function $V^{x_0}:\Om\to\R$ such that 
\begin{enumerate}
\item $\lim_{n\to\infty}V_{\ep_n}=V^{x_0}$ in $C^{\al}(\Om)$ for any $\al\in (0,1)$, 

\item $V^{x_0}(x_0)=0$, $V^{x_0}\geq0$ on $\AAa$ and $V^{x_0}>0$ in $\Om\sm \AAa$,

\item $V^{x_0}$ is a viscosity solution of the HJE \eqref{eqn-HJE}.    
\end{enumerate}    
\end{prop}
\begin{proof}
Since $\{\mu_{\ep}\}_{\ep}$ is tight by assumption, each sequence $\{\ep_n\}_n$ satisfying $\lim_{n\to \infty}\ep_n=0$ has a subsequence, still denoted by $\{\ep_n\}_n$, such that $\mu_{\ep_n}$ weakly converges to some probability measure $\mu_0$ on $\Om$. It is well-known (see e.g. \cite{Kifer88}) that $\mu_0$ is an invariant measure of \eqref{main-ode} and supported on $\AAa$.     

The proof is broken into three steps.

\medskip

\paragraph{\bf Step 1} There is $x_0\in\AAa$ such that $\mu_0(\NN)>0$ for any open neighbourhood $\NN$ of $x_0$.

Suppose on the contrary that for each $x\in \AAa$, there exists an open neighbourhood $\NN_x$ of $x$ such that $\mu_0(\NN_x)=0$. Obviously, $\{\NN_x\}_{x\in \AAa}$ forms an open cover of $\AAa$. The compactness of $\AAa$ ensures the existence of a finite open cover $\{\NN_{x_i}\}_{i=1}^N$. Then, $1=\mu_0(\AAa)\leq \sum_{i=1}^N \mu_0(\NN_{x_i})=0$, leading to a contradiction.

\medskip

\paragraph{\bf Step 2} We show that up to extracting a subsequence,  $\lim_{n\to \infty}V_{\ep_n}(x_0)=0$.

Up to extracting a subsequence, we may assume without loss of generality that $\ell:=\lim_{n\to \infty}V_{\ep_n}(x_0)\in[-\infty,\infty]$ exists. We claim that
\begin{itemize}
    \item if $\ell<0$, then there is $r_{0}>0$ such that
\begin{equation}\label{claim-june-19}
V_{\ep_n}(x)\leq \frac{1}{2}\max\{-1,\ell\},\quad\forall x\in B_{r_0}(x_0)\andd n\gg 1;  
\end{equation}

\item if $\ell>0$, then there is $r'_0>0$ such that
\begin{equation}\label{claim-june-19-1}
    V_{\ep_n}(x)>\frac{1}{2}\min\{1,\ell\},\quad\forall x\in B_{r'_0}(x_0)\andd n\gg 1.
\end{equation}
\end{itemize}
We finish the proof assuming the above claim. Since $u_{\ep_n}=e^{-\frac{2}{\ep_n^2}V_{\ep_n}}$, we see that if $\ell<0$, then \eqref{claim-june-19} yields that $1\geq\int_{B_{r_0}(x_0)}u_{\ep_n}\geq e^{-\frac{1}{\ep_n^2}\max\{-1,\ell\}}|B_{r_0}(x_0)|\to \infty$ as $n\to \infty$, leading to a contradiction. Similarly, if $\ell>0$, then \eqref{claim-june-19-1} implies that $\int_{B_{r'_0}(x_0)}u_{\ep_n}\leq e^{-\frac{1}{\ep_n^2}\min\{1,\ell\}}|B_{r'_0}(x_0)|\to 0$ as $n\to \infty$, which contradicts {\bf Step 1} and the weak convergence $\lim_{n\to\infty}\mu_{\ep_{n}}=\mu_0$. As a result, $\ell=0$.

It remains to show \eqref{claim-june-19} and \eqref{claim-june-19-1}. We here only prove \eqref{claim-june-19}; the proof of \eqref{claim-june-19-1} follows similarly.
\begin{itemize}
    \item If $-\infty<\ell<0$, we see from Lemma \ref{prop-uniform-gradient-V_ep} that  there exists $0<r_1\ll 1$ such that 
    \begin{equation*}
    \sup_{x\in B_{r_1}(x_0)}|V_{\ep_n}(x)-V_{\ep_n}(x_0)|\leq r_1 \sup_n \|\nabla V_{\ep_n}\|_{L^{\infty}(B_{r_1}(x_0))}\leq \frac{|\ell|}{4}.
    \end{equation*}
Since $\lim_{n\to \infty}V_{\ep_n}(x_0)=\ell$, we find $|V_{\ep_n}(x_0)-\ell|\leq \frac{|\ell|}{4}$ for all $n\gg1$, and thus, 
\begin{equation*}
\begin{split}
V_{\ep_n}(x)&\leq \ell+ \left|V_{\ep_n}(x)-V_{\ep_n}(x_0) +V_{\ep_n}(x_0)-\ell\right|\leq \ell+\frac{|\ell|}{2}=\frac{\ell}{2}, \quad\forall x\in B_{r_1}(x_0)\andd n\gg 1.
\end{split}
\end{equation*}

\item If $\ell=-\infty$, we see from Lemma \ref{prop-uniform-gradient-V_ep} that 
$\sup_{x\in B_{r_2}(x_0)}|V_{\ep_n}(x)-V_{\ep_n}(x_0)|\leq \frac{1}{2}$ for some $0<r_2\ll 1$. Since $\lim_{n\to \infty}V_{\ep_n}(x_0)=\ell$, there holds $V_{\ep_n}(x_0)\leq -1$ for all $n\gg1$, and thus, 
$$
V_{\ep_n}(x)\leq V_{\ep_n}(x_0)+ |V_{\ep_n}(x)-V_{\ep_n}(x_0)|\leq -1+\frac{1}{2}=-\frac{1}{2},\quad\forall x\in B_{r_2}(x_0)\andd n\gg 1.
$$
\end{itemize}
Setting $r_0:=\min\{r_1,r_2\}$, we arrive at \eqref{claim-june-19}. This completes the proof in this step.

\medskip

\paragraph{\bf Step 3} We finish the proof. It follows from $\lim_{n\to \infty}V_{\ep_n}(x_0)=0$ (by {\bf Step 2}) and Lemma \ref{prop-uniform-gradient-V_ep} that for each open $\OO\subset\subset\Om$, $\{V_{\ep_n}\}_{n}$ is uniformly bounded and equicontinuous on $\ol{\OO}$ and there is $C_{\ol{\OO}}>0$ such that 
$\sup_{n}\|V_{\ep_n}\|_{Lip(\ol{\OO})}\leq C_{\ol{\OO}}$. We apply Arzel\'{a}-Ascoli theorem and the standard diagonal argument to find a subsequence, still denoted by $\{V_{\ep_n}\}_n$, and a function $V^{x_0}\in C(\Om)$ such that 
$\lim_{n\to\infty}V_{\ep_n}=V^{x_0}$ locally uniformly in $\Om$. In particular, $V^{x_0}(x_0)=0$ in (2) holds. Since $\OO$ is arbitrary and
$$
\frac{|V^{x_0}(x)-V^{x_0}(y)|}{|x-y|}=\lim_{n\to \infty}\frac{|V_{\ep_n}^{x_0}(x)-V_{\ep_n}^{x_0}(y)|}{|x-y|}\leq C_{\ol{\OO}},\quad x,y\in \ol{\OO},\,\,x\neq y,
$$
we conclude that $V^{x_0}$ is locally Lipschitz continuous, and hence, (1) holds. 

Noting that \eqref{eq:Aepsilon0} is equivalent to
\begin{equation}\label{eq:Aepsilon1}
    -\frac{\epsilon^2}{2}\sum_{i,j=1}^{d}a^{ij}\partial_{ij}^{2}V_{\epsilon}+\max_{\alpha\in \mathbb R^d}\left[(b_{\epsilon}+\alpha)\cdot\nabla V_{\epsilon}-\frac{1}{4}\alpha^{\top} A^{-1}\alpha+\frac{\epsilon^2}{2}c_\epsilon\right]=0,
\end{equation}
we follow standard arguments in the theory of viscosity solutions (see e.g. \cite[Proposition VI.1]{CL83}) to find that $V^{x_0}$ is a viscosity solution of the equation
\begin{equation}\label{steadyHJB}
    \max_{\alpha\in \R^d}\left[(b+\alpha)\cdot\nabla W-\frac{1}{4}\alpha^{\top} A^{-1}\alpha\right]=0\quad \text{in} \quad \Om,
\end{equation}
which is just the Hamilton-Jacobi equation in (3).

Next, we show that $V^{x_0}\geq 0$ on $\AAa$ in (2). Suppose on the contrary that $V^{x_0}(x_*)<0$ for some $x_*\in \AAa$. Then, $\lim_{n\to\infty}V_{\ep_n}(x_*)=V^{x_0}(x_*)$ thanks to (1). Following arguments leading to \eqref{claim-june-19}, we find $r>0$ such that $V_{\ep_n}(x)\leq \frac{1}{2}\max\left\{-1, V^{x_0}(x_*)\right\}<0$ for all $x\in B_r(x_*)$ and $n\gg 1$. This together with $u_{\ep_n}=e^{-\frac{2}{\ep_n^2}V_{\ep_n}}$ results in $1\geq \int_{B_r(x_*)} u_{\ep_n}\geq e^{-\frac{1}{\ep^2_n}\max\{-1,V^{x_0}(x_*)\}}|B_r(x_*)|\to \infty$ as $n\to \infty$, leading to a contradiction.

It remains to show $V^{x_0}>0$ in $\Om\sm \AAa$ in (2). It is actually a simple consequence of \cite[Theorem A (2)]{JSY19}, where the authors show that any $\OO\stst \Om\setminus\AAa$, there are $\ga_{\OO}>0$ and $0<\ep_{\OO}\ll1$ such that
$\sup_{\OO}u_{\ep}\leq e^{-\frac{\ga_{\OO}}{\ep^{2}}}$ for all $\ep\in(0,\ep_{\OO})$. It follows from $V_{\ep}=-\frac{\ep^2}{2}\ln u_{\ep}$ that $\liminf_{n\to \infty}V_{\ep_n}>0$, and hence, $V^{x_0}>0$ in $\Om\setminus\AAa$. We point out that the original estimates were developed for stationary distributions in \cite{JSY19}; however the arguments are still applicable here since $\lim_{\ep\to0}\la_{\ep}=0$. This completes the proof.
\end{proof}

\subsection{Variational representation}\label{subsec-Variational-representation}

In this subsection, we prove a variational representation for the function $V^{x_0}$ derived in Proposition \ref{prop-Convergent-subsequence-V^x_0} by adapting the control theoretical approach proposed in \cite{BB09}, but we have to modify the global technique in \cite{BB09} to develop much more technical localization arguments in order to deal with exit events caused by weak conditions imposed on $\partial\Om$. 

We introduce some notations before stating the result. Set
$$
\rho_*:=\liminf_{x\to \pa\Om}V^{x_0}(x)\quad \andd\quad \Om_*:=\left\{x\in \Om: V^{x_0}(x)< \rho_*\right\}.
$$
The meaning of $x\to \pa\Om$ should be clarified. If $\Om$ is bounded, it is just in the usual sense. Otherwise, we need to use the homeomorphism $h$ from the extended Euclidean space $\R^d\cup \pa \R^d$ to the closed unit ball $\B:=\{x\in \R^d: |x|\leq 1\}$. Here, $\pa\R^d:=\{x_*^{\infty}:x_*\in \B\}$ with $x_*^{\infty}$ denoting the infinity element of the ray through $x_*$. Then, we define $\pa\Om$ as the preimage of $\pa h(\Om)$ and regard $x\to \pa\Om$ if $h(x)\to h(\pa \Om)$.

For each $\rho\in (0,\rho_*]$, we denote by $\Om_{\rho}$ the $\rho$-sublevel set of $V^{x_0}$ in $\Om$, namely,
$$
\Om_{\rho}=\Om_{\rho}(x_0):=\{x\in\Om: V^{x_0}(x)<\rho\}.
$$
In particular, $\Om_*=\Om_{\rho_*}$. Obviously, $x_0\in   \Om_{\rho_{1}}\stst \Om_{\rho_{2}}$ and $\lim_{x\to \pa\Om_{\rho_2}}V^{x_0}(x)=\rho_2$ for all $0<\rho_{1}<\rho_{2}\leq\rho_{*}$. Note that the dependence of $\rho_*$, $\Om_*$ and $\Om_{\rho}$ on $x_0$ are suppressed for the sake of simplicity. 

For $\al\in\Gamma:=L_{loc}^{2}([0,\infty);\R^{d})$ and $x\in\Om$, Carath\'eodory's existence theorem and assumptions on $b$ ensure the existence of a unique absolutely continuous solution $X_{\al,x}$ of the following control system:
\begin{equation}\label{IVP-control}
    \begin{cases}
    \dot X(t)=-b(X(t))-\alpha(t),\quad t>0,\\
    X(0)=x.
    \end{cases}
\end{equation}
Note that $X_{\al,x}(t)$ may not exist for all $t\geq0$, but this shall cause no trouble as restrictions will be imposed on when it is used. Whenever it is well-defined, we set for $t\in(0,\infty]$,
$$
I_{\al,x}(t):=\frac{1}{4}\int_0^{t}\alpha(s)^{\top}A^{-1}(X_{\alpha,x}(s))\alpha(s)ds.
$$



Below is the main result in this subsection. It, in particular, gives the variational representation for $V^{x_0}$ in $\Om_*$. Set 
\begin{equation}\label{eqn-Gamma-x}
\Gamma_x:=\left\{\alpha\in\Gamma:X_{\alpha,x}(t)\in\Om \text{ for all $t\geq 0$ and }\lim_{t\to \infty}{\rm dist}(X_{\alpha,x}(t),\AAa)=0\right\}.    
\end{equation}

\begin{prop}\label{lem:lim}
The following hold.
\begin{itemize}
    \item[(1)]  For each $x\in \Om_*$, $V^{x_0}(x)=\min_{\alpha\in\Gamma_x}I_{\al,x}(\infty)$.

    \item[(2)] For each $\rho\in (0,\rho_*]$, $\Om_{\rho}$ is connected and $\AAa\stst \Om_{\rho}$.
\end{itemize}

\end{prop}

The rest of this subsection is devoted to the proof of Proposition \ref{lem:lim}. The key step, circumventing difficulties caused by the lack of information about $V^{x_0}$ on $\partial\Om$, is to establish a local version of the variational representation for $V^{x_0}$ (see Lemma \ref{lem:lim-1}). 


We need some notations and preliminary results. Let $\OO_{\AAa,\Om}$ be the set of open and connected sets $\OO$ satisfying $\AAa\subset \OO\stst \Om$. For each $\OO\in\OO_{\AAa,\Om}$, we denote by $\tau_{\al,x,\OO}$ the first time that $X_{\al,x}$ exits from $\OO$, namely, $\tau_{\al,x,\OO}:=\inf\left\{t\geq 0: X_{\al,x}(t)\in \pa\OO\right\}$ with the convention that $\inf \emptyset =\infty$, and set $\rho_{\OO}:=\min_{\pa\OO}V^{x_0}$. Proposition \ref{prop-Convergent-subsequence-V^x_0} ensures that $\rho_{\OO}$ is well-defined and positive. For $\rho\in (0,\rho_{\OO}]$, let $\OO_{\rho}$ be the $\rho$-sublevel set of $V^{x_0}$ in $\OO$, namely, $\OO_{\rho}:=\left\{x\in\OO: V^{x_0}(x)<\rho\right\}$. It is easy to see that $x_0\in\OO_{\rho_{1}}\stst \OO_{\rho_{2}}$, $\lim_{x\to \pa\OO_{\rho_2}}V^{x_0}(x)=\rho_2$ for all $0<\rho_{1}<\rho_{2}\leq\rho_{\OO}$ and $\overline{\OO}_{\rho_{\OO}}\cap\partial\OO\neq\emptyset$.

\begin{lem}\label{representation-of-V}
    Let $\OO\in\OO_{\AAa,\Om}$. Then,
    \begin{equation*}
    V^{x_0}(x)=\inf_{\al\in\Gamma}\left[I_{\al,x}(t\land\tau)+V^{x_0}(X_{\al,x}(t\land \tau))\right],\quad\forall (x,t)\in \ol{\OO}\times[0,\infty),
\end{equation*}
where $\tau=\tau_{\al,x,\OO}$.
\end{lem}
\begin{proof}


Fix $\OO\in\OO_{\AAa,\Om}$ and set
\begin{equation*}
    \tilde V^{x_0}(x,t):=\inf_{\al\in\Gamma}\left[I_{\al,x}(t\land\tau)+V^{x_0}(X_{\al,x}(t\land \tau))\right],\quad \forall (x,t)\in \ol{\OO}\times[0,\infty).
\end{equation*}
Since $V^{x_0}$ is a viscosity solution of \eqref{steadyHJB} and is Lipschitz continuous on $\overline{\OO}$ (see Proposition \ref{prop-Convergent-subsequence-V^x_0}), we see that $\tilde V^{x_0}$ is a Lipschitz viscosity solution (see e.g. \cite{CL83}) of the following problem:
\begin{equation*}\label{tildeV}
\begin{cases}
\partial_t W+\sup_{\alpha\in\R^d} \left[(b+\alpha)\cdot\nabla W-\frac{1}{4}\alpha^{\top}A^{-1}\alpha\right]=0  \quad\text{in}\quad\OO\times(0,\infty),\\
W = V^{x_0}\quad\text{on}\quad \left(\pa\OO\times (0,\infty)\right)\cup \left(\ol{\OO}\times \{0\}\right).    
\end{cases}
\end{equation*} 
Obviously, $V^{x_0}$ is also a viscosity solution of the above problem. By the uniqueness result in \cite[Theorem VI.1]{CL87}, we conclude that $V^{x_0}=\tilde V^{x_0}$. The lemma follows.
\end{proof}

\begin{lem}\label{lem-vanishing-on-M}
$V^{x_0}=0$ on $\AAa$. In particular, $\AAa\stst \OO_{\rho}$ for each $\rho\in (0,\rho_{\OO}]$ and $\OO\in \OO_{\AAa,\Om}$.  
\end{lem}
\begin{proof}
Fix $x\in \AAa$. Since $\VV(x,x_0)=\VV(x_0,x)=0$ by {\bf (H3)}, there exist absolutely continuous functions $\hat\phi_n:[-t_n,0]\to \Om$ with $t_n>0$ for $n\in\N$ satisfying $\hat\phi_n(0)=x$ and $\hat\phi_n(-t_n)=x_0$ such that 
\begin{equation}\label{approximation-june-20}
\frac{1}{4}\int_{-t_n}^0 \left[b(\hat\phi_n)-\dot{\hat\phi}_n\right]^{\top}A^{-1}(\hat\phi_n ) \left[b(\phi_n)-\dot{\hat\phi}_n\right]\leq \frac{1}{n}.
\end{equation}

For each $n\in \N$, we set $\phi_n(t):=\hat\phi_n(-t)$ and $\al_n(t):=-b(\phi_n(t))-\dot{\phi}_n(t)$ for $t\in [0,t_n]$ and $\al_n(t)=0$ for $t>t_n$. Clearly, $\al_n\in \Ga$. Note that
\begin{equation*}
    \begin{cases}
        \dot{\phi}_n(t)=-b(\phi_n(t))-\al_n(t),\quad t\in(0,t_{n}],\\
        \phi_n(0)=x.
    \end{cases}
\end{equation*}
That is, $X_{\al_n,x}=\phi_n$ on $[0,t_n]$ and $X_{\al_n,x}(t_n)=x_0$. 
The absolute continuity of $X_{\al_n,x}$ yields the existence of  $\OO_{n}\in\OO_{\AAa,\Om}$ such that ${\range}(X_{\al_n,x}|_{[0,t_n]})\subset\OO_{n}$. Since $I_{\al_n, x}(t_n)\leq \frac{1}{n}$ due to \eqref{approximation-june-20}, we derive from Lemma \ref{representation-of-V} and the fact $\tau_{\al_{n},x,\OO_{n}}>t_{n}$ that
\begin{equation*}
\begin{split}
    V^{x_0}(x)&\leq \limsup_{n\to \infty}\left[I_{\al_{n},x}(t_{n})+V^{x_0}(X_{\al_n,x}(t_n))\right]
    =V^{x_0}(x_0)=0.
\end{split}
\end{equation*}
As $V^{x_0}(x)\geq0$ thanks to Proposition \ref{prop-Convergent-subsequence-V^x_0} (2), we conclude $V^{x_0}(x)=0$. The ``In particular" part follows readily from the definition of $\OO_{\rho}$.  This finishes the proof. 
\end{proof}

The following lemma establishes the variational formula for $V^{x_0}$ in the largest sublevel set $\OO_{\rho_{\OO}}$ of $V^{x_0}$ in $\OO$.

\begin{lem}\label{lem:lim-1}
Let $\OO\in\OO_{\AAa,\Om}$. Then, the following hold.
\begin{itemize}
    \item[(1)] For each $x\in \OO_{\rho_{\OO}}$,
\begin{equation}\label{formula-V^x_0-in-O}
    V^{x_0}(x)=\min_{\alpha\in\Gamma_{x,\OO}}I_{\al,x}(\infty), 
\end{equation}
where $\Gamma_{x,\OO}:=\left\{\alpha\in\Gamma:\tau_{\al,x,\OO}=\infty\,\,\text{and}\,\,\lim_{t\to \infty}{\rm dist}(X_{\alpha,x}(t),\AAa)=0\right\}$.

\item[(2)] For each $\rho\in (0,\rho_{\OO}]$, $\OO_{\rho}$ is connected.
\end{itemize}
\end{lem}
\begin{proof}

Since $V^{x_0}|_{\AAa}=0$ by Lemma \ref{lem-vanishing-on-M} and $\AAa$ is invariant under $\vp^t$, the variational representation \eqref{formula-V^x_0-in-O} is trivial by selecting $\al\equiv 0$ if $x\in\AAa$. The rest of the proof is broken into four steps. In {\bf Steps 1-3}, we fix $x\in \mathbb \OO_{\rho_{\OO}}\setminus \AAa$ and prove \eqref{formula-V^x_0-in-O}. {\bf Step 4} is devoted to proving the connectedness of $\OO_{\rho}$.

\medskip 

\paragraph{\bf Step 1} There holds $V^{x_0}(x)\leq\inf_{\al\in\Gamma_{x,\OO}}I_{\al,x}(\infty)$.

Indeed, the fact $\Ga_{x,\OO}\subset \Ga$ and Lemma \ref{representation-of-V} yield
\begin{equation*}
\begin{split}
    V^{x_0}(x)
    &\leq \inf_{\al\in\Gamma_{x,\OO}}\left[I_{\al,x}(t\land\tau)+V^{x_0}(X_{\al,x}(t\land \tau))\right],\quad\forall t\geq0,
\end{split}
\end{equation*}
where $\tau=\tau_{\al,x,\OO}$. Letting $t\to\infty$ in the above inequality, we conclude this step from $\tau=\infty$ due to the definition of $\Gamma_{x,\OO}$ and $V^{x_0}|_{\AAa}=0$ (by Lemma \ref{lem-vanishing-on-M}).

\medskip

\paragraph{\bf Step 2} We show the existence of $\al_x\in \Gamma$ such that $\tau_{\al_x,x,\OO}=\infty$ and 
\begin{equation}\label{implicit-representation-V-june-20}
    V^{x_0}(x)=I_{\al_{x},x}(t)+V^{x_0}(X_{\al_x,x}(t)),\quad\forall t\geq0.
\end{equation}
Moreover, $t\mapsto V^{x_0}(X_{\alpha_{x},x}(t))$ is decreasing on $[0,\infty)$.

Recall that $V_{\ep}$ satisfies \eqref{eq:Aepsilon1} and the maximum is attained at $\alpha=2A\nabla V_{\epsilon}$, which is uniformly bounded in $\OO$ thanks to Lemma \ref{prop-uniform-gradient-V_ep}. Thus, there is a compact set $K_\OO\subset \R^d$ such that
\begin{equation}\label{eq:Aepsilon2}
\begin{split}
&-\frac{\epsilon^2}{2}\sum_{i,j=1}^da^{ij}\partial_{ij}^{2}V_{\epsilon}+\max_{\alpha\in\mathbb R^d}\left[(b_{\epsilon}+\alpha)\cdot \nabla V_{\epsilon}-\frac{1}{4}\alpha^{\top}A^{-1}\alpha+\frac{\epsilon^2}{2}c_{\epsilon}\right]\\
&\qquad =-\frac{\epsilon^2}{2}\sum_{i,j=1}^da^{ij}\partial_{ij}^{2}V_{\epsilon}+\max_{\alpha\in K_{\OO}}\left[(b_{\epsilon}+\alpha)\cdot\nabla V_{\epsilon}+\frac{1}{4}\alpha^{\top}A^{-1}\alpha+\frac{\epsilon^2}{2}c_{\epsilon}\right]=0,
\end{split}
\end{equation}
when $V_{\ep}$ is considered in $\OO$. Proposition \ref{prop-Convergent-subsequence-V^x_0} and standard arguments in the theory of viscosity solutions ensure that $V^{x_0}$ is a viscosity solution of $\max_{\alpha\in K_{\OO}}\left[(b+\alpha)\cdot\nabla W-\frac{1}{4}\alpha^{\top}A^{-1}\alpha\right]=0$ in $\OO$. Given this, we can follow the proof of Lemma \ref{representation-of-V} to show 
\begin{equation}\label{eqn-2023-06-24-1}
    V^{x_0}(x)=\inf_{\al\in\Gamma_{\OO}}\left[I_{\al,x}(t\land\tau)+V^{x_0}(X_{\al,x}(t\land \tau))\right],\quad\forall t\geq0,
\end{equation}
where $\Gamma_{\OO}:=L^{2}_{loc}([0,\infty);K_{\OO})$. 

For each $t\geq 0$, we define  the functional $\mathcal{F}_t:\Ga_{\OO}\to\mathbb R$ as follows: 
\begin{equation*}
  \mathcal{F}_t(\alpha):=I_{\al,x}(t\land\tau)+V^{x_0}(X_{\alpha,x}(t\land \tau)),
\end{equation*} 
and claim 
\begin{equation}\label{eqn-2023-07-07}
\exists\,\al\in \Ga_{\OO}\,\,\,\text{s.t.}\,\,\,\tau >t\,\,\,\text{and}\,\,\,V^{x_0}(x)=\FF_t(\al).
\end{equation}

To prove \eqref{eqn-2023-07-07}, we fix $t>0$ and choose a minimizing sequence $\{\al_n\}_n\subset \Ga_{\OO}$ for $\FF_t$, namely, $\lim_{n\to \infty}\FF_{t}(\al_n)=V^{x_0}(x)$. Denote $X_n:=X_{\al_n,x}$ and $\tau_n:=\tau_{\al_n,x,\OO}$ for simplicity. As $\KK_{\OO}$ is compact and $\dot{X}_n=-b(X_n)-\al_n$, it is not hard to deduce that $\{X_n\}_n$ is uniformly bounded and equicontinuous over any finite interval. By Arzel\`a–Ascoli theorem and the standard diagonal argument, we find a subsequence of $\{X_n\}_n$, still denoted by $\{X_n\}_n$, such that $X_n$ converges locally uniformly to some Lipschitz continuous function $X$ on $[0,\infty)$. 

Clearly, $\{\al_n\}_n$ is relatively compact in $L^2([0,s],K_{\OO})$ under the weak topology  for any $s\geq 0$ thanks to the compactness of $\KK_{\OO}$. Then, we apply the standard diagonal argument to find a subsequence of $\{\al_n\}_n$, still denoted by $\{\al_n\}_n$, such that $\al_n$ weakly converges to some $\al\in \Ga_{\OO}$. It is easy to verify that $\dot{X}=-b(X)-\al$, indicating $X_{\al,x}=X$. Then, $\tau\leq \liminf_{n\to \infty} \tau_n$. 

By standard programming arguments, $\{\al_n\}_n$ is also a minimizing sequence for $\FF_{t\land \tau}$. Note that 
\begin{equation*}
    \begin{split}
        \mathcal{F}_{t\land \tau}(\alpha_n)&=\frac{1}{4}\int_0^{t\land \tau\land \tau_n}\alpha_n^{\top}A^{-1}(X)\alpha_n+\frac{1}{4}\int_0^{t\land \tau\land \tau_n}\alpha_n^{\top}\left(A^{-1}(X_n)-A^{-1}(X)\right)\alpha_n\\
        &\quad +V^{x_0}(X_n(t\land \tau\land \tau_n))=:\RN{1}_n+\RN{2}_n+\RN{3}_n.
    \end{split}
\end{equation*}
The continuity of $A$ and $V^{x_0}$ and the fact $\sup_n \|\al_n\|_{L^2([0,t])}<\infty$ ensure that $\lim_{n\to \infty}(\RN{2}_n+\RN{3}_n)=V^{x_0}(X(t\land \tau))$, where we used the fact $\lim_{n\to \infty}t\land \tau\land \tau_n=t\land\tau$. Noting that $\al\mapsto\al^{\top}A^{-1}(X)\al$ is convex and $\lim_{n\to \infty}\tau\land \tau_n=\tau$, we apply the standard arguments in control theory to derive $\int_0^{t\land \tau}\al^{\top}A^{-1}(X)\al ds\leq \liminf_{n\to \infty}\RN{1}_n$. As a result, $\FF_{t\land \tau}(\al) \leq \liminf_{n\to \infty}\FF_t(\al_n)=V^{x_0}(x)$. This together with \eqref{eqn-2023-06-24-1} leads to  $V^{x_0}(x)=\FF_{t\land \tau}(\al)\geq V^{x_0}(X(t\land \tau))$. Since $V^{x_0}(x)< \rho_{\OO}$ and $V^{x_0}(X(\tau))\geq \rho_{\OO}$ (if $\tau<\infty$), we conclude that $\tau>t$ and $V^{x_0}(x)=\FF_{t}(\al)=I_{\al,x}(t)+V^{x_0}(X_{\al,x}(t))$. This proves the claim \eqref{eqn-2023-07-07}. 

Set $\Gamma_{x,t}:=\left\{\alpha\in\Gamma_{\OO}: V^{x_0}(x)=\mathcal{F}_t(\alpha)\right\}$, which is nonempty by \eqref{eqn-2023-07-07}. Moreover, the proof of \eqref{eqn-2023-07-07} indicates the closedness of $\Ga_{x,t}$ under the topology of weak convergence over closed intervals. We show that
\begin{equation}\label{claim-2023-12-05}
\Gamma_{x,t}\subset\Gamma_{x,s}\quad\text{for}\quad s<t.
\end{equation}
To see this, we fix $s<t$ and take $\al\in \Gamma_{x,t}$. Then, 
\begin{equation*}
    \begin{split}
     V^{x_0}(x)=I_{\al,x}(t)+V^{x_0}(X_{\alpha,x}(t))=I_{\al,x}(s)+\frac{1}{4}\int_{s}^{t}\alpha^{\top}A^{-1}(X_{\alpha,x})\alpha+V^{x_0}(X_{\alpha,x}(t)).
    \end{split}
\end{equation*}
A standard dynamic programming argument asserts $\frac{1}{4}\int_{s}^{t}\alpha^{\top}A^{-1}(X_{\alpha,x})\alpha+V^{x_0}(X_{\alpha,x}(t))=V^{x_0}(X_{\alpha,x}(s))$, resulting in $V^{x_0}(x)=I_{\al,x}(s)+V^{x_0}(X_{\alpha,x}(s))$. Hence, $\al\in \Gamma_{x,s}$, verifying \eqref{claim-2023-12-05}. 

Note that $\cap_{n\in \N}\Gamma_{x,n}\neq\emptyset$. Indeed, if $\al_n\in \Gamma_{x,n}$ for $n\in \N$, then the sequence $\{\al_n\}_n$ when restricted on $L^2([0,m];K_{\OO})$ is relatively compact under the weak topology for each $m\in\N$. Applying the standard diagonal argument results in the existence of $\al\in \Gamma_{\OO}$ being the limit of $\al_n|_{[0,m]}$ for any $m$. Moreover, the closedness of $\Ga_{x,n}$ for any $n\in \N$ ensures that $\al\in \cap_{n\in \N}\Gamma_{x,n}$.

Let $\al_x\in\cap_{n\in\N}\Gamma_{x,n}$. Then, $V^{x_0}(x)=\mathcal{F}_t(\alpha_{x})=I_{\al_{x},x}(t)+V^{x_0}(X_{\al_x,x}(t))$ for all $t\geq 0$, implying that $\tau_{\al_x,x,\OO}=\infty$ and $V^{x_0}(X_{\alpha_{x},x}(t))$ is decreasing in $t$. 

\medskip

\paragraph{\bf Step 3} We show $\al_x\in \Ga_{x,\OO}$. Consequently, letting $t\to \infty$ in \eqref{implicit-representation-V-june-20}, we arrive at $V^{x_0}(x)=I_{\al_{x},x}(\infty)$, which together with {\bf Step 1} yields \eqref{formula-V^x_0-in-O}.

Obviously, there is $\rho_{x}\in(0,\rho_{\OO})$ such that $x\in \OO_{\rho_x}$. The monotonicity of $t\mapsto V^{x_0}(X_{\alpha_{x},x}(t))$ (by {\bf Step 2}) implies $X_{\alpha_{x},x}(t)\in \OO_{\rho_x}$ for all $t\in [0,\tau_{\al_x,x,\OO})$. Since $\OO_{\rho_x}\stst \OO$, we see that $X_{\al_x,x}$ does not exit from $\OO$, and thus,  $\tau_{\al_x,x,\OO}=\infty$.

It remains to show $\lim_{t\to\infty}{\rm dist}(X_{\alpha_{x},x}(t),\AAa)=0$. By Proposition \ref{prop-Convergent-subsequence-V^x_0} (2) and Lemma \ref{lem-vanishing-on-M}, we only need to prove $\lim_{t\to \infty}V^{x_0}(X_{\alpha_{x},x}(t))=0$. Since $t\mapsto V^{x_0}(X_{\alpha_{x},x}(t))$ is decreasing, the limit $\underline{\rho}:=\lim_{t\to \infty}V^{x_0}(X_{\alpha_{x},x}(t))$ exists. Suppose on the contrary that $\underline{\rho}>0$.

Consider the time-reversed version of \eqref{main-ode}, that is,
\begin{equation}\label{ode-time-reversed}
\dot{x}=-b(x).    
\end{equation}
The assumption on the dynamics of \eqref{main-ode} implies that \eqref{ode-time-reversed} has no invariant set in $\OO_{\rho_{x}}\sm\OO_{\underline{\rho}}$ and all the orbits starting in $\OO_{\rho_x}\sm\OO_{\underline{\rho}}$ enter $\BB_{\AAa}\sm \OO$ after some fixed time. From this, it is easy to see that $X_{\al_x,x}$, satisfying $X_{\al_x,x}(0)\in \OO_{\rho_x}\sm\OO_{\underline{\rho}}$, is far away from solving \eqref{ode-time-reversed}. We apply \cite[Lemma 3.1]{F78} to conclude that
\begin{equation*}
    \begin{split}
        I_{\al_{x},x}(\infty)=\frac{1}{4}\int_{0}^{\infty}\left[b(X_{\al_{x},x})+\dot{X}_{\al_{x},x}\right]^{\top}A^{-1}(X_{\al_{x},x})\left[b(X_{\al_{x},x})+\dot{X}_{\al_{x},x}\right]=\infty.
    \end{split}
\end{equation*}
Letting $t\to \infty$ in \eqref{implicit-representation-V-june-20} yields $V^{x_0}(x)=\infty$, leading to a contradiction. Hence, $\underline{\rho}=0$. 

\medskip

\paragraph{\bf Step 4} Let $\rho \in (0,\rho_{\OO}]$. Since $V^{x_0}|_{\AAa}=0$ (by Lemma \ref{lem-vanishing-on-M}), there is an open and connected set $\NN$ satisfying $\AAa\subset\NN\subset \OO_{\rho}$. Let $y\in \OO_{\rho}$. Results in {\bf Steps 1-3} ensure the existence of $\al_y\in \Ga_{y,\OO}$ such that $X_{\al_y,y}(0)=y$, $\lim_{t\to\infty}{\rm dist}(X_{\al_y,y}(t),\AAa)=0$, and $t\mapsto V^{x_0}(X_{\al_y,y}(t))$ is decreasing. That is, the continuous path $X_{\al_y,y}$ starts at $y$, lies in $\OO_{\rho}$, and eventually enters $\NN$. Since $y\in \OO_{\rho_{\OO}}$ is arbitrary and $\NN$ is connected, the connectedness of $\OO_{\rho}$ follows. 
\end{proof}

Now, we prove Proposition \ref{lem:lim}.

\begin{proof}[Proof of Proposition \ref{lem:lim}]
We first show that $\Om_*=\bigcup_{\OO\in\OO_{\AAa,\Om}}\OO_{\rho_{\OO}}$. Obviously, $\bigcup_{\OO\in\OO_{\AAa,\Om}}\OO_{\rho_{\OO}}\subset \Om_*$. To show the converse inclusion, we take $y_0\in \Om_*$. Then, $V^{x_0}(y_0)<\rho_*$ and $\text{dist}(\{x\in \Om: V^{x_0}(x)\leq V^{x_0}(y_0)\},\pa\Om)>0$. Clearly, there is $y_1\in \Om_*$ such that $V^{x_0}(y_0)<V^{x_0}(y_1)<\rho_*$. It is not hard to find $\OO\in \OO_{\AAa,\Om}$ such that $\{x\in \Om: V^{x_0}(x)\leq V^{x_0}(y_0)\}\stst \OO$ and  $y_1\in \pa\OO$. Therefore, $V^{x_0}(y_0)<\rho_{\OO}\leq V^{x_0}(y_1)$, leading to $y_0\in \OO_{\rho_{\OO}}$. Since $y_0$ is arbitrary in $\Om_*$, we conclude $\Om_*\subset \bigcup_{\OO\in\OO_{\AAa,\Om}}\OO_{\rho_{\OO}}$. 

Next, we prove the variational representation. Let $x\in \Om_*$ and $\OO\in \OO_{\AAa,\Om}$ be such that $x\in \OO_{\rho_{\OO}}$. Since $\Ga_{x,\OO}\subset \Ga_x$, it follows from Lemma \ref{lem:lim-1} that $V^{x_0}(x)=\min_{\alpha\in\Gamma_{x,\OO}}I_{\al,x}(\infty)\geq \inf_{\alpha\in\Gamma_{x}}I_{\al,x}(\infty)$. Note that the desired conclusion follows immediately if there holds
\begin{equation}\label{2023-06-21-4}
V^{x_0}(x)\leq \inf_{\alpha\in\Gamma_{x}\sm \Ga_{x,\OO}}I_{\al,x}(\infty).    
\end{equation}

We establish \eqref{2023-06-21-4} to finish the proof. Let $\al\in \Gamma_{x}\sm \Ga_{x,\OO}$. This implies in particular that $X_{\al,x}$ leaves $\OO$ before eventually staying in $\OO_{\rho_{\OO}}$. Thus, the last time that $X_{\al,x}(t)$ re-enters $\OO_{\rho_{\OO}}$ is positive and finite, namely, $\tau_{\al,x}:=\sup\left\{t\geq 0: X_{\al,x}(t)\notin \OO_{\rho_{\OO}}\right\}<\infty$.
Clearly, $X_{\al,x}(\tau_{\al,x})\in \pa\OO_{\rho_{\OO}}$ and  $X_{\al,x}(t)\in \OO_{\rho_{\OO}}$ for all $t>\tau_{\al,x}$. 

Let $\tilde{\OO}\in \OO_{\AAa,\Om}$ be such that $\OO\stst \tilde{\OO}$. Set $\tilde{x}:=X_{\al,x}(\tau_{\al,x})$ and  $\tilde{\al}:=\al(\cdot +\tau_{\al,x})$. Obviously, $\tilde{x}\in \pa\OO_{\rho_{\OO}}\stst \tilde{\OO}_{\rho_{\tilde{\OO}}}$. Since
\begin{equation*}
        \frac{d}{ds}X_{\al,x}(s+\tau_{\al,x})=-b(X_{\al,x}(s+\tau_{\al,x}))-\al(s+\tau_{\al,x}),\quad s\geq 0,
\end{equation*}
we see that $X_{\tilde{\al},\tilde{x}}(s)=X_{\al,x}(s+\tau_{\al,x})\in \OO$ for $s\geq 0$, indicating $\tilde{\al}\in \Ga_{\tilde{x},\tilde{\OO}}$. Therefore, an application of Lemma \ref{lem:lim-1} yields $V^{x_0}(\tilde{x})\leq I_{\tilde{\al},\tilde{x}}(\infty)=\frac{1}{4}\int_{\tau_{\al,x}}^{\infty}\alpha^{\top}A^{-1}(X_{\alpha,x})\alpha$. Since $x\in \OO_{\rho_{\OO}}$ and $\tilde{x}\in \pa \OO_{\rho_{\OO}}$, we have $V^{x_0}(x)<\rho_{\OO}=V^{x_0}(\tilde{x})$, leading to $V^{x_0}(x)<I_{\al,x}(\infty)$. As $\al$ is arbitrary in $\Ga_{x}\sm \Ga_{x,\OO}$, \eqref{2023-06-21-4} follows. 

Finally, we establish the connectedness of the sublevel sets. For $\rho\in (0,\rho_*)$, we see from Lemma \ref{lem-vanishing-on-M} and the definition of $\Om_{\rho}$ that $\AAa\stst\Om_{\rho}\stst \Om_*\subset \Om$. Therefore, there exists  $\OO\in \OO_{\AAa,\Om}$ such that $\Om_{\rho}\stst \OO$. Obviously, $\rho<\rho_{\OO}$ and $\Om_{\rho}= \OO_{\rho}$. This together with Lemma \ref{lem:lim-1} ensures $\Om_{\rho}$ is connected. The connectedness of $\Om_*$ then follows from the fact $\Om_*=\lim_{\rho\to \rho_*^{-}}\Om_{\rho}$. This completes the proof.
\end{proof}

\subsection{Proof of Theorem \ref{thm-main-result}}\label{subsec-proof-LDP}

Recall from Proposition \ref{prop-Convergent-subsequence-V^x_0} that $\lim_{n\to\infty}V_{\ep_n}=V^{x_{0}}$ in $C^{\al}(\Om)$ for any $\al\in (0,1)$. Since the variational representation for $V^{x_0}$ in Proposition \ref{lem:lim} is independent of $x_0$, one would expect that the whole family $V_{\ep}$ converges locally uniformly to $V^{x_0}$ as $\ep\to 0$. However, this does not follow in such a straightforward way since the variational representation in Proposition \ref{lem:lim} only holds in $\Om_*$, which depends on $x_0$, and thus, the choice of $\{\ep_n\}_n$.


\begin{lem}\label{two-variational-representation-coincide}
For each $x\in\Om$, $V(x)=\inf_{\alpha\in\Gamma_x}I_{\al,x}(\infty)$.
\end{lem}
\begin{proof}
Fix $x\in \Om$. We define $\Pi_x:\Ga_x\to\Phi_x$ by setting $(\Pi_x\al)(t)=X_{\al,x}(-t)$ for $t\in (-\infty,0]$ and $\al\in \Ga_x$. If $\al\in \Ga_x$, we readily see that $X_{\al,x}(-\cdot)\in \Phi_x$. Hence, $\Pi_x$ is well-defined. 

We show $\Pi_x\Ga_x=\Phi_x$. Let $\phi\in \Phi_x$. Since $\dot{\phi}\in L^2_{loc}((-\infty,0];\R^{d})$ and $b$ is continuous, it follows that $\al(t):=-b(\phi(-t))+\dot{\phi}(-t)$ is well-defined a.e. in $[0,\infty)$ and $\al\in \Ga$. Clearly, $X_{\al,x}=\phi(-\cdot)$ satisfies required properties in the definition $\Ga_x$. Hence, $\al\in \Ga_x$ and $\Pi_x\al=\phi$. 

Consequently,
$$
\inf_{\al\in \Ga_x}I_{\al,x}(\infty)=\inf_{\al\in \Ga_x}I(\Pi_x\al)= \inf_{\phi\in \Phi_x }I(\phi)=V(x),
$$
proving the lemma. 
\end{proof}

\begin{proof}[Proof of Theorem \ref{thm-main-result}]
Recall $V_{\ep}=-\frac{\ep^2}{2}\ln u_{\ep}$ and choose a sequence $\{\ep_n\}_n\subset (0,\infty)$ with $\lim_{n\to \infty}\ep_n=0$. By Propositions \ref{prop-Convergent-subsequence-V^x_0} and \ref{lem:lim}, there exists a subsequence, still denoted by $\{\ep_n\}_n$, and $x_0\in \AAa$ such that 
$$
\lim_{n\to\infty}V_{\ep_n}(x)=V^{x_0}(x)=\min_{\alpha\in\Gamma_x}I_{\al,x}(\infty)\,\,\text{locally uniformly in $x\in \Om_*$},
$$
where $\Om_*=\{x\in \Om: V^{x_0}(x)<\rho_*\}$ and $\rho_*=\liminf_{x\to \pa\Om }V^{x_0}(x)$. It follows from Lemma \ref{two-variational-representation-coincide} and Proposition \ref{lem:lim} that $ V^{x_0}=V$ in $\Om_*$. Moreover, Proposition \ref{prop-Convergent-subsequence-V^x_0} ensures that $\lim_{n\to\infty}V_{\ep_n}=V$ in $C^{\al}(\Om_*)$ for any $\al\in (0,1)$.

Next, we prove 
\begin{equation}\label{claim-july-09}
V\geq \rho_*\quad\text{in}\quad \Om\sm \Om_*.   
\end{equation}
Then, $\rho^V=\liminf_{x\to \pa \Om}V(x)=\rho_*$ and $\Om_*=\Om^V$. Since $\{\ep_n\}_n$ is arbitrary, we arrive at $\lim_{\ep\to 0} V_{\ep}=V$ in $C^{\al}(\Om^V)$ for any $\al\in (0,1)$. 

Suppose \eqref{claim-july-09} fails. Namely, there exists $x_1\in \Om\sm \Om_*$ such that $\rho_1:=V(x_1)\in (0,\rho_*)$. It follows from the standard dynamic programming arguments that 
$V(x_1)=\inf_{\alpha\in\Gamma_x}\left[I_{\al,x_1}(t)+V(X_{\alpha,x_1}(t))\right]$ for all $t\geq 0$. Let $\rho_2\in (\rho_1,\rho_*)$ and $\al\in \Ga_x$. Since $\lim_{t\to\infty}{\rm dist}(X_{\al,x}(t),\AAa)=0$ and $\AAa\stst \Om_{\rho_2}$, there is $t_0=t_0(\al)>0$ such that $X_{\al,x}(t_0)\in \pa\Om_{\rho_2}$, and thus, $V(X_{\al,x}(t_0))=\rho_2$. As a result, $I_{\al,x_{1}}(t_0)+V(X_{\alpha,x_1}(t_0))\geq \rho_2$. Since $\al$ is arbitrary, it follows that $V(x_1)\geq \rho_2$. This contradicts the fact that $V(x_1)=\rho_1<\rho_2$, and thus, proves \eqref{claim-july-09}.  
\end{proof}

The following corollary is needed later.

\begin{cor}\label{cor-local-variational-representation}
Let $\OO\in\OO_{\AAa,\Om}$. For $x\in\OO$, set 
$$
\Phi_{x,\OO}:=\left\{\phi\in AC((-\infty,0];\OO): \dot{\phi}\in L^2_{loc}((-\infty,0];\R^{d}),\,\,\phi(0)=x,\,\,\lim_{t\to -\infty}\dist(\phi(t),\AAa)=0\right\}.
$$ 
For each $x\in \OO^{V}:=\{x\in \OO: V(x)<\rho_{V,\OO}\}$, where $\rho_{V,\OO}:=\liminf_{x\to \pa\OO} V(x)$, there holds $V(x)=\min_{\phi\in \Phi_{x,\OO}}I(\phi)$.
\end{cor}
\begin{proof}
It is shown in the proof of Theorem \ref{thm-main-result} that  $V=V^{x_0}$ in $\Om^V$. Recall $\Ga_{x,\OO}$ from Lemma \ref{lem:lim-1} and the map $\Pi_x$ from the proof of Lemma \ref{two-variational-representation-coincide}. It is easy to verify that $\Pi_x\Ga_{x,\OO}=\Phi_{x,\OO}$, leading to $\inf_{\al\in \Ga_{x,\OO}} I_{\al,x}(\infty)=\inf_{\phi\in \Phi_{x,\OO}} I(\phi)$. The desired result then follows from Lemma \ref{lem:lim-1}.
\end{proof}


\subsection{Minimizer and Hamiltonian system}

In this section, we provide some intuition for investigating the regularity of $V$. Let $L:\Om\times\R^{d}\to\R$ be the Lagrangian associated with the Hamiltonian $H$, that is,
\begin{equation}\label{Lagrangian}
L(x,v)=\sup_{p\in\R^{d}}\left\{v\cdot p-H(x,p)\right\}=\frac{1}{4}\left(b(x)-v\right)^{\top}A^{-1}(x)\left(b(x)-v\right).    
\end{equation}
The regularity and uniform convexity of $H$ ensure the following classical result.
\begin{lem}\label{Hamiltonian-Lagrangian-eqivalence}
    For each $x\in\Om$, the following statements are equivalent.
    \begin{itemize}
        \item $p=\partial_v L(x,v)$.
        \item $v=\partial_p H(x,p)$.
        \item $L(x,v)+H(x,p)=v\cdot p$.
    \end{itemize}
Whenever the above statements are true, there holds $\partial_x L(x,v)=-\partial_x H(x,p)$.
\end{lem}

In terms of $L$, the quasi-potential function $V$ defined in \eqref{def-quasi-potential-V-FW-sense} can be written as
\begin{equation}\label{def-V-lagrange}
\begin{split}
    V(x)=\inf_{\phi\in\Phi_{x}}\int_{-\infty}^0 L(\phi, \dot{\phi}),\quad x\in\Om.
\end{split}
\end{equation}

\begin{defn}[Minimizer]\label{defn-minimizer-V}
Let $x\in \Om$. A function $\phi\in\Phi_{x}$ is called a \emph{minimizer corresponding to $V(x)$} if 
    $V(x)=\int_{-\infty}^0 L(\phi, \dot{\phi})$.
\end{defn}

Lemma \ref{lem-quasi-potential-basic-results} indicates that for each $x\in\Om^{V}$, there exist minimizers corresponding to $V(x)$. The proof of the following result is standard, and thus, omitted (see e.g. \cite[Theorem 1]{DD85}).



\begin{lem}\label{thm-extremal-orbit}
Let $x_0\in\Om$. Suppose $x(t)$ is a minimizer corresponding to $V(x_0)$ and set $p(t):=\partial_vL(x(t), \dot{x}(t))$. Then, $(x(t), p(t))$  satisfies the Hamiltonian system \eqref{hamiltonian-system-intro}.
Moreover, $H(x(t),p(t))=0$ for all $t\leq0$ and $\lim_{t\to -\infty}{\rm dist}\left((x(t),p(t)),\AAa\times\{0\}\right)=0$.
\end{lem}

\begin{rem}\label{rem-general}
    We comment on how Lemma \ref{thm-extremal-orbit} guides us to establish the regularity of the quasi-potential $V$ in a neighborhood of the maximal attractor $\AAa$. 

    Given the local uniqueness result Lemma \ref{lem-unique-minimizer}, it suffices to construct an open neighborhood $\OO$ of $\AAa$ and a non-negative function $W\in C^{k'+1}(\OO)$ that vanishes only on $\AAa$ and satisfies $H(x,\nabla W(x))=0$ for $x\in\OO$. 
    
    Lemma \ref{thm-extremal-orbit} suggests studying the unstable set of $\AAa\times\{0\}$, namely,
    $$
    \left\{(x,p)\in\Om\times\R^{d}:\bigcup_{t\leq0}\Phi^{t}_{H}(x,p)\subset\Om\times\R^{d}\,\,\text{and}\,\,\lim_{t\to-\infty}{\rm dist}\left(\Phi_{H}^{t}(x,p),\AAa\times\{0\}\right)=0\right\}.
    $$
    Inspired by the work in the case that $\AAa$ is a linearly stable equilibrium \cite{DD85} so that $\AAa\times\{0\}$ is a hyperbolic fixed point and those trajectories $(x(t),p(t))$ naturally lives on its unstable manifold, we wish to show that this unstable set is actually an $d$-dimensional $C^{k'}$ manifold in a neighborhood of $\AAa\times\{0\}$ and is given by the graph of the gradient field of some $C^{k'+1}$ function $W$, as expected. This often requires $\AAa\times\{0\}$ to be an invariant manifold of $\Phi_{H}^{t}$ (equivalently, $\AAa$ is an invariant manifold of $\vp^{t}$) that is normally  hyperbolic \cite{Fen71,HPS77}. 
    
    Unfortunately, this can not be the case unless $\AAa$ is a linearly stable equilibrium. In fact, if $\AAa$ is a normally contracting invariant manifold as in {\bf(H4)}, then $\AAa\times\{0\}$ admits unstable, center, and stable directions, and therefore, the unstable set is generally determined by both the unstable manifold and the center manifold associated with the unstable direction and the stable direction, respectively. However, it is a notorious fact that the dynamics on the center manifold can hardly be studied, and therefore, it is unlike to argue that $(x(t),p(t))$ lives on the unstable manifold at this point. The reader is referred to Remark \ref{rem-key-ideas} below for more details as well as our idea of addressing this issue. 
\end{rem}


\section{\bf Hamiltonian dynamics and regular solutions of HJE}\label{sec-regular-sol-HJE}

Throughout this section, we assume that $\vp^{t}$ admits a normally contracting invariant manifold $\MM$ (see Definition \ref{defn-linearly-stable-manifold}) and intend to construct a regular solution of the HJE \eqref{eqn-HJE} in a neighborhood of $\MM$ that is non-negative and vanishes only on $\MM$. The motivation for seeking such a solution of \eqref{eqn-HJE} is explained in Remark \ref{rem-general}, where we apply it to the maximal attractor $\AAa$ assumed to be a normally contracting invariant manifold. Our approach builds on analyzing the Hamiltonian system \eqref{hamiltonian-system-intro}. In Subsection \ref{subsec-Hamiltonian-system}, we construct the local unstable manifold of the invariant manifold $\MM\times\{0\}$ of \eqref{hamiltonian-system-intro}. A desired regular solution of \eqref{eqn-HJE} is then constructed in Subsection \ref{subsec-regular-sol-HJE}.

\begin{rem}
Note that we work with a general normally contracting invariant manifold $\MM$ of $\vp^{t}$, which is of course the maximal attractor of $\vp^{t}$ in its basin of attraction or in a positively invariant open set contained in its basin of attraction, but \emph{not} assumed to be an equivalence class as in {\bf(H3)} for $\AAa$. 
\end{rem}

For simplicity, set $x\cdot t:=\vp^t(x)$ for $x\in \Om$ and $t$ belonging to the maximal interval of existence. Of course, if $x\in\MM$, then the maximal interval of existence is $\R$. Consider the linearization of ODE \eqref{main-ode} along the orbit $x\cdot t$: 
\begin{equation}\label{ode-linearized}
    \dot y= \nabla b(x\cdot t) y,\quad y\in \R^d,
\end{equation}
and denote by $\Psi(x,t):=D\varphi^{t}(x)$ its principal fundamental matrix solution at initial time $0$ \cite{Chicone2006}, that is, $\Psi(x,t)$ the fundamental matrix solution of \eqref{ode-linearized} with $\Psi(x,0)=I_{d\times d}$.

\begin{defn}[Normally contracting invariant manifold]\label{defn-linearly-stable-manifold}
A $\vp^{t}$-invariant $C^{k}$ compact and connected submanifold $\MM$ with dimension $d_{\MM}<d$ is said to be a \emph{normally contracting invariant manifold} of $\vp^{t}$ if there are constants $0<\be_0<\be_*$ satisfying $\frac{\be_0}{\be_*-\be_0}<\frac{1}{k'}$ for some $1\leq k'\leq k-1$ such that for each $x\in\MM$,
\begin{itemize}
    \item $T_x\R^d=S(x)\oplus T_x\MM$,
\item $\Psi(x,t)S(x)=S(x\cdot t)$ for $t\in \R$, 

\item $\|\Psi(x,t)|_{S(x)}\|\lesssim e^{-\be_* t}$, $t\geq 0$, 
\item $\|\Psi(x,t)|_{T_x\MM}\|\lesssim e^{\be_0 |t|}$, $t\in \R$.
\end{itemize}
\end{defn}

Here and after, we use the notation $\lesssim$ (resp. $\gtrsim$) to mean that an inequality $\leq$ (resp.  $\geq$) holds up to a multiplicative constant that is independent of $x\in\MM$ and other parameters involved (e.g., $t\geq0$ or $t\in\R$ as above, or elements belonging to a subspace of $\R^{d}$). While this is not absolutely clarified, no confusion shall be caused. The notation $\approx$ means both $\lesssim$ and $\gtrsim$. 

In Definition \ref{defn-linearly-stable-manifold}, we only assume $k'\geq1$, which is sufficient for constructing the local unstable manifold of $\MM\times\{0\}$. For constructing a regular solution of the HJE \eqref{eqn-HJE}, we need $k'\geq2$ as required in {\bf(H4)}.

It is useful to point out that $\Psi$ is a cocycle on $\MM$:
\begin{equation}\label{cocycle}
\Psi(x,t+s)=\Psi(x\cdot s,t)\Psi(x,s),\quad\forall t,s\in\R,\,\, x\in\MM.    
\end{equation}
Taking the transpose and then the inverse results in
\begin{equation}\label{cocycle-transpose-inverse}
\left(\Psi^{\top}\right)^{-1}(x,t+s)=\left(\Psi^{\top}\right)^{-1}(x\cdot s,t)\left(\Psi^{\top}\right)^{-1}(x,s),\quad\forall t,s\in\R,\,\, x\in\MM,  
\end{equation}
that is, $\left(\Psi^{\top}\right)^{-1}$ is also a cocycle on $\MM$. In fact, $\left(\Psi^{\top}\right)^{-1}(x,t)$ is the principal fundamental matrix solution at initial time $0$ of $\dot{y}=-\nabla b^{\top}(x\cdot t)y$ (see \eqref{dual-equation-fundamental-solution-matrix} below).


\subsection{Local unstable manifold}\label{subsec-Hamiltonian-system}

Recall that $\Phi_{H}^{t}$ is the local flow generated by the Hamiltonian system \eqref{hamiltonian-system-intro}. For $x\in\MM$, the linearization of \eqref{hamiltonian-system-intro} along the orbit $\Phi^{t}_{H}(x,0)=(x\cdot t, 0)$ reads
\begin{equation}\label{H-ODE-linearized}
    \begin{cases}
    \dot y=\nabla b(x\cdot t)y+2 A(x\cdot t) q,\\
    \dot q=-\nabla b^{\top}(x\cdot t)q
\end{cases}
\quad\text{in}\quad\R^{d}\times\R^{d}.
\end{equation}
Denote by $\Psi_H(x,t)$ the principal fundamental matrix solution at initial time $0$ of \eqref{H-ODE-linearized}. 

\begin{rem}\label{rem-key-ideas}
We briefly comment on the invariant splitting of \eqref{H-ODE-linearized}, the associated local invariant manifolds of $\MM\times\{0\}$, and our approach including both ideas and techniques. In particular, this explains why we only look at the local unstable manifold of $\MM\times\{0\}$.
\begin{itemize}
    \item[(1)] From \eqref{H-ODE-linearized}, we see that the invariant manifold $\MM\times\{0\}$ is NOT normally hyperbolic but partially hyperbolic, which is in sharp contrast to the situation considered in \cite{DD85}, where $\MM$ is a linearly stable equilibrium of \eqref{main-ode} so that $\MM\times\{0\}$ is a hyperbolic fixed point of \eqref{hamiltonian-system-intro}. 

Indeed, in the space $\{q=0\}$, the system \eqref{H-ODE-linearized} is reduced to $\dot y=\nabla b(x\cdot t)y$, which determines the $d_{\MM}$-dimensional tangent direction and $(d-d_{\MM})$-dimensional stable direction to $\MM\times\{0\}$. Note that the $q$-equation, which is decoupled from the $y$-equation and dual to \eqref{ode-linearized},  has $d_{\MM}$-dimensional central dynamics and $(d-d_{\MM})$-dimensional unstable dynamics. They together with the $y$-equation determine the $d_{\MM}$-dimensional central direction and $(d-d_{\MM})$-dimensional unstable direction to $\MM\times\{0\}$. 

This gives rise to the invariant splitting under $\Psi_H(x,t)$: for $x\in\MM$,
\begin{equation*}
    \begin{split}
T_{(x,0)}(\R^{d}\times\R^{d})&=E^{u}_{(x,0)}\oplus E^{c}_{(x,0)} \oplus E^{s}_{(x,0)},\\
\Psi_H(x,t)E^{i}_{(x,0)}&=E^{i}_{(x\cdot,0)},\quad i=u,c,s,
    \end{split}
\end{equation*}
where $T_{(x,0)}(\MM\times\{0\})\subsetneqq E^{c}_{(x,0)}$. Moreover, $\Psi_H(x,t)$ expands $E^{u}_{(x,0)}$ sharply than $E^{c}_{(x,0)}$, and $\Psi_H(x,t)$ contracts $E^{s}_{(x,0)}$ sharply than $E^{c}_{(x,0)}$.

\item[(2)] Given the invariant splitting in (1), the classical invariant manifold theory, recalled in Theorem \ref{app-thm-Fenichel} for readers' convenience, ensures the existence of a unique $d$-dimensional local $C^{k'}$ unstable manifold $\WW^{u}$ and a unique $d$-dimensional local $C^{k'}$ stable manifold $\WW^{s}$ of $\MM\times\{0\}$. Observe that $\WW^{s}$ extends to $\Om\times\{0\}$. Applying the center manifold theory developed in \cite{CLY00}, we can find an $2d_{\MM}$-dimentional $C^{k'}$ center manifold $\WW^{c}$ of $\MM\times\{0\}$, which is not unique in general.

\item[(3)] For our purposes, it is of crucial importance to study the \emph{unstable set} of $\MM\times\{0\}$, as suggested by Lemma \ref{thm-extremal-orbit} and commented in Remark \ref{rem-general}. The unstable set is determined by the unstable manifold $\WW^{u}$ and the center manifold $\WW^{c}$ if it contains trajectories approaching $\MM\times\{0\}$ as $t\to-\infty$. However, as the asymptotic dynamics on $\WW^{c}$ is unknown and can hardly be investigated due to complexity, bifurcation, etc., the structure of the unstable set of $\MM\times\{0\}$ can barely be identified. But, this is not so necessary anyway since an $d$-dimensional structure suffices to determine what we want, while the unstable set may have dimension higher than $d$. 

Given that the unstable manifold $\WW^{u}$ is $d$-dimensional, we directly work with it and manage to show that it is given by the graph of the gradient field of some non-negative $C^{k'+1}$ function that vanishes only on $\MM$, as expected. 
\end{itemize}
\end{rem}

Recall that $\Psi(x,t)$ is the principal fundamental matrix solution at initial time $0$ of \eqref{ode-linearized}. It is well-known that $\left(\Psi^{\top}\right)^{-1}(x,t)$ is the principal fundamental matrix solution at initial time $0$ of the $q$-equation in \eqref{H-ODE-linearized}, which is the dual equation of \eqref{ode-linearized}. In fact, by the derivative of the inverse matrix formula, 
\begin{equation}\label{dual-equation-fundamental-solution-matrix}
\begin{split}
    \frac{d}{dt}\left(\Psi^{\top}\right)^{-1}(x,t)&=-\left(\Psi^{\top}\right)^{-1}(x,t)\frac{d}{dt}\Psi^{\top}(x,t)\left(\Psi^{\top}\right)^{-1}(x,t)\\
    &=-\left(\Psi^{\top}\right)^{-1}(x,t)\Psi^{\top}(x,t)\nabla b^{\top}(x\cdot t)\left(\Psi^{\top}\right)^{-1}(x,t)\\
    &=-\nabla b^{\top}(x\cdot t)\left(\Psi^{\top}\right)^{-1}(x,t).
\end{split}
\end{equation}
The variation of constants formula asserts that the unique solution $(y(t),q(t)):=\Psi_H(x,t)(y_0,q_0)$ of \eqref{H-ODE-linearized} with initial condition $(y_0,q_0)$ at $t=0$ satisfies 
\begin{equation}\label{soln-H-ODE-linearized}
\begin{cases}
    y(t)=\Psi(x,t)y_0+2\displaystyle\int_0^t \Psi(x\cdot s,t-s)A(x\cdot s) q(s) ds,\\
    q(t)=\left(\Psi^{\top}\right)^{-1}(x,t) q_0.
\end{cases}
\end{equation}

Recall the splitting $T_x\R^d=S(x)\oplus T_x\MM$ from Definition \ref{defn-linearly-stable-manifold}. Let $P(x)$ be the projection onto $S(x)$ along $T_{x}\MM$. Then, $I_{d\times d}-P(x)$ is the projection onto $T_{x}\MM$ along $S(x)$. Since $P(x\cdot t)\Psi(x,t)=\Psi(x,t)P(x)$ for all $t\in\R$ and $x\in\MM$, we deduce
\begin{equation}\label{invariance-transpose}
    \begin{split}
        P^{\top}(x\cdot t)\left(\Psi^{\top}\right)^{-1}(x,t)=\left(\Psi^{\top}\right)^{-1}(x,t)P^{\top}(x),\quad \forall t\in\R,\,\,x\in\MM.
    \end{split}
\end{equation}
Set $E(x):=\range P^{\top}(x)$ and $N(x):=\range\left(I_{d\times d}-P^{\top}(x)\right)$ for $x\in\MM$. The following lemma concerning the invariant splitting under $\left(\Psi^{\top}\right)^{-1}$ follows. 


\begin{lem}\label{lem-invariant-splitting-q}
For each $x\in\MM$, the following hold.
\begin{itemize}
    \item [(1)] $\R^d=E(x) \oplus N(x)$.
    \item [(2)] $\left(\Psi^{\top}\right)^{-1}(x,t)E(x)=E(x\cdot t)$ and $\left(\Psi^{\top}\right)^{-1}(x,t)N(x)=N(x\cdot t)$ for $t\in \R$.
    \item [(3)] There hold
\begin{subequations}\label{linear-HODE-splitting-q-direction}
\begin{align}
    \left\|\left(\Psi^{\top}\right)^{-1}(x,t)\big|_{E(x)}\right\|&\lesssim e^{\be_* t},\,\,\,\quad  t\leq 0,\label{eqn-2023-05-29-3}\\
    \left\|\left(\Psi^{\top}\right)^{-1}(x,t)\big|_{N(x)}\right\|&\lesssim e^{\be_0 |t|},\quad  t\in\R.\label{eqn-2023-05-29-4}
\end{align}
\end{subequations}
\end{itemize}
\end{lem}
\begin{proof}
    The first two conclusions follow immediately from \eqref{invariance-transpose}. Note from \eqref{cocycle-transpose-inverse} that
    $$
    \left(\Psi^{\top}\right)^{-1}(x,t)P^{\top}(x)=\Psi^{\top}(x\cdot t,-t)P^{\top}(x)=\left[P(x)\Psi(x\cdot t,-t)\right]^{\top},
    $$
    which together with Definition \ref{defn-linearly-stable-manifold} implies \eqref{eqn-2023-05-29-3}. Similar arguments lead to \eqref{eqn-2023-05-29-4}.
\end{proof}

The next result is crucial for determining in particular the unstable direction to $\MM\times\{0\}$.

\begin{lem}\label{lem-Q(x)}
    For each $x\in\MM$, the following hold.
    \begin{itemize}
    \item[(1)] The matrix
\begin{equation*}
    Q(x):=2\int_{-\infty}^0 \Psi(x\cdot s,-s) A(x\cdot s) \left(\Psi^{\top}\right)^{-1}(x,s)P^{\top}(x)  ds
\end{equation*}
is well-defined, and satisfies $\|Q(x)\|\lesssim1$.
        
    \item[(2)] $Q(x)$ is symmetric on $E(x)$, that is, $q_{1}^{\top}Q(x)q_{2}=q_{1}^{\top}Q^{\top}(x)q_{2}$ for all $q_1,q_2\in E(x)$.

\item[(3)] $Q(x)$ is positive-definite on $E(x)$, that is, $q^{\top}Q(x)q\gtrsim|q|^{2}$ for all $q\in E(x)$. In particular, $|Q(x)q|\approx |q|$ for all $q\in E(x)$.     
\end{itemize}
\end{lem}

\begin{rem}
    The definition of $Q(x)$ in Lemma \ref{lem-Q(x)} is inspired by the Lyapunov-Perron method. The main idea is to identify a subspace for $(q_0,y_0)\in\R^{d}\times\R^{d}$ such that $(y(t),q(t))=\Psi_H(x,t)(y_0,q_0)$ converges to $(0,0)$ sufficiently fast as $t\to-\infty$. By Lemma \ref{lem-invariant-splitting-q}, there must hold $q_0\in E(x)$. It follows from \eqref{soln-H-ODE-linearized} that $y(t)$ satisfies  
    $$
    y(t)=\Psi(x,t)y_0+2\int_0^t \Psi(x\cdot s,t-s)A(x\cdot s)\left(\Psi^{\top}\right)^{-1}(x,s) q_0 ds,
    $$
    or
    $$
    y_0=\Psi^{-1}(x,t)y(t)+2\int^0_t \Psi(x\cdot s,-s)A(x\cdot s)\left(\Psi^{\top}\right)^{-1}(x,s) q_0 ds.
    $$
Formally, letting $t\to-\infty$ in the above equality yields $y_0=2\int^0_{-\infty} \Psi(x\cdot s,-s)A(x\cdot s)\left(\Psi^{\top}\right)^{-1}(x,s) q_0 ds$, that is, $y_0=Q(x)q_0$. 
\end{rem}

\begin{proof}[Proof of Lemma \ref{lem-Q(x)}]
(1) Definition \ref{defn-linearly-stable-manifold} implies $\|\Psi(x,t)\|\lesssim e^{\be_0 t}$ for all $t\geq 0$. This together with \eqref{eqn-2023-05-29-3} yields  that $|Q(x)q|\lesssim 2|P^{\top}(x)q|\max_{\MM}\|A\|\int_{-\infty}^0 e^{-\be_0 s}e^{\be_* s}ds\lesssim|P^{\top}(x)q|$. The results follow.

\medskip

\noindent (2) Let $q_1, q_2\in E(x)$. Note that
\begin{equation*}
    q_1^{\top}Q^{\top}(x)q_2=2\int_{-\infty}^0 q_1^{\top}\Psi^{-1}(x,s)A^{\top}(x\cdot s)\Psi^{\top}(x\cdot s,-s)q_2ds.
\end{equation*}
The equality $\Psi^{-1}(x,s)=\Psi(x\cdot s,-s)$ (by \eqref{cocycle}) and the symmetry of $A$ then give rise to $q_1^{\top}Q^{\top}(x)q_2=q_1^{\top}Q(x)q_2$. 

\medskip

\noindent (3)  Since $\Psi(x,t)$ is the principal fundamental matrix solution at initial time $0$ of \eqref{ode-linearized} and $\max_{\MM}\|\nabla b\|<\infty$, there exists $\be^*>0$ such that $\|\Psi^{\top}(x,t)\|\lesssim e^{-\be^*t}$ for all $t\leq0$, implying that $\left|\left(\Psi^{\top}\right)^{-1}(x,t)y\right|\gtrsim e^{\be^* t}|y|$ for all $y\in\R^{d}$ and $t\leq0$. This together with the equality $\Psi^{\top}(x\cdot s,-s)=\left(\Psi^{\top}\right)^{-1}(x,s)$ (by \eqref{cocycle-transpose-inverse}) yields
\begin{equation*}
\begin{split}
    q^{\top}Q(x)q&=2\int_{-\infty}^0\left[\left(\Psi^{\top}\right)^{-1}(x,s)q\right]^{\top} A(x\cdot s)\left(\Psi^{\top}\right)^{-1}(x,s)qds\\
    &\geq 2\la \int_{-\infty}^0\left|\left(\Psi^{\top}\right)^{-1}(x,s)q\right|^2ds\gtrsim \int_{-\infty}^0 e^{2\be^{*}s}|q|^2ds\gtrsim|q|^2,\quad \forall q\in E(x),
\end{split}
\end{equation*}
where $\la>0$ is such that $y^{\top}A(x)y\geq \la |y|^2$ for all $y\in \R^d$ and $x\in \MM$. The ``In particular" part then follows readily from (1) and $|q^{\top}||Q(x)q|\geq q^{\top}Q(x)q$ for all $q\in E(x)$. 
\end{proof}

Now, we can determine the unstable direction to $\MM\times\{0\}$ and establish the invariant splitting under $\Psi_{H}$. For $x\in\MM$, set 
\begin{equation}\label{def-E(x,0)-N(x,0)}
\begin{split}
\EE(x,0):&=\text{graph}\,Q(x)^{-1}=\left\{(Q(x)q,q):q\in E(x)\right\},\\
\NN(x,0):&=T_{x}\R^{d}\times N(x),    
\end{split}
\end{equation}
where $Q(x)^{-1}:{\range}\, Q(x)\to E(x)$ denotes the inverse of $Q(x)$. It is well-defined thanks to Lemma \ref{lem-Q(x)} (3), implying $\ker Q(x)=N(x)$.

\begin{lem}\label{lem-linear-Hsystem}
For each $x\in\MM$, the following hold.
\begin{enumerate}
\item   $T_{(x,0)}(\R^{d}\times\R^{d})=\EE(x,0)\oplus \NN(x,0)$.

    \item $T_{(x,0)}(\MM\times \{0\})\subset \NN(x,0)$. 
    
    \item  $\Psi_H(x,t)\EE(x,0)=\EE(x\cdot t,0)$ and $\Psi_H(x,t)\NN(x,0)=\NN(x\cdot t,0)$ for $t\in \R$.
    
    \item For any fixed $\be'_0\in (\be_0,\be_*)$ satisfying $\frac{\be'_0}{\be_*-\be'_0}<\frac{1}{k'}$,
\begin{subequations}\label{eqn-2023-05-29-111}
\begin{align}
    \left\|\Psi_H(x,t)|_{\EE(x,0)}\right\|&\lesssim e^{\be_* t},\,\,\,\quad t\leq 0,\label{eqn-2023-05-29-1}\\
    \|\Psi_H(x,t)|_{\NN(x,0)}\|&\lesssim e^{\be'_0 |t|},\quad t\in \R. \label{eqn-2023-05-29-2}
\end{align}
\end{subequations}

\end{enumerate}
\end{lem}

\begin{rem}
    It should be clear from the statement of Lemma \ref{lem-linear-Hsystem} (4) that both inequalities $\lesssim$ depends on $\beta_{0}'$, but we choose not to highlight this dependence as it is only used with a fixed $\beta_{0}'$.
\end{rem}

\begin{proof}[Proof of Lemma \ref{lem-linear-Hsystem}] (1) and (2) follow readily from the definition of $\EE(x,0)$ and $\NN(x,0)$. 

\medskip

\noindent (3) The invariance of $\NN$ under $\Psi_H$ is a result of \eqref{soln-H-ODE-linearized} and the invariance of $N$ under $\left(\Psi^{\top}\right)^{-1}$ (by Lemma \ref{lem-invariant-splitting-q} (2)). We show the invariance of $\EE$ under $\Psi_H$. This is equivalent to showing that if $q_0\in E(x)$ and $y_0=Q(x)q_0$, then $(y(t),q(t)):= \Psi_{H}(x,t)(y_0,q_0)$ satisfies
\begin{equation}\label{soln-Q-H-ODE-linearized}
    y(t)=Q(x\cdot t)q(t).
\end{equation}

We prove \eqref{soln-Q-H-ODE-linearized}. It follows from \eqref{soln-H-ODE-linearized} that 
\begin{equation*}
\begin{split}
    y(t)&=\Psi(x,t)Q(x)q_0+2\int_0^{t}\Psi(x\cdot s,t-s) A(x\cdot s) \left(\Psi^{\top}\right)^{-1}(x,s) q_0 ds\\
    &=\Psi(x,t)\left(Q(x)q_0+2\int_0^{t}\Psi(x\cdot s,-s) A(x\cdot s) \left(\Psi^{\top}\right)^{-1}(x,s) q_0 ds\right)\\
    &=2\Psi(x,t)\int_{-\infty}^t \Psi(x\cdot s,-s) A(x\cdot s) \left(\Psi^{\top}\right)^{-1}(x,s) q_0 ds,
\end{split}
\end{equation*}
where we used \eqref{cocycle} in the second equality and the definition of $Q$ in the last equality. 

Clearly, $q(t)=\left(\Psi^{\top}\right)^{-1}(x,t)q_0\in E(x\cdot t)$ thanks to Lemma \ref{lem-invariant-splitting-q} (2). Note that
\begin{equation*}
\begin{split}
    Q(x\cdot t)q(t)&=2\int_{-\infty}^0 \Psi(x\cdot t\cdot s,-s) A(x\cdot t\cdot s) \left(\Psi^{\top}\right)^{-1}(x\cdot t,s) \left(\Psi^{\top}\right)^{-1}(x,t) q_0 ds\\
    &=2\int_{-\infty}^0 \Psi(x,t)\Psi(x\cdot (t+s),-(t+s)) A(x\cdot (t+s))\left(\Psi^{\top}\right)^{-1}(x,t+s)q_0 ds\\
        &=2\Psi(x,t)\int_{-\infty}^t \Psi(x\cdot s,-s) A(x\cdot s) \left(\Psi^{\top}\right)^{-1}(x,s) q_0 ds,
\end{split}
\end{equation*}
where we used \eqref{cocycle-transpose-inverse} and \eqref{cocycle} in the second equality. Hence, \eqref{soln-Q-H-ODE-linearized} holds.

\medskip

\noindent (4) Note that the invariance of $\EE$ under $\Psi_{H}$ and Lemma \ref{lem-Q(x)} (3) ensure that $|(y(t),q(t))|\approx |q(t)|$ for all $t\in\R$ and $(y_0,q_0)\in \EE(x,0)$. Then, \eqref{eqn-2023-05-29-1} follows immediately from \eqref{eqn-2023-05-29-3}.

To show \eqref{eqn-2023-05-29-2}, we let $(y_0,q_0)\in \NN(x,0)$. Applying Definition \ref{defn-linearly-stable-manifold} and \eqref{eqn-2023-05-29-4}, we find $|q(t)|\lesssim e^{\be_0 |t|}|q_0|$ for all $t\in\R$ and
\begin{equation*}
\begin{split}
    |y(t)|&\leq |\Psi(x,t) y_0|+2\left|\int_0^t\Psi(x\cdot s, t-s)A(x\cdot s) q(s)ds\right|\\
    &\lesssim e^{\be_0 |t|}|y_0|+2|q_0|\max_{\MM}\|A\|\left|\int_0^t e^{\be_0|t-s|} e^{\be_0 |s|}ds\right|\lesssim e^{\be_0 |t|}|y_0|+|t|e^{\be_0 |t|}|q_0|,\quad t\in \R.
\end{split}
\end{equation*}
For fixed $\be'_0\in (\be_0,\be_*)$ satisfying $\frac{\be'_0}{\be_*-\be'_0}<\frac{1}{k'}$, we then deduce 
\begin{equation*}
    \left|\Psi_H(x,t)(y_0,q_0)\right|\leq |y(t)|+|q(t)|\lesssim  e^{\be'_0 |t|}|(y_0,q_0)|,\quad t\in \R\andd (y_0,q_0)\in \NN(x,0),
\end{equation*}
giving rise to \eqref{eqn-2023-05-29-2}. 
\end{proof}

Finally, we state and prove the local unstable manifold theorem. Set $(X_{x,p}(t),P_{x,p}(t)):=\Phi_H^t(x,p)$.

\begin{thm}\label{lem-main-Hsystem}
The invariant manifold $\MM\times\{0\}$ admits a $C^{k'}$ local unstable manifold $\WW^{u}$, which is a Lagrangian submanifold and satisfies the following properties.
\begin{enumerate}
    \item [(i)] $\MM\times \{0\}\subset \WW^{u}$ and $T_{(x,0)}\WW^{u}=\EE(x,0)\oplus T_{(x,0)}(\MM\times \{0\})$ for all $x\in\MM$.

    \item [(ii)] $\WW^{u}=\cup_{x\in\MM}\WW_{x}^{u}$, where
    \begin{itemize}
        \item $\WW_x^{u}$ is a $C^{k'}$ manifold and $(x,0)\in\WW_{x}^{u}$,
        
        \item $\WW_x^{u}$ is diffeomorphic to a ball in $\EE(x,0)$ and tangent to $\EE(x,0)$ at $(x,0)$,

        \item $\WW_x^{u}\cap \WW_{x'}^{u}=\emptyset$ if $x\neq x'$,
        
        \item $\Phi_H^t \WW_x^{u}\subset \WW_{x\cdot t}^{u}$ for $t\leq 0$.
    \end{itemize}
    
    
    
    \item [(iii)] For any $0<\be<\be_*$, there is $C>0$ such that 
    \begin{equation*}
    \text{\rm dist}\left(\Phi_H^t(x,p),(\pi_{\WW^{u}}(x,p)\cdot t, 0)\right)\leq C e^{\be t},\quad \forall(x,p)\in \WW^{u}\andd t\leq 0,
    \end{equation*}
where $\pi_{\WW^{u}}:\WW^{u}\to \MM\times \{0\}$ is the $C^{k'}$ projection defined by $\pi_{\WW^{u}}(x,p)=\tilde{x}$ if $(x,p)\in\WW_{\tilde{x}}$.
    
\end{enumerate}

Moreover, there exist an open set $\MM\subset\OO_0\stst\Om$ and a function $F\in C^{k'}(\OO_0;\R^{d})$ such that the following hold.
\begin{enumerate}
    \item $\graph F=\WW^{u}\cap \left(\OO_0\times\R^{d}\right)$ and $F= 0$ on $\MM$.

    \item For $x\in \OO_0$, let $\pi(x)\in\MM$ be such that $\pi_{\WW^{u}}(x,F(x))=(\pi(x),0)$. Then, $\pi$ is $C^{k'}$ and for any $\be\in (0,\be_*)$,
    $$
    \lim_{t\to -\infty}  \sup_{x\in \OO_0}e^{\be |t|}\left(\left| X_{x,F(x)}(t)-\pi(x)\cdot t\right|+|P_{x,F(x)}(t)|\right)=0.
    $$  

    \item There is an open set $\MM\subset \OO_1\stst \OO_0$ such that $\Phi_H^t\left(\graph F|_{\OO_{1}}\right)\subset \graph F$ for all $t\leq 0$.

    \item $H(x,F(x))=0$ for all $x\in \OO_0$.

    \item For each $x\in \MM$, $T_x\R^d=T_x\MM\oplus\range Q(x)$ and the Jacobian $\nabla F(x)$ satisfies
    \begin{equation*}
    \nabla F(x) y=
        \begin{cases}
        0,&\quad y\in T_x\MM,\\
        Q(x)^{-1}y,& \quad  y\in\range Q(x).
        \end{cases}
        \end{equation*}
    
    \item For each $y_0\in\R^{d}$,
    $$
    \Psi_H(x,t)(y_0,\nabla F(x)y_0)=(y(t), \nabla F(x\cdot t)y(t)),\quad\forall t\leq 0, 
    $$
    where $y(t)$ satisfies \eqref{soln-H-ODE-linearized} with initial condition $(y_0,q_0)=(y_0,\nabla F(x)y_0)$.
\end{enumerate}
\end{thm}

\begin{proof}
First, given Lemma \ref{lem-linear-Hsystem}, we can apply the classical invariant manifold theory (recalled in Appendix \ref{app-Fenichel-theory} for reader's convenience) to find the local unstable manifold $\WW^{u}$ of $\MM\times\{0\}$. More precisely, replacing \eqref{app-ode}, $E(x)$ and $N(x)$ in Appendix \ref{app-Fenichel-theory} by \eqref{hamiltonian-system-intro}, $\EE(x,0)$ and $\NN(x,0)$, respectively, we apply Theorem \ref{app-thm-Fenichel} to find the $C^{k'}$ local unstable manifold $\WW^{u}$ of $\MM\times\{0\}$ with properties (i)-(iii).

\medskip

Next, we claim that there exist a neighborhood $\OO_0$ of $\MM$ and a $C^{k'}$ function $F:\OO_0\to \R^d$ such that $\graph F=\WW^{u}\cap \left(\OO_0\times\R^{d}\right)$. Indeed, as $\EE(x,0)=\graph Q(x)^{-1}$, any non-zero vector in $\EE(x,0)$ is not perpendicular to the $x$-plane. Clearly, $T_{(x,0)}(\MM\times\{0\})$ lies on the $x$-plane. Then, we see from (i) that all the non-zero tangent vectors of $\WW^{u}$ along $\MM\times \{0\}$ are not perpendicular to the $x$-plane. This ensures that near $\MM\times \{0\}$, $\WW^{u}$ can be regarded as the graph of some function over the $x$-plane, which shares the same regularity as $\WW^{u}$. The claim follows.

\medskip

Now, we verify desired properties of $F$.
\begin{enumerate}
\item We have shown $\graph F=\WW^{u}\cap \left(\OO_0\times\R^{d}\right)$. From which, $F=0$ on $\MM$ follows. 

\item It follows directly from (iii).

\item By (2), there is an open set $\MM\subset\OO_1\stst\OO_0$ such that $X_{x,F(x)}(t)\in \OO_0$ for all $x\in \OO_1$ and $t\leq 0$. The conclusion then follows from the negative invariance of $\WW^{u}$.

\item Since \eqref{hamiltonian-system-intro} is a Hamiltonian system, we conclude from (2) and the fact $H(x,0)=0$ that for any $x\in \OO_0$, $H(x,F(x))=\lim_{t\to -\infty} H(X_{x,F(x)}(t), P_{x,F(x)}(t))=0$.

\item Let $x\in\MM$. It follows from (i) that 
\begin{equation}\label{eqn-2023-06-26}
\begin{split}
 T_{(x,0)}\WW^{u}&=\EE(x,0)\oplus   T_{(x,0)}(\MM\times \{0\})\\
 &=\left\{(y,Q(x)^{-1}y):y\in {\range}\, Q(x)\right\} \oplus (T_x\MM\times \{0\}).
\end{split}
\end{equation}
Projecting $\EE(x,0)$ and $T_{(x,0)}(\MM\times \{0\})$ onto the first argument, we find $T_x\MM$ and ${\range}\, Q(x)$ are complementary in $\R^d$, namely,  $T_x\R^d=T_x\MM\oplus {\range}\, Q(x)$.

Since $\graph F=\WW^{u}\cap \left(\OO_0\times\R^{d}\right)$ by (1), any tangent vector in $T_{(x,0)}\WW^{u}$ takes the form $(y,\nabla F(x)y)$ for some $y\in T_x\R^d$. Comparing with \eqref{eqn-2023-06-26}, we find 
\begin{equation*}
    \nabla F(x) y=
    \begin{cases}
        0,&\quad y\in T_x\MM,\\
        Q(x)^{-1}y,& \quad  y\in {\range}\, Q(x).
    \end{cases}
\end{equation*}

\item Recall that $\Psi_{H}^t$ is the flow of the linearized system \eqref{hamiltonian-system-intro}. Since $\WW^{u}$ is a local unstable manifold and $(y_0,\nabla F(x)y_0)\in T_{(x,0)}\WW^{u}$, we see that $\Psi^t_H(y_0,\nabla F(x)y_0)\in T_{(x\cdot t, 0)}\WW^{u}$ for all  $t\leq 0$. Thus, $\Psi^t_H(y_0,\nabla F(x)y_0)=(y(t), \nabla F(x\cdot t)y(t))$ for all $t\leq0$.
\end{enumerate}

\medskip

Finally, we show that $\WW^{u}$ is a Lagrangian submanifold. Since the symplectic form $\om:=\sum_{i=1}^d dx^i\wedge dp^i$ is invariant under $\Phi^t_H$, and each point on $\WW^{u}$ is connected to some point on $\graph F$ via $\Phi^t_H$, we only need to prove that $\graph F$ is Lagrangian. Note that
$$
\om =\sum_{i=1}^d dx^i \wedge \left(\sum_{j=1}^d \pa_j F^i(x) dx^j\right)=\sum_{i=1}^d\sum_{j\neq i} \pa_j F^i(x) dx^i\wedge dx^j\quad\text{on}\quad \graph F,
$$
where $F_j$ is the $j$-th component of $F$. Since $\nabla F$ is symmetric by (5), it follows that $\om=0$ on $\graph F$. This completes the proof.
\end{proof}


\subsection{Conservativeness and regular solutions of HJE}\label{subsec-regular-sol-HJE}

Let $F$ and $\OO_1$ be as in Theorem \ref{lem-main-Hsystem} and recall $(X_{x,p}(t),P_{x,p}(t))=\Phi_H^t(x,p)$ and the constant $k'$ from Definition \ref{defn-linearly-stable-manifold}. 

\begin{thm}\label{conservativeness-F}
    If $k'\geq2$, then $F$ is conservative in $\OO_1$, that is, $\int_{\ga} F \cdot d\ga =0$ for any closed, continuous, and piecewise $C^{1}$ curve $\ga:[0,1]\to \OO_1$. 
\end{thm}

Theorem \ref{conservativeness-F} ensures that $F$ is the gradient of some function that we figure out now. Define 
\begin{equation}\label{def-hatV}
    \hat{V}(x):=\int_{-\infty}^0 L(X_{x,F(x)}, \dot X_{x,F(x)}), \quad x\in\OO_1,
\end{equation}
where $L$ is the Lagrangian associated with $H$ and is given in \eqref{Lagrangian}.




\begin{cor}\label{prop-hat-V-HJE}
If $k'\geq2$, then $\hat{V}$ is well-defined, non-negative, and vanishes only on $\MM$. Moreover, the following hold.  
\begin{enumerate}
 \item $\hat{V}\in C^{k'+1}(\OO_1)$ and $\nabla \hat{V}=F$.

    \item  $H(x,\nabla\hat{V}(x))=0$ for $x\in\OO_1$.
    
    \item $\text{\rm Hess}(\hat{V})|_x:T_x\R^d \otimes T_x\R^d\to \R$ is semi-positive definite and satisfies: 
    \begin{equation*}
        \text{\rm Hess}(\hat{V})|_x(y,z)
        \begin{cases}
            =0& \text{if\quad $y$ or $z\in T_x\MM$},\\
            >0& \text{if\quad $y=z\in T^{\bot}_x\MM$}, 
        \end{cases}
    \end{equation*}
where $T^{\bot}_x\MM\subset T_x\R^d$ denotes the orthogonal complement of $T_x\MM$. 
\end{enumerate}
\end{cor}

The rest of this subsection is devoted to the proof of Theorem \ref{conservativeness-F} and Corollary \ref{prop-hat-V-HJE}.

\begin{proof}[Proof of Theorem \ref{conservativeness-F}]
Let $\Ga$ be a closed, continuous, and piecewise $C^{1}$ curve in $\OO_1$ parameterized by $\ga:[0,1]\to \OO_1$. Then, $\ga(0)=\ga(1)$. For each $t\leq0$, let $\Ga_t$ be the curve parameterized by 
\begin{equation*}
\ga_t(u):=X_{\ga(u),F(\ga(u))}(t),\quad u\in [0,1].
\end{equation*}
In particular, $\ga_0=\ga$. We break the proof into three steps.

\medskip
\paragraph{\bf Step 1}
We show $\int_{\Ga} F \cdot d\ga=\int_{\Ga_t} F\cdot d\ga_t$ for all $t\leq 0$.   

Straightforward calculations show that 
\begin{equation*}
\begin{split}
    \frac{d }{d t}\int_{\Ga_t} F\cdot d\ga_t&=\frac{d }{d t}\int_0^1 F(\ga_t(u))\cdot \frac{\pa \ga_t(u)}{\pa u}du\\
    &=\int_0^1 \left[\nabla F(\ga_t(u))\frac{\pa \ga_t(u)}{\pa t}\cdot\frac{\pa \ga_t(u)}{\pa u}+F(\ga_t(u))\frac{\pa^2 \ga_t(u)}{\pa t\pa u}\right]du\\
    &=\int_0^1\frac{\pa }{\pa u}\left[ F(\ga_t(u))\cdot \frac{\pa \ga_t(u)}{\pa t}\right]du=F(\ga_t(1))\cdot \frac{\pa \ga_t(1)}{\pa t}-F(\ga_t(0))\cdot \frac{\pa \ga_t(0)}{\pa t}=0,
\end{split}
\end{equation*}
where we used the symmetry of $\nabla F$ in the third equality. The result follows.

\medskip

\paragraph{\bf Step 2} Let $\be'_0\in (\be_0,\be_*)$ satisfy $\frac{\be'_0}{\be_*-\be'_0}<\frac{1}{k'}$ and be fixed. We prove $\sup_{u\in [0,1]} \left|\frac{\pa \ga_t(u)}{\pa u}\right|\lesssim e^{\be'_0 |t|}$ for all $t\leq 0$.

The definition of $\ga_t$ ensures that for each $u\in[0,1]$, $t\mapsto\left(\ga_t(u),F(\ga_t(u))\right)$ is a solution of \eqref{hamiltonian-system-intro}. Hence, $\frac{\partial\ga_t(u)}{\partial t}=H_p\left(\ga_t(u),F(\ga_t(u))\right)=(2AF+b)(\ga_t(u))$. Differentiating this equation with respect to $u$ results in
\begin{equation}\label{eqn-2023-06-04-1}
    \frac{\pa^{2}\ga_t(u)}{\pa t\pa u}=(2\nabla AF+2A\nabla F+\nabla b)(\ga_t(u))\frac{\pa \ga_t(u)}{\pa u},
\end{equation}
where $\nabla AF$ is a $d\times d$ matrix whose $i,j$-entry is $\sum_{k}\partial_{j}a^{ik}F_{k}$, in which, $F_{k}$ is the $k$-th component of $F$. Since $\ga_t(u)=X_{\ga(u),F(\ga(u))}(t)$, we see from \eqref{eqn-2023-06-04-1} that $t\mapsto\frac{\pa \ga_t(u)}{\pa u}$ solves the linear system
\begin{equation}\label{eqn-2022-04-26-2}
    \dot Z=(2\nabla AF+2A\nabla F+\nabla b)(X_{x,F(x)}(t)) Z,\quad t\leq0.
\end{equation}
with $x=\ga(u)\in \OO_1$ and $Z(0)=\frac{\pa\ga_0(u)}{\pa u}$.

For $x\in\OO_{1}$, denote by $Z^{x}$ the unique solution of \eqref{eqn-2022-04-26-2} with initial condition $Z^{x}(0)\in\R^{d}$. We claim that 
\begin{equation}\label{eqn-2023-07-20-1}
    \begin{split}
        |Z^{x}(t)|\lesssim e^{\be'_0 |t|}|Z^{x}(0)|,\quad\forall t\leq 0.
    \end{split}
\end{equation}
Since $\sup_{u\in[0,1]}\left|\frac{\pa\ga_0(u)}{\pa u}\right|<\infty$, the conclusion in this step follows readily from \eqref{eqn-2023-07-20-1}. 

We verify \eqref{eqn-2023-07-20-1} to finish the proof of {\bf Step 2}. Fix $x\in\OO_1$ and $\be\in (\beta_0',\be_*)$. Theorem \ref{lem-main-Hsystem} (2) yields
\begin{equation}\label{eqn-2022-04-26-5}
    \lim_{t\to -\infty}e^{\be |t|}\sup_{y\in\OO_1}\left|X_{y,F(y)}(t)-\pi(y)\cdot t\right|=0,
\end{equation}
Since $\pi(x)\cdot t\in\MM$ for all $t\leq0$ and $F=0$ on $\MM$ by Theorem \ref{lem-main-Hsystem} (1), we can regard \eqref{eqn-2022-04-26-2} as a perturbation of the following linear system: 
\begin{equation*}\label{eqn-2022-04-26-2-1}
    \dot{\ZZ}=(2A\nabla F+\nabla b)(\pi(x)\cdot t)\ZZ,\quad t\leq0.
\end{equation*}
Denoted by $\Psi_{\ZZ}(\pi(x),t)$ its principal fundamental matrix solution at initial time $0$. Fix $0<\de\ll 1$ such that $\be'_0-\de\in (\be_0,\be_*)$ and $\frac{\be'_0-\de}{\be_*-\be'_0+\de}<\frac{1}{k'}$.  Note that 
\begin{equation}\label{eqn-2022-04-26-4}
    \|\Psi_{\ZZ}(\pi(x),t)\|\lesssim e^{(\be'_0-\de) |t|},\quad t\leq 0.
\end{equation}
Indeed, we deduce from Theorem \ref{lem-main-Hsystem} (6) that 
$$
(y(t),q(t)):=\Psi_H(\pi(x),t)\left(y_0,\nabla F(\pi(x))y_0\right)=\left(y(t), \nabla F(\pi(x)\cdot t)y(t)\right).
$$
Then, the $y$-equation in \eqref{H-ODE-linearized} becomes $\dot {y}=2A(\pi(x)\cdot t) \nabla F(\pi(x)\cdot t)y+\nabla b(\pi(x)\cdot t)y$, which is just the equation for $\ZZ$. It follows from Lemma \ref{lem-linear-Hsystem} that $\|\Psi_{\ZZ}(\pi(x),t)\|\leq \|\Psi_H(\pi(x),t)\|\lesssim e^{(\be'_0 -\de)|t|}$ for all $t\leq 0$, that is, \eqref{eqn-2022-04-26-4} is true.

Setting
$$
G^{x}(t):=(2\nabla AF+2A\nabla F+\nabla b)(X_{x,F(x)}(t))-(2A\nabla F+\nabla b)(\pi(x)\cdot t),
$$
we can rewrite \eqref{eqn-2022-04-26-2} as
\begin{equation*}
    \dot Z=(2A\nabla F+\nabla b)(\pi(x)\cdot t)Z+G^{x}(t) Z,\quad t\leq0.
\end{equation*}
An application of the variation of constants formula yields
\begin{equation}\label{eqn-2022-04-26-3}
    Z^{x}(t)=\Psi_{\ZZ}(\pi(x),t)Z^{x}(0)+\int_0^t \Psi_{\ZZ}(\pi(x)\cdot s,t-s)G^{x}(s) Z^{x}(s)ds,\quad t\leq0.
\end{equation}

Since $\pi(x)\cdot t\in\MM$ for all $t\leq0$ and $F=0$ on $\MM$, we see that 
\begin{equation}\label{estimate-2024-10-02}
    \begin{split}
    \left|F(X_{x,F(x)}(t))\right|\leq \|\nabla F\|_{L^{\infty}(\OO_1)}\times \left|X_{x,F(x)}(t)-\pi(x)\cdot t\right|\lesssim e^{\be t},\quad \forall t\leq 0,
    \end{split}
\end{equation}
and thus,
\begin{equation*}
\begin{split}
    |G^{x}(t)|&\leq 2|(\nabla A F)(X_{x,F(x)}(t))|+\left\|\nabla (2A\nabla F+\nabla b)\right\|_{L^{\infty}(\OO_1)}\times |X_{x,F(x)}(t)-\pi(x)\cdot t|\\
    &\lesssim |X_{x,F(x)}(t)-\pi(x)\cdot t|
    \lesssim e^{\be t},\quad \forall t\leq 0,
\end{split}
\end{equation*}
where we used \eqref{eqn-2022-04-26-5} in the last inequality. It follows from \eqref{eqn-2022-04-26-4} and \eqref{eqn-2022-04-26-3} that there is $C>0$ such that
\begin{equation}\label{eqn-Mar-25}
    \begin{split}
        |Z^{x}(t)|&\leq Ce^{(\be'_0-\de) |t|}|Z^{x}(0)|+C\int_{t}^0 e^{(\be'_0-\de) |t-s|}e^{\be s}|Z^{x}(s)|ds\\
        & = Ce^{-(\be'_0-\de)t}|Z^{x}(0)|+C\int_t^0 e^{-(\be'_0-\de)(t-s)+\be s} |Z^{x}(s)|ds,\quad \forall t\leq 0.
    \end{split}
\end{equation}
Setting $f(t):=\int_t^0 e^{-(\be'_0-\de)(t-s)+\be s} |Z^{x}(s)|ds$, we derive from $f'(t)=-e^{\be t}|Z^{x}(t)|-(\be'_0-\de) f(t)$ and \eqref{eqn-Mar-25} that 
$$
-f'(t)-(\be'_0-\de) f(t)\leq Ce^{(\be-\be'_0+\de) t}|Z^{x}(0)|+Ce^{\be t}f(t),\quad t\leq 0.
$$
Clearly, there exists $t_0=t_0(\de)<0$ such that $e^{\be t}\leq \de$ for $t\leq t_0$, and thus, $f'(t)+\be'_0 f(t)+Ce^{(\be-\be'_0+\de) t}|Z^{x}(0)|\geq 0$. An application of Gr\"onwall's inequality results in
$$
f(t)\leq C e^{\be'_0(t_0-t)}f(t_0)+\frac{C|Z^{x}(0)|}{\be+\de} e^{(\be+\de)t_0-\be'_0t}\lesssim e^{-\be'_0 t},\quad \forall t\leq t_0,
$$
which together with \eqref{eqn-Mar-25} leads to $|Z^{x}(t)|\lesssim e^{-(\be'_0-\de) t}|Z^{x}(0)|+e^{-\be'_0 t}\lesssim e^{\be'_0 |t|}$ for all $t\leq 0$, proving \eqref{eqn-2023-07-20-1}.

\medskip

\paragraph{\bf Step 3} Since $\ga_t(u)=X_{\ga(u),F(\ga(u))}(t)$, we apply \eqref{estimate-2024-10-02} to find $\sup_{u\in [0,1]}|F(\ga_t(u))|\lesssim e^{\be t}$ for $t\leq 0$. It follows from  {\bf Step 2} that
\begin{equation*}
    \left|\int_{\Ga_t} F \cdot d\ga_t\right|=\left|\int_0^1 F(\ga_t(u))\cdot  \frac{\pa \ga_t(u)}{\pa u} du\right|\lesssim e^{(\be-\be'_0)t}, \quad t\leq 0,
\end{equation*}
which together with {\bf Step 1} yields $\int_{\Ga} F \cdot d\ga= \lim_{t\to -\infty}\int_{\Ga_t} F \cdot d\ga_t=0$. This completes the proof.
\end{proof}


Finally, we prove Corollary \ref{prop-hat-V-HJE}.

\begin{proof}[Proof of Corollary \ref{prop-hat-V-HJE}]
By Theorem \ref{conservativeness-F}, there exists $\tilde{V}\in C^{k'+1}(\OO_1)$ such that $F=\nabla \tilde{V}$ in $\OO_1$. The fact $F=0$ on $\MM$ ensures that $\tilde{V}$ is constant on $\MM$. Since $X_{x,F(x)}(t)$ approaches $\MM$ as $t\to -\infty$ thanks to Theorem \ref{lem-main-Hsystem} (2), we find $\lim_{t\to -\infty}\tilde{V}(X_{x,F(x)}(t))=\tilde{V}|_{\MM}$.

Let $x\in \OO_1$. Note that Theorem \ref{lem-main-Hsystem} (4) yields 
$$
H(X_{x,F(x)}(t),P_{x,F(x)}(t))=H(X_{x,F(x)}(t),F(X_{x,F(x)}(t)))=0,\quad\forall  t\leq 0.
$$
It then follows from Lemma \ref{Hamiltonian-Lagrangian-eqivalence} that
\begin{equation*}
\begin{split}
    L(X_{x,F(x)}(t), \dot X_{x,F(x)}(t))&=\dot X_{x,F(x)}(t)\cdot P_{x,F(x)}(t)-H(X_{x,F(x)}(t),P_{x,F(x)}(t))\\
    &=\dot X_{x,F(x)}(t) \cdot F(X_{x,F(x)}(t)),\quad \forall t\leq 0,
\end{split}
\end{equation*}
and hence, 
$$
\hat{V}(x)=\int_{-\infty}^0 \dot X_{x,F(x)} \cdot F(X_{x,F(x)})=\left.\tilde{V}(X_{x,F(x)}(t))\right|_{t=-\infty}^0=\tilde{V}(x)-\tilde{V}|_{\MM}.
$$ 
This indicates that $\hat{V}$ is well-defined and vanishes only on $\MM$. Hence, (1) holds. The non-negativeness of $\hat{V}$ follows directly from its definition in \eqref{def-hatV}. Theorem \ref{lem-main-Hsystem} (4)  yields (2) readily.




For (3), we note that $\text{Hess}(\hat{V})|_x (y,z)=\langle y, \nabla F(x)z\rangle$ for $(y,z)\in T_x\R^d\otimes T_x\R^d$. This together with Theorem \ref{lem-main-Hsystem} (5) implies that 
$\text{Hess}(\hat{V})|_x (y,z)=0$ if $y$ or $z\in T_x\MM$.  As $T^{\bot}_x\MM$ is complementary to $T_x\MM$, we deduce that $\text{Hess}(\hat{V})|_x$ is non-degenerate on $T^{\bot}_x\MM$, leading to $\text{Hess}(\hat{V})|_x (y,y)>0$ for $y\in T^{\bot}_x\MM$. This completes the proof.
\end{proof}


\section{\bf Regularity of the quasi-potential}\label{sec-regularity-quasi-potential}

In this section, we study the regularity of the quasi-potential $V$, defined in \eqref{def-quasi-potential-V-FW-sense}, under the additional assumption that $\AAa$ is a normally contracting invariant manifold (see Definition \ref{defn-linearly-stable-manifold}). In particular, we prove Theorems \ref{thm-regularity-V} and \ref{thm-global-regularity} in Subsections \ref{subsec-local-regularity} and \ref{subsec-global-regularity}, respectively.


\subsection{Local regularity}\label{subsec-local-regularity}

In this subsection, we apply the theory developed in Section \ref{sec-regular-sol-HJE} to study the regularity of $V$ in a neighborhood of the maximal attractor $\AAa$ and prove Theorem \ref{thm-regularity-V}. We first prove a local uniqueness result asserting that any non-negative $C^1$ solution of the HJE \eqref{eqn-HJE} that vanishes only on $\AAa$ must coincide with $V$ in a neighborhood of $\AAa$. 

Recall that $\OO_{\AAa,\Om}$ denotes the set of open and connected sets $\OO$ satisfying $\AAa\subset \OO\stst \Om$. Let $\OO\in\OO_{\AAa,\Om}$. For $U\in C(\OO)$, set $\rho_{U,\OO}:=\liminf_{x\to \pa\OO} U(x)$, $\OO^{U}:=\{x\in \OO: U(x)<\rho_{U,\OO}\}$, and $\OO^{U}_{\rho}:=\{x\in \OO: U(x)<\rho\}$ for $\rho\in(0,\rho_{U,\OO})$. Recall from Corollary \ref{cor-local-variational-representation} that 
\begin{equation}\label{minimize-2024-10-07}
V(x)=\min_{\phi\in\Phi_{x,\OO}}I(\phi),\quad\forall x\in \OO^{V},    
\end{equation}
where $\Phi_{x,\OO}=\left\{\phi\in AC((-\infty,0];\OO):\dot{\phi}\in L^2_{loc}((-\infty,0];\R^{d}),\,\,\phi(0)=x,\,\,\lim_{t\to -\infty}\dist(\phi(t),\AAa)=0\right\}$.

\begin{lem}\label{lem-unique-minimizer}
Assume {\bf(H1)}-{\bf(H3)}. Let $\OO\in\OO_{\AAa,\Om}$. Assume that \eqref{eqn-HJE} admits a solution $W\in C^1(\OO)$ that is non-negative and vanishes only on $\mathcal{A}$. Then, the following hold for each $x\in\OO^W\cap \OO^V$.
\begin{itemize}
    \item[(1)] $W(x)=V(x)$.
    \item[(2)] $\phi$ is a minimizer corresponding to $V(x)$ (see Definition \ref{defn-minimizer-V}) if and only if $\phi\in\Phi_{x,\OO}$ and solves
\begin{equation}\label{eqn-Aug-23}
    \dot{\phi}=2A(\phi)\nabla V(\phi)+b(\phi)\quad\text{a.e. on }\quad(-\infty,0).
\end{equation} 

\item[(3)] If, in addition, $\nabla W$ is locally Lipschitz, then there is a unique minimizer corresponding to $V(x)$, and the unique minimizer belongs to $C^{1}((-\infty,0];\OO^{W}\cap\OO^{V})$.
\end{itemize}

\end{lem}
\begin{proof}
It is more or less a rephrase of \cite[Theorem 5.4.3]{FW12} that is adapted to meet our needs. We provide the proof for reader's convenience. Set $\ell:=b+A\nabla W$ and $\hat{b}:= A\nabla W+\ell=2A\nabla W+b$ in $\OO$. As $W$ solves the HJE \eqref{eqn-HJE}, that is, $W$ satisfies 
\begin{equation}\label{HJE-W}
    \nabla W\cdot A\nabla W+b\cdot\nabla W=0\quad \text{in} \quad \OO, 
\end{equation}
we find
\begin{equation}\label{orthogonal-decomp}
    \nabla W\cdot \ell=0\quad \text{in} \quad \OO.
\end{equation}

In the rest of the proof, we fix some $x\in\OO^V\cap \OO^W$. We first claim that
\begin{equation}\label{eqn-Oct-6}
    I(\phi)=\frac{1}{4}\int_{-\infty}^0 \left[\hat{b}(\phi)-\dot \phi\right] \cdot A^{-1}(\phi) \left[\hat{b}(\phi)-\dot\phi\right]+W(x),\quad \forall\phi\in\Phi_{x,\OO}.
\end{equation}
Indeed, for $\phi\in \Phi_{x,\OO}$, we apply the definition of $\hat{b}$ to calculate
\begin{equation*}
    \begin{split}
        I(\phi)
        &=\frac{1}{4}\int_{-\infty}^0 \left[\hat{b}(\phi)-\dot \phi\right] \cdot A^{-1}(\phi) \left[\hat{b}(\phi)-\dot\phi\right]-\int_{-\infty}^0 \nabla W(\phi)\cdot \left[\hat{b}(\phi)-\dot\phi\right]+\int_{-\infty}^0\nabla W(\phi)\cdot A(\phi)\nabla W(\phi).
    \end{split}
\end{equation*}
It follows from \eqref{HJE-W} that
\begin{equation*}
    \begin{split}
        &I(\phi)-\frac{1}{4}\int_{-\infty}^0\left[\hat{b}(\phi)-\dot \phi\right] \cdot A^{-1}(\phi) \left[\hat{b}(\phi)-\dot\phi\right]\\
        &\qquad=-\int_{-\infty}^0 \nabla W(\phi)\cdot \left[\hat{b}(\phi)-\dot \phi+b(\phi)\right]\\
        &\qquad=-\int_{-\infty}^0 \nabla W(\phi)\cdot \left[2\ell(\phi)-\dot \phi\right]=\int_{-\infty}^0 \nabla W(\phi)\cdot \dot \phi=W(\phi(0))-\lim_{t\to-\infty}W(\phi(t))=W(x),
    \end{split}
\end{equation*}
where we used \eqref{orthogonal-decomp} in the third equality and $W=0$ on $\AAa$ in the last equality.

Given \eqref{minimize-2024-10-07} and \eqref{eqn-Oct-6}, it is clear that if there exists a $\hat\phi\in\Phi_{x,\OO}$ satisfying \eqref{eqn-Aug-23} with $V$ reaplaced by $W$, then $\hat{\phi}$ is a minimizer corresponding to $V(x)$ and $V(x)=W(x)$, proving (1). 

Now, we construct such a $\hat{\phi}$. Note that $\hat{b}$ is only continuous in $\OO$ so that the classical local well-posedness result for ODEs can not be applied here. Note that for fixed $\rho_1,\rho_2\in (0,\rho_{W,\OO})$ with $\rho_1<\rho_2$, there holds $\OO_{\rho_1}^{W}\stst \OO_{\rho_2}^{W}\stst \OO$. Since $\hat{b}$ is uniformly bounded on $\ol{\OO^{W}_{\rho_2}}$, a careful examination of the proof of the classical Peano's existence theorem yields a $\de>0$ such that for each $y\in \OO_{\rho_1}^W$, \eqref{eqn-Aug-23} admits a solution $\phi_y\in C^1((-\de,0];\OO_{\rho_2}^W)$ with $\phi_{y}(0)=y$. An application of \eqref{HJE-W} results in 
\begin{equation*}
    \begin{split}
        \frac{d}{dt}W(\phi_y(t))&=(2\nabla W\cdot A\nabla W+b\cdot \nabla W)(\phi_y(t))=(\nabla W\cdot A\nabla W)(\phi_y(t))\geq0.
    \end{split}
\end{equation*}
That is, $t\mapsto W(\phi_y(t))$ is non-decreasing on $(-\de,0]$. Hence,  $\phi_y(t)\in\OO_{\rho_1}^W$ for all $t\in (-\de,0]$. Since this is true for any $y\in \OO^W_{\rho_1}$, we can extend $\phi_y$ to be $\hat{\phi}_{y}$ in $\OO_{\rho_1}^W$ over $(-\infty,0]$ as follows: 
\begin{equation*}
    \begin{split}
        \hat{\phi}_{y}(t)=\phi_{y}(t),&\quad t\in(-\de,0],\\
        \hat{\phi}_{y}(t)=\phi_{\hat{\phi}_{y}(-k\de)}(t+k\de),&\quad t\in(-(k+1)\de,-k\de], \quad k=1,2,\dots.
    \end{split}
\end{equation*}
Obviously, $\hat{\phi}_{y}:(-\infty,0]\to\OO_{\rho_1}^W$ is absolutely continuous and satisfies \eqref{eqn-Aug-23}. We make $\de$ smaller (if necessary) so that $\phi_y\in C^1([-\de,0];\OO_{\rho_2}^W)$. Then, $\dot{\hat{\phi}}_{y}\in L^{2}_{loc}((-\infty,0];\R^{d})$. 

As $\OO^W=\bigcup_{\rho\in (0,\rho_{W,\OO})} \OO^W_{\rho}$, we follow the above procedure to find a $\hat{\phi}_x\in AC((-\infty,0];\OO)$ satisfying $\dot{\hat{\phi}}_{x}\in L^{2}_{loc}((-\infty,0];\R^{d})$ and \eqref{eqn-Aug-23}. We prove $\lim_{t\to -\infty}\dist(\hat\phi_x(t),\mathcal{A})=0$ so that $\hat{\phi}_{x}\in \Phi_{x,\OO}$. Since $t\mapsto W(\hat{\phi}_{x}(t))$ is non-decreasing and $W$ is non-negative and vanishes only on $\mathcal{A}$, it suffices to prove $\underline{\rho}:=\lim_{t\to -\infty}W(\hat{\phi}_{x}(t))=0$. Suppose on the contrary that $\underline{\rho}>0$. Then, $\hat{\phi}_{x}$ stays in $\OO\sm \OO^W_{\underline{\rho}}$, which is away from $\mathcal{A}$. Noting that \eqref{main-ode} has no invariant set in $\OO\sm \OO^W_{\underline{\rho}}$, we apply \cite[Lemma 3.1]{F78} to conclude that $W(x)-\lim_{t\to -\infty}W(\hat\phi_x(t))=\int_{-\infty}^0 L(\hat{\phi}_{x},\dot{\hat{\phi}}_{x})=\infty$, leading to a contradiction. Hence, $\lim_{t\to -\infty}W(\hat{\phi}_{x}(t))=0$. Clearly, $\hat{\phi}:=\hat{\phi}_{x}$ satisfies all the requirements. 

\medskip

Now, we prove (2). The sufficiency is an immediate consequence of \eqref{minimize-2024-10-07} and \eqref{eqn-Oct-6}. To show the necessity, let $\phi$ be a minimizer corresponding to $V(x)$. It follows from the standard dynamic programming arguments that for each $t\leq 0$, $\phi(\cdot +t)$  is a minimizer corresponding to $V(\phi(t))$. 
Thus, $t\mapsto V(\phi(t))$ is increasing and $\phi(t)\in \ol{\OO^V_{\rho}}\subset \OO$ for all $t\in (-\infty,0]$ with $\rho:=V(x)$. This shows $\phi\in \Phi_{x,\OO}$. Thanks to \eqref{eqn-Oct-6} and $V(x)=I(\phi)=W(x)$, we arrive at $\int_{-\infty}^0 \left[\hat{b}(\phi)-\dot{\phi}\right] \cdot A^{-1}(\phi) \left[\hat{b}(\phi)-\dot{\phi}\right]=0$. Hence, $\phi$ satisfies \eqref{eqn-Aug-23}.

\medskip


For (3), we note that if $W$ is locally Lipschitz, so is $\hat{b}$, and therefore, the classical local well-posedness result of ODE applies, leading to the uniqueness of the minimizer corresponding to $V(x)$. If $\phi$ is the unique minimizer corresponding to $V(x)$, the monotonicity of $t\mapsto W(\phi(t))$ and $t\mapsto V(\phi(t))$ yields that $\phi\in C^{1}((-\infty,0];\OO^{W}\cap\OO^{V})$.
\end{proof}

Now, we prove Theorem \ref{thm-regularity-V}.

\begin{proof}[\bf Proof of Theorem \ref{thm-regularity-V}] 
Under the assumption of Theorem \ref{thm-regularity-V}, we can apply Theorem \ref{lem-main-Hsystem} and Corollary \ref{prop-hat-V-HJE} to find the domain $\AAa\subset\OO_1\stst\Om$ and a function $\hat{V}\in C^{k'+1}(\OO_1)$ defined in \eqref{def-hatV}. It then follows from Lemma \ref{lem-unique-minimizer} (with $W=\hat{V}$ and $\OO=\OO_1$) that $V=\hat{V}$ in $\OO_{2}:=\OO_1^{\hat V}\cap \OO_1^V$.  Other conclusions in the statement of Theorem \ref{thm-regularity-V} follow from Theorem \ref{lem-main-Hsystem}, Corollary \ref{prop-hat-V-HJE}, and Lemma \ref{lem-unique-minimizer}. 
\end{proof}

\subsection{Global regularity}\label{subsec-global-regularity}

We study the global regularity of $V$ and prove Theorem \ref{thm-global-regularity}. Theorem \ref{thm-global-regularity} actually follows from the strategies laid out in \cite{DD85} (treating the problem in the case that $\AAa$ is a linearly stable equilibrium) that build upon two key elements. One is the attainability of the infimum in the variational formula \eqref{def-quasi-potential-V-FW-sense} used to define the quasi-potential $V$ (see Lemma \ref{lem-quasi-potential-basic-results}), that is, for each $x\in\Om^{V}$, there exist minimizers corresponding to $V(x)$ (see Definition \ref{defn-minimizer-V}). The other is stated in the following result. 

\begin{lem}\label{lem-minimizer-on-unstable-manifold}
    Assume {\bf(H1)}-{\bf(H4)}. Let $x_0\in\Om^{V}$. Suppose $x(t)$ is a minimizer corresponding to $V(x_0)$ and set $p(t):=\partial_vL(x(t), \dot{x}(t))$. Then, there exists $T>0$ such that $(x(t),p(t))\in \WW^{u}$ for all $t\leq -T$, where $\WW^{u}$ is the local unstable manifold given in Theorem \ref{thm-regularity-V}. 
\end{lem}
\begin{proof}
Lemma \ref{thm-extremal-orbit} says that $(x(t),p(t))$ satisfies \eqref{hamiltonian-system-intro} and $\lim_{t\to -\infty}{\rm dist}(x(t),\AAa)=0$. Then, there exists $T>0$ such that $x(t)\in\OO_2$ for all $t\leq-T$, where $\OO_{2}$ is given in Theorem \ref{thm-regularity-V}. By the standard dynamic programming arguments, $x(\cdot-T)$ is a minimizer corresponding to $V(x(-T))$. Since $x(-T)\in\OO_{2}$, Theorem \ref{thm-regularity-V} asserts that $x(\cdot-T)$ is the only minimizer corresponding to $V(x(-T))$ and satisfies 
$$
\dot{x}(-T)=2A(x(-T))\nabla V(x(-T))+b(x(-T)). 
$$
The fact that $(x(t),p(t))$ satisfies \eqref{hamiltonian-system-intro} yields 
$$
\dot{x}(-T)=2 A(x(-T))p(-T)+b(x(-T)). 
$$
It follows from the positive definiteness of $A$ that $p(-T)=\nabla V(x(-T))$, leading to $(x(-T),p(-T))\in \WW^{u}$, and hence, $(x(t),p(t))\in \WW^{u}$ for all $t\leq -T$. 
\end{proof}

\begin{rem}
    When the maximal attractor $\AAa$ is a linearly stable equilibrium so that $\AAa\times\{0\}$ is a hyperbolic fixed point of the Hamiltonian system \eqref{hamiltonian-system-intro}, the conclusion in Lemma \ref{lem-minimizer-on-unstable-manifold} is a direct consequence of Lemma \ref{thm-extremal-orbit}. For a general normally contracting invariant manifold $\AAa$, Lemma \ref{thm-extremal-orbit} alone can not yield Lemma \ref{lem-minimizer-on-unstable-manifold} thanks to the existence of center manifolds of $\AAa\times\{0\}$ on which the dynamics of \eqref{hamiltonian-system-intro} remains unclear (see Remark \ref{rem-key-ideas}). 
\end{rem}

\begin{proof}[Proof of Theorem \ref{thm-global-regularity}]
    We brief on how aforementioned two key elements enter into the proof and refer the reader to \cite[Section 5]{DD85} for details.

    For any given $x_0\in\Om^{V}$, let $x(t)$ be a minimizer corresponding to $V(x_0)$. The existence of minimizers is ensured by Lemma \ref{lem-quasi-potential-basic-results}. By Lemma \ref{lem-minimizer-on-unstable-manifold} or its proof, there exists $T>0$ such that $x(-T)\in\OO_{2}$, where $\OO_2$ is given in Theorem \ref{thm-regularity-V}, and $V(x_0)=V(x(-T))+\int_{-T}^{0}L(x(t),\dot{x}(t))dt$. Moreover, $p(t):=\partial_{v}L(x(t),\dot{x}(t))=\nabla V(x(t))$ for all $t\leq-T$. Given these, we can follow the proof of \cite[Theorem 6]{DD85} to show that for each $t<0$, $V$ is $C^{k'+1}$ in a neighborhood of $x(t)$.  

    Now, let $G$ be the union of all the open sets in $\Om^{V}$ where $V$ is $C^{k'+1}$. Obviously, $G$ is non-empty and open. The above arguments imply that if $x_0\in\Om^{V}$ and $x(t)$ is a minimizer corresponding to $V(x_0)$, then $x(t)\in G$ for all $t<0$. The density of $G$ in $\Om^{V}$ follows.

    Finally, for a fixed $x_0\in G$. We follow \cite[Corollary 5]{DD85} to show that there is a unique minimizer $x(t)$ corresponding to $V(x_0)$ and $p(0)=\nabla V(x_0)$. It follows that $p(t)=\nabla V(x(t))$ for all $t\leq0$, which together with Lemma \ref{thm-extremal-orbit} implies that $x(t)$ satisfies $\dot{x}=2A(x)\nabla V(x)+b(x)$. The fact $(x(t),V(x(t)))\in\WW^{u}$ for all $t\leq-T$ for some $T>0$ then follows from Lemma \ref{lem-minimizer-on-unstable-manifold}.
\end{proof}


\section{\bf Applications}\label{sec-application}

In this section, we discuss applications of Theorems \ref{thm-main-result}, \ref{thm-regularity-V} and \ref{thm-global-regularity} to stationary and quasi-stationary distributions of the following randomly perturbed dynamical system:
\begin{equation}\label{main-sde}
    dx=b(x)dt+\ep\si(x)dW_{t},\quad x\in\UU,
\end{equation}
where $0<\ep\ll1$, $\UU\subset\R^{d}$ is open and connected, $b\in C^{k}(\UU,\R^{d})$ with $k>2$, $\si\in C^{k+1}(\UU,\R^{d\times m})$ for some $m\geq d$, $A=(a^{ij}):=\si\si^{\top}$ is pointwise positive definite in $\UU$, and $W_{t}$ is the standard $m$-dimensional Wiener process.

\begin{rem}
We use $\UU$ instead of $\R^{d}$ as the whole state space for the reason that $\UU$ is typically $(0,\infty)^{d}$ for applications in chemical reactions, ecology, and epidemiology. 

In some applications, the SDE \eqref{main-sde} is naturally equipped with absorbing or reflecting boundary conditions. When it is clear whether stationary or quasi-stationary distributions should be considered in the given context, the subsequent discussions can be easily adapted accordingly. Hence, we do not bother with boundary conditions in the following discussion.
\end{rem}

Assumptions on $b$ and $\si$ guarantee the local existence and uniqueness of strong solutions of \eqref{main-sde}. Denote by $X_{t}^{\ep}$ the unique strong solution of \eqref{main-sde} and by $\P_{\mu}^{\ep}$ the law of $X_{t}^{\ep}$ with initial distribution $\mu$. The expectation associated with $\P_{\mu}^{\ep}$ is denoted by $\E_{\mu}^{\ep}$. When $\mu=\de_{x}$, the Dirac measure at $x$, $\P_{\mu}^{\ep}$ and $\E_{\mu}^{\ep}$ are written as $\P_{x}^{\ep}$ and $\E_{x}^{\ep}$, respectively. The generator and Fokker-Planck operator associated with \eqref{main-sde} are denoted by $L_{\ep}:=\frac{\ep^{2}}{2}\sum_{i,j=1}^{d}a^{ij}\partial_{ij}^{2}+b\cdot\nabla$ and $L_{\ep}^{*}u:=\frac{\ep^{2}}{2}\sum_{i,j=1}^{d}\partial_{ij}^{2}(a^{ij}u)-\nabla\cdot(bu)$, respectively.

Let $\vp^{t}$ be the local flow generated by the ODE
\begin{equation}\label{app-ode}
     \dot{x}=b(x),\quad x\in \UU,
\end{equation}
whose local well-posedness is ensured by the regularity assumption on $b$.

If $\Om\subset\UU$ is open, connected, and positively invariant under $\vp^{t}$, and $\AAa_{\Om}$ is the maximal attractor of $\vp^{t}$ in $\Om$ and is an equivalence class (see Definition \ref{def-equiv-class}), we introduce the quasi-potential
\begin{equation*}
    V_{\Om}(x):=\inf_{\phi\in\Phi_{x,\Om}}\frac{1}{4}\int_{-\infty}^0\left[b(\phi)-\dot{\phi}\right]\cdot A^{-1}(\phi)\left[b(\phi)-\dot{\phi}\right], \quad \forall x\in \Om,
\end{equation*}
where  
$$
\Phi_{x,\Om}:=\left\{\phi\in AC((-\infty,0];\Om):\dot{\phi}\in L^2_{loc}((-\infty,0];\R^{d}),\,\,\phi(0)=x,\,\,\lim_{t\to -\infty}\dist(\phi(t),\AAa)=0\right\}.
$$
Set $\rho_{\Om}:=\liminf_{x\to\partial\Om}V_{\Om}(x)$ and $\Om^{*}=\{x\in\Omega:V_{\Om}(x)<\rho_{\Om}\}$.

\subsection{Stationary distributions}\label{subsec-application-sd}

Recall that for each $\ep$, a probability measure $\mu$ on $\UU$ is called a \emph{stationary distribution} of \eqref{main-sde} if $\P_{\mu}^{\ep}[X_{t}^{\ep}\in\bullet]=\mu$ for all $t\geq0$. We point out that under the current assumptions on the regularity of $b$ and $\si$ and positive definiteness of $\si\si^{\top}$, there holds the uniqueness of stationary distributions of \eqref{main-sde} for each fixed $\ep$ (see e.g. \cite{DaZ92}). The existence of stationary distributions needs dissipative conditions often given by Lyapunov functions. 

\begin{defn}
    Suppose $U\in C^{2}(\UU)$ satisfies $U(x)\to\infty$ as $x\to\partial\UU$. $U$ is called a \emph{uniform Lyapunov function} of \eqref{main-sde} if there is a compact set $K\subset\UU$ and $\ga>0$ such that $\sup_{\ep}L_{\ep}U\leq-\ga$ in $\UU\setminus K$.
\end{defn}

\begin{rem}\label{rem-application}
If $U$ is a uniform Lyapunov function of \eqref{main-sde}, then it is also a Lyapunov function of \eqref{app-ode}, and hence, $\UU$ is positively invariant under $\vp^{t}$ and $\vp^{t}$ admits the global attractor $\AAa_{\UU}$ (see Appendix \ref{appendix-attractor}).
\end{rem}

The following result, addressing the existence, uniqueness, tightness, and concentration estimates of stationary distributions of \eqref{main-sde}, is known (see e.g. \cite{Khasminskii12,BKRS15,HJLY15,JSY19}).

\begin{prop}\label{prop-stationary-measure}
    Suppose \eqref{main-sde} admits a uniform Lyapunov function. Then, the following hold.
\begin{itemize}
    \item[(1)] For each $\ep$, \eqref{main-sde} admits a unique stationary distribution $\mu_{\ep}$ with a density $u_{\ep}\in C^{2}(\UU)$ obeying
\begin{equation*}\label{sd-june-23}
    \begin{cases}
        L_{\ep}^{*}u_{\ep}=0\quad\text{in}\quad \UU,\\
        u_{\ep}>0\quad\text{in}\quad\UU,\quad\displaystyle\int_{\UU}u_{\ep}=1.
    \end{cases}
\end{equation*}

\item[(2)] For any open neighbourhood $\OO$ of the global attractor $\AAa_{\UU}$ (see Remark \ref{rem-application}), there exist $0<\ep_\OO\ll1$ and $\ga_{\OO}>0$ such that
$\mu_{\ep}(\UU\setminus \OO)\leq e^{-\frac{\ga_{\OO}}{\ep^2}}$ for all $\ep\in(0,\ep_{\OO})$.
In particular, $\{\mu_{\ep}\}_{\ep}$ is tight. 
\end{itemize}
\end{prop}

The next result follows immediately from Theorem \ref{thm-main-result}. 

\begin{thm}\label{thm-LDP-SD}
    Assume that \eqref{main-sde} admits a uniform Lyapunov function and let $\mu_{\ep}$ be the unique stationary distribution of \eqref{main-sde} with the density $u_{\ep}$. 
\begin{itemize}
    \item[(1)] If the global attractor $\AAa_{\UU}$ is an equivalence class, then
    $\lim_{\ep\to0}\frac{\ep^{2}}{2}\ln u_{\ep}=-V_{\UU}$ in $C^{\alpha}(\UU^{*})$ for any $\alpha\in(0,1)$.
    
    \item[(2)] Suppose that $\Om\subsetneqq\UU$ is open, connected, and positively invariant under $\vp^{t}$, and that $\AAa_{\Om}$ is the maximal attractor of $\vp^{t}$ in $\Om$ and is an equivalence class. If $\left\{\frac{\mu_{\ep}|_{\Om}}{\mu_{\ep}(\Om)}\right\}_{\ep}$ is tight as probability measures on $\Om$, then 
    $\lim_{\ep\to0}\frac{\ep^{2}}{2}\ln \frac{u_{\ep}}{\mu_{\ep}(\Om)}=-V_{\Om}$ in $C^{\alpha}(\Om^{*})$ for any $\alpha\in(0,1)$.
\end{itemize}
\end{thm}

The condition that ``$\left\{\frac{\mu_{\ep}|_{\Om}}{\mu_{\ep}(\Om)}\right\}_{\ep}$ is tight" in the statement of Theorem \ref{thm-LDP-SD} (2) can be verified in some interesting situations.

\begin{cor}
    Assume that \eqref{main-sde} admits a uniform Lyapunov function and let $\mu_{\ep}$ be the unique stationary distribution of \eqref{main-sde} with the density $u_{\ep}$. Suppose that $\Om\subsetneqq\UU$ is open, connected, and positively invariant under $\vp^{t}$, and that $\AAa_{\Om}$ is the maximal attractor of $\vp^{t}$ in $\Om$ and is an equivalence class.
    \begin{itemize}
        \item[(1)] If there is $\OO\stst\Om$ such that $\lim_{\ep\to0}\mu_{\ep}(\OO)=1$, then
$\lim_{\ep\to0}\frac{\ep^{2}}{2}\ln u_{\ep}=-V_{\Om}$ in $C^{\alpha}(\Om^{*})$ for any $\alpha\in(0,1)$.

        \item[(2)] If $\Om\stst\UU$, then
          $\lim_{\ep\to0}\frac{\ep^{2}}{2}\ln \frac{u_{\ep}}{\mu_{\ep}(\Om)}=-V_{\Om}$ in $C^{\alpha}(\Om^{*})$ for any $\alpha\in(0,1)$.
    \end{itemize}
\end{cor}
\begin{proof}
To apply Theorem \ref{thm-LDP-SD} (2), we need to check that $\left\{\frac{\mu_{\ep}|_{\Om}}{\mu_{\ep}(\Om)}\right\}_{\ep}$ is tight. In the case of (1), this is an immediate result of $\lim_{\ep\to0}\frac{\mu_{\ep}(\Om\setminus\overline{O})}{\mu_{\ep}(\Om)}=0$. The conclusion follows from Theorem \ref{thm-LDP-SD} (2) and the fact that $\lim_{\ep\to0}\mu_{\ep}(\Om)=1$.

(2) Given the positive invariance of $\Om$ under $\vp^{t}$, the proof of \cite[Theorem 3.2]{XCJ23} (more precisely, the fourth line on \cite[page 86]{XCJ23}) shows particularly that 
for each $x\in\partial\Om$, there are small $\de_{x}^{1}>0$, $\de_{x}^{2}>0$, $\ka_{x}>0$ and $\ep_{x}>0$ such that $\mu_{\ep}(B_{\de_{x}^{1}}(x))\leq e^{-\frac{\ka_{x}}{\ep^{2}}}\mu_{\ep}(B_{\de_{x}^{2}}(\AAa_{\Om}))$ for all $\ep\in(0,\ep_{x})$. Here, $B_{\de}(x)$ and $B_{\de}(\mathcal{A}_{\Om})$ are $\de$-neighborhoods of $x$ and $\mathcal{A}_{\Om}$, respectively. This together with the boundedness of $\Om$ yields the tightness of the family $\left\{\frac{u_{\ep}|_{\Om}}{\mu_{\ep}(\Om)}\right\}_{\ep}$. The result then follows from Theorem \ref{thm-LDP-SD} (2).
\end{proof}




\subsection{Quasi-stationary distributions}\label{subsec-application-qsd}

Let $\Om\subset\UU$ be open and connected and denote by $T_{\Om}^{\ep}$ the first time that $X_{t}^{\ep}$ exits $\Om$, namely, $T_{\Om}^{\ep}=\inf\{t>0:X_{t}^{\ep}\not\in\Om\}$. We assume that $\P_{x}^{\ep}\left[T_{\Om}^{\ep}<\infty\right]=1$ for all $x\in\Om$. Note that $\Om=\UU$ is not excluded. 

Recall that for each $\ep$, a probability measure $\mu$ on $\Om$ is called a \emph{quasi-stationary distribution} (QSD) of \eqref{main-sde} in $\Om$ if $\P^{\ep}_{\mu}[X_{t}^{\ep}\in\bullet|t<T_{\Om}^{\ep}]=\mu$ for all $t\geq0$. It is known from the general theory of QSDs (see e.g. \cite{MV12,CMS13}) that if $\mu_{\ep}$ is a QSD of \eqref{main-sde} in $\Om$, then there is $\la_{\ep}>0$ such that $\E_{\mu_{\ep}}^{\ep}[T_{\Om}^{\ep}>t]=e^{-\la_{\ep}t}$ for all $t\geq0$. The number $\la_{\ep}$ is often called the \emph{exit rate} (or escape rate, extinction rate) associated with $\mu_{\ep}$.

In the rest of this subsection, we only consider an open and connected set $\Om\stst\UU$. While it is possible to generalize our results to an open and connected $\Om\subset\UU$, achieving the tightness of QSDs $\{u_{\ep}\}_{\ep}$ often requires substantial efforts. We do not pursue this generalization here but refer the reader to Remark \ref{rem-discussion-QSD} below for relevant discussion. 

The following result about the existence and uniqueness of QSDs is particularly proven in \cite{GQZ88,CV18}.

\begin{prop}\label{prop-QSD-existence-uniqueness}
Suppose $\Om\stst\UU$ is open and connected. Then, for each $\ep$,  \eqref{main-sde} admits a unique QSD $\mu_{\ep}$ in $\Om$ with a positive density $u_{\ep}\in C^{2}(\Om)$. Moreover, $u_{\ep}$ and the associated exit rate $\la_{\ep}$ satisfy $L_{\ep}^{*}u_{\ep}=-\la_{\ep}u_{\ep}$ in $\Om$.
\end{prop}

In the next result, we establish the LDP for QSDs.

\begin{thm}
Suppose that $\Om\stst\UU$ is open, connected, and positively invariant under $\vp^{t}$, and that $\AAa_{\Om}$ is the maximal attractor of $\vp^{t}$ in $\Om$ and is an equivalence class. Denote by $\mu_{\ep}$ with density $u_{\ep}$ the unique QSD of \eqref{main-sde} in $\Om$. If $\Om$ can be approximated by smooth domains in the Hausdorff distance, then $\lim_{\ep\to0}\frac{\ep^{2}}{2}\ln u_{\ep}=-V_{\Om}$ in $C^{\alpha}(\Om^{*})$ for any $\alpha\in(0,1)$.
\end{thm}
\begin{proof}
For any open set $\OO\subset\Om$, we denote by $T^{\ep}_{\OO}$ the first time that $X_{t}^{\ep}$ exits $\OO$, that is, $T_{\OO}^{\ep}=\inf\{t>0:X_{t}^{\ep}\not\in\OO\}$. By the definition of QSDs, for any open set $\OO\subset\Om$, 
\begin{equation}\label{equality-june-26}
\begin{split}
\mu_{\ep}(\Om\setminus\ol{\OO})=\frac{\P^{\ep}_{\mu_{\ep}}\left[X_{t}^{\ep}\in\Om\setminus\ol{\OO},t<T_{\Om}^{\ep}\right]}{\P^{\ep}_{\mu_{\ep}}[t<T_{\Om}^{\ep}]}=\P^{\ep}_{\mu_{\ep}}\left[X_{t}^{\ep}\in\Om\setminus\ol{\OO},t<T_{\Om}^{\ep}\right]e^{\la_{\ep}t}.
\end{split}
\end{equation}
Given the dynamical properties of $\vp^{t}$ on $\Om$, it is known (see e.g. \cite{Friedman72/73}) that  $\limsup_{\ep\to0}\frac{\ep^{2}}{2}\ln\la_{\ep}<0$. 

We are going to find an open set $\OO_{0}$ satisfying $\AAa_{\Om}\subset\OO_0\stst\Om$ and a time scale $t_{\ep}$ such that $\lim_{\ep\to0}\la_{\ep}t_{\ep}=0$ and $\limsup_{\ep\to0}\frac{\ep^2}{2}\ln \P^{\ep}_{\mu_{\ep}}\left[X_{t_{\ep}}^{\ep}\in\Om\setminus\ol{\OO}_{0},t_{\ep}<T_{\Om}^{\ep}\right]<0$ so that setting $\OO=\OO_0$ and $t=t_\ep$ in \eqref{equality-june-26} yields
\begin{equation}\label{estimate-nov-05-2024}
\limsup_{\ep\to0}\frac{\ep^2}{2}\ln\mu_{\ep}(\Om\setminus\overline{\OO}_0)<0. 
\end{equation}
In particular, $\{\mu_{\ep}\}_{\ep}$ is tight, and hence, the result follows from Theorem \ref{thm-main-result}.

Let $\NN$ be an open, smooth, and narrow neighbourhood of $\partial\Om$ and $\OO_0$ be open, connected, and positively invariant under $\vp^{t}$ and satisfy $\Om\setminus\overline{\NN}\stst\OO_{0}\stst\Om$. Such a $\OO_0$ always exists. Indeed, let $\OO_{00}$ be open, connected, and satisfy $\Om\setminus\overline{\NN}\stst\OO_{00}\stst\Om$, and set $\OO_0:=\cup_{t\geq0}\vp^{t}(\OO_{00})$. Then, $\OO_0$ obviously satisfies all the conditions except $\ol{\OO}_{0}\subset\Om$. Suppose this is not the case. Since there exists $t_{00}>0$ such that $\cup_{t\geq t_{00}}\vp^{t}(\OO_{00})$ is contained in a small neighbourhood of $\AAa$, we find convergent sequences $\{t_{n}\}_{n}\subset(0,t_{00})$ and $\{x_{n}\}_{n}\subset\OO_{00}$ such that $\lim_{n\to\infty}\vp^{t_{n}}(x_{n})\in\partial\Om$. Setting $t_{*}:=\lim_{n\to\infty}t_{n}$ and $x_{*}:=\lim_{n\to\infty}x_{n}$, we conclude $\vp^{t_{*}}(x_{*})\in\partial\Om$, leading to a contradiction.

Split    
\begin{equation}\label{split-june-30}
        \begin{split}
\P^{\ep}_{\mu_{\ep}}\left[X_{t}^{\ep}\in\Om\setminus\ol{\OO}_0,t<T_{\Om}^{\ep}\right]
            &=\int_{\Om\cap\NN}\P^{\ep}_{\bullet}\left[X_{t}^{\ep}\in\Om\setminus\ol{\OO}_0,t<T_{\Om}^{\ep},T_{\Om\cap\NN}^{\ep}>t\right]d\mu_{\ep}\\
            &\quad+\int_{\Om\cap\NN}\P^{\ep}_{\bullet}\left[X_{t}^{\ep}\in\Om\setminus\ol{\OO}_0,t<T_{\Om}^{\ep},T_{\Om\cap\NN}^{\ep}\leq t\right]d\mu_{\ep}\\
            &\quad+\int_{\Om\setminus\NN}\P^{\ep}_{\bullet}\left[X_{t}^{\ep}\in\Om\setminus\ol{\OO}_0,t<T_{\Om}^{\ep}\right]d\mu_{\ep}\\
            &=:I_{1}(\ep,t)+I_{2}(\ep,t)+I_{3}(\ep,t).
        \end{split}
    \end{equation}

Since $w_{\ep}:=\E_{\bullet}^{\ep}[T_{\NN}^{\ep}]$ satisfies $L_{\ep}w_{\ep}=-1$ in $\NN$ and $w_{\ep}=0$ on $\partial\NN$, we can adapt the proof of the classical a priori estimate for Poisson equations (see e.g. \cite{GT01,HL97}) to show the existence of $\ga_{\NN}>0$ satisfying that $\ga_{\NN}$ approaches $0$ as $\NN$ narrows down to $\partial\Om$ such that
\begin{equation}\label{claim-july-2}
    \limsup_{\ep\to0}\frac{\ep^{2}}{2}\ln\left(\sup_{\NN}\E_{\bullet}^{\ep}[T_{\NN}^{\ep}]\right)\leq\ga_{\NN}.
\end{equation}

Define
\begin{equation*}
    V_{\ol{\OO}_{0}}(x):=\inf_{\phi\in\Phi_{x,\ol{\OO}_{0}}}\frac{1}{4}\int_{-\infty}^0\left[b(\phi)-\dot{\phi}\right]\cdot A^{-1}(\phi)\left[b(\phi)-\dot{\phi}\right], \quad \forall x\in \ol{\OO}_{0},
\end{equation*}
where $\Phi_{x,\ol{\OO}_{0}}:=\left\{\phi\in AC((-\infty,0];\ol{\OO}_{0}):\dot{\phi}\in L^2_{loc}((-\infty,0];\R^{d}),\phi(0)=x,\,\,\lim_{t\to -\infty}\dist(\phi(t),\AAa)=0\right\}$. Set $V_{0}:=\inf_{\partial\OO_{0}}V_{\ol{\OO}_{0}}>0$. Then, for any open $\OO\stst\OO_0$, there are $\ga_{\OO}>0$ and $0<\ep_{\OO}\ll1$ such that
\begin{equation}\label{claim-Nov-09-2024}
    \sup_{x\in\OO}\P_{x}^{\ep}\left[T^{\ep}_{\OO_{0}}\leq e^{\frac{2}{\ep^{2}}\frac{V_{0}}{2}}\right]\leq e^{-\frac{\ga_{\OO}}{\ep^{2}}},\quad \forall \ep\in(0,\ep_{\OO}].
\end{equation}
Indeed, following arguments in the proof of \cite[Theorem 5.7.11]{DZ98} (more precisely, arguments leading to the first inequality in \cite[page 230]{DZ98}), we find that for any open $\OO\stst\OO_0$, there exist $\ga_{\OO}>0$ and $0<\ep_{\OO}\ll1$ such that
$$
\P_{x}^{\ep}\left[T^{\ep}_{\OO_{0}}\leq e^{\frac{2}{\ep^{2}}\frac{V_{0}}{2}}\right]\leq \P_{x}^{\ep}\left[X^{\ep}_{T^{\ep}_{\OO_{0}\setminus\ol{\AAa}_{\de}}}\in\partial\OO_{0}\right]+\frac{1}{2}e^{-\frac{\ga_{\OO}}{\ep^{2}}},\quad \forall x\in \OO,\,\, \ep\in(0,\ep_{\OO}],
$$
where $\AAa_{\de}$ is the $\de$-neighbourhood of $\AAa$ for some fixed small $\de$ so that $\AAa_{\de}\stst\OO_{0}$. Making $\ga_{\OO}$ and $\ep_{\OO}$ smaller if necessary, we conclude from the sample path large deviation, the positive invariance of $\OO_{0}$ under $\vp^{t}$, and the existence of $t_{\OO}>0$ such that $\vp^{t}(\OO)\subset\AAa_{\de}$ for all $t\geq t_{\OO}$ that
$$
\sup_{x\in\OO}\P_{x}^{\ep}\left[X^{\ep}_{T^{\ep}_{\OO_{0}\setminus\ol{\AAa}_{\de}}}\in\partial\OO_{0}\right]\leq \frac{1}{2}e^{-\frac{\ga_{\OO}}{\ep^{2}}},\quad\forall \ep\in(0,\ep_{\OO}].
$$
Hence, \eqref{claim-Nov-09-2024} holds.

Let us fix $\NN$ so narrow that $\ga_{\NN}<\frac{V_0}{2}$ and set $t_\ep:=e^{\frac{2}{\ep^{2}}\frac{V_0}{2}}$. It is known that $\limsup_{\ep\to0}\frac{\ep^{2}}{2}\ln\la_{\ep}\leq-\frac{2V_0}{3}$ (see e.g. \cite{Friedman72/73}), and hence, $\lim_{\ep\to0}\la_{\ep}t_{\ep}=0$. 
It remains to treat \eqref{split-june-30} with $t=t_\ep$. It follows from Chebyshev's inequality and \eqref{claim-july-2} that
\begin{equation*}
\limsup_{\ep\to0}\frac{\ep^2}{2}\ln I_{1}(\ep,t_{\ep})\leq\limsup_{\ep\to0}\frac{\ep^2}{2}\ln\int_{\NN}\P^{\ep}_{\bullet}[T_{\NN}^{\ep}>t_\ep]d\mu_{\ep}\leq\ga_{\NN}-\frac{V_0}{2}<0.
\end{equation*}
Applying \eqref{claim-Nov-09-2024}, we find the existence of $\ga_{\Om\setminus\NN}>0$ such that
\begin{equation*}
\begin{split}
\limsup_{\ep\to0}\frac{\ep^2}{2}\ln I_{3}(\ep,t_{\ep})&\leq \limsup_{\ep\to0}\frac{\ep^2}{2}\ln\left(\sup_{x\in\Om\setminus\NN}\P^{\ep}_{x}\left[X_{t_{\ep}}^{\ep}\in\Om\setminus\ol{\OO}_0,t_{\ep}<T_{\Om}^{\ep}\right]\right)\\
&\leq\limsup_{\ep\to0}\frac{\ep^2}{2}\ln\left(\sup_{x\in\Om\setminus\NN}\P^{\ep}_{x}\left[T_{\OO_0}^{\ep}\leq t_{\ep}\right]\right)\leq-\frac{\ga_{\Om\setminus\NN}}{2}<0.
\end{split}    
\end{equation*}
For $I_{2}(\ep,t)$, we note that for each $x\in\Om\cap\NN$, the event $\left[t<T_{\Om}^{\ep},T_{\Om\cap\NN}^{\ep}\leq t\right]$ yields $X_{T_{\Om\cap\NN}^{\ep}}^{\ep}\in\partial_{1}\NN:=\partial\NN\cap\Om$. Thus,  the strong Markov property implies that
\begin{equation*}
\begin{split}
&\P^{\ep}_{x}\left[X_{t}^{\ep}\in\Om\setminus\ol{\OO}_0,t<T_{\Om}^{\ep},T_{\Om\cap\NN}^{\ep}\leq t\right]\\
&\qquad=\E_{x}\left[1_{\left[T_{\Om\cap\NN}^{\ep}\leq t\right]}\P^{\ep}_{X_{T_{\Om\cap\NN}^{\ep}}^{\ep}}\left[X_{t-T_{\Om\cap\NN}^{\ep}}^{\ep}\in\Om\setminus\ol{\OO}_0,t-T_{\Om\cap\NN}^{\ep}<T_{\Om}^{\ep}\right]\right]\leq\sup_{x\in\partial_{1}\NN}\P^{\ep}_{x}\left[T_{\OO_0}^{\ep}\leq t\right].
\end{split}    
\end{equation*}
It then follows from \eqref{claim-Nov-09-2024} that there exists $\ga_{\partial_{1}\NN}>0$ such that
$$
\limsup_{\ep\to0}\frac{\ep^2}{2}\ln I_{2}(\ep,t_{\ep})\leq\limsup_{\ep\to0}\frac{\ep^2}{2}\ln\left(\sup_{x\in\partial_{1}\NN}\P^{\ep}_{x}\left[T_{\OO_0}^{\ep}\leq t_{\ep}\right]\right)\leq-\frac{\ga_{\partial_{1}\NN}}{2}<0.
$$
Therefore, $\limsup_{\ep\to0}\frac{\ep^2}{2}\ln \P^{\ep}_{\mu_{\ep}}\left[X_{t_{\ep}}^{\ep}\in\Om\setminus\ol{\OO}_{0},t_{\ep}<T_{\Om}^{\ep}\right]<0$ holds. This completes the proof.
\end{proof}

\begin{rem}\label{rem-discussion-QSD}
When $\Om\subset\UU$ but not $\Om\stst\UU$, the existence and uniqueness of QSDs of \eqref{main-sde} in $\Om$ have been studied in various settings (see \cite{GQZ88, CCLMMS09,MV12,CMS13,HK19,CV18,GRS21,HQSY,GNW2020} and references therein). Notably, when $\Om$ is unbounded, the uniqueness typically requires strong dissipative conditions near infinity, such as coming down from infinity \cite{CCLMMS09}. Otherwise, the scenario of infinitely many QSDs could happen (see e.g. \cite{MS94,LS00,CMS13,Yamato22}).  

Even when a uniform-in-noise strong dissipative condition near infinity is satisfied, establishing the tightness of QSDs can still suffer from properties of the vector field and the noise coefficient at boundary points where trajectories may exit the domain. For instance, let us consider the stochastic logistic equation:
\begin{equation}\label{stochastic-logistic-equation}
    \dot{x}=x(1-x)+\ep\sqrt{x}\dot{W}_{t}. 
\end{equation}
It is shown in \cite{CCLMMS09} that for each $\ep$, \eqref{stochastic-logistic-equation} admits a unique QSD $\mu_{\ep}$ in $\Om=(0,\infty)$. Note that the vector field vanishes at the boundary $\partial\Om=\{0\}$ and the noise coefficient shows both degeneracy and singularity at $\partial\Om$. The tightness of $\{\mu_{\ep}\}_{\ep}$ is established in \cite{SWY2024} by leveraging the one-dimensional structure of \eqref{stochastic-logistic-equation}. 

However, extending this result remains an open problem, even for the relatively simple stochastic competitive Lotka-Volterra system:
\begin{equation}\label{stochastic-competitive-system}
    \begin{cases}
        \dot{x}_{1}=x_{1}(r_{1}-c_{11}x_{1}-c_{12}x_{2})+\ep\sqrt{x_{1}}\dot{W}_{t}^{1},\\
        \dot{x}_{2}=x_{2}(r_{2}-c_{22}x_{2}-c_{21}x_{1})+\ep\sqrt{x_{2}}\dot{W}_{t}^{2},        
    \end{cases}
        \end{equation}
where $r_{i}>0$ and $c_{ij}>0$ for $i,j=1,2$. For each $\ep$, the system \eqref{stochastic-competitive-system} admits a unique QSD $\mu_{\ep}$ in $\Om=(0,\infty)^{2}$ \cite{CV18,HQSY}, but the tightness of $\{\mu_{\ep}\}_{\ep}$ remains unknown. Note that the vector field has vanishing components on the boundary $\partial\Om=\left(\{0\}\times(0,\infty)\right)\cup(0,0)\cup\left((0,\infty)\times\{0\}\right)$ and the noise coefficient shows both degeneracy and singularity on $\partial\Om$. 
\end{rem}


\subsection{Macroscopic fluctuation theory}\label{subsec-macro-potential-fluctuation}

In this subsection, we assume {\bf(H1)}-{\bf(H4)} are satisfied and $\{u_{\ep}\}_{\ep}$ are either densities of stationary or quasi-stationary distributions of \eqref{main-sde}. For stationary distributions, $\Om=\UU$ and $\AAa$ is the global attractor.

In stochastic thermodynamics, the vector field $\ga_{\ep}$, defined in Corollary \ref{cor-macro-potential-flux}, is called \emph{Onsager's thermodynamic flux}. Note that $\nabla\cdot(u_{\ep}\ga_{\ep})=\la_{\ep}u_{\ep}$. In the case of a stationary distribution (i.e., $\la_{\ep}=0$), zero flux $\ga_{\ep}\equiv0$ is equivalent to the detailed balance (see e.g. \cite{JQQ04}). Therefore, $\ga_{\ep}\not\equiv0$ if and only if \eqref{SDE-introduction} is out of equilibrium, or irreversible. In which case, $\ga_{\ep}$ describes irreversible cyclic fluctuations in the steady state. 

The limits of $V_{\ep}$ and $\ga_{\ep}$ are respectively the quasi-potential $V$ and $\ga$, which are responsible for macroscopic dissipation and fluctuations. In fact, the relaxation dynamics to the attractor $\AAa$ is described by 
$$
\dot{x}=b(x)=-A(x)\nabla V(x)+\ga(x).
$$
Fluctuations away from the attractor $\AAa$ are rare events, which occur with overwhelming probability along paths satisfying 
$$
\dot{x}=2A(x)\nabla V(x)+b(x)=A(x)\nabla V(x)+\ga(x).
$$
The orthogonality $\nabla V\cdot\ga=0$ and Lyapunov property $b\cdot\nabla V=-\|A\nabla V\|^{2}_{A^{-1}}\leq0$ imply that $\ga$ describes cyclic transitions, while $V$ or $A\nabla V$ is responsible for both the relaxation to $\AAa$ and fluctuation away from $\AAa$. Moreover, when $\ga\not\equiv0$, fluctuation paths are not simply reversed relaxation paths.  

The potential function $V$ is often regarded as the generalized free energy or relative entropy in the classical irreversible thermodynamics. More precisely, if $x(t)$ solves $\dot{x}=b(x)$ with initial condition $x(0)=x_0$, Corollary \ref{cor-macro-potential-flux} (4)-(5) implies that
\begin{equation}\label{intro-energy-balance-eqn}
\frac{d}{dt}V(x(t))=-\|A(x(t))\nabla V(x(t))\|^{2}_{A^{-1}(x(t))}=\|\ga(x(t))\|^{2}_{A^{-1}(x(t))}-\|b(x(t))\|^{2}_{A^{-1}(x(t))},    
\end{equation}
which is an instantaneous energy balance law with free energy dissipation $\|A(x(t))\nabla V(x(t))\|^{2}_{A^{-1}(x(t))}$, house-keeping heat $\|\ga(x(t))\|^{2}_{A^{-1}(x(t))}$, and total entropy production $\|b(x(t))\|^{2}_{A^{-1}(x(t))}$ \cite{QCY2020,HQ20}. 

For reader's convenience, we include Appendix \ref{sec-app-energy-balance-equation} to brief on the formal derivation of \eqref{intro-energy-balance-eqn} from the energy balance equation for the relative entropy or Kullback–Leibler divergence $\int_{\Om}p_{\ep}(x,t)\ln\frac{p_{\ep}(x,t)}{u_{\ep}(x)}dx$ in the case of stationary distributions, where $p_{\ep}(x,t)$ is the density of the distribution of $X_{\ep}(t)$ starting at $X_{\ep}(0)=x_0$. This is closely related to de Bruijn's identity and Boltzmann's $H$-theorem. 

Recently, there have been exciting applications of nonequilibrium stochastic thermodynamics to nonequilibrium complex systems arising from neural network dynamics, ecology, cell biology, and so on (see e.g. \cite{WXW08,Wang15,FKLW2019,XPSLW21}). Built on the kinetic potential $V_{\ep}$ and Onsager's thermodynamic flux $\ga_{\ep}$ and their limits $V$ and $\ga$, the potential landscape and flux framework is formulated to address issues such as the origin of the underlying driving force for nonequilibrium systems and the underlying mechanisms for different physical and biological systems.


\subsection{Sub-exponential LDP}\label{subsec-subexponential-LDP}
In this subsection, we assume {\bf(H1)}-{\bf(H4)} are satisfied and $\{u_{\ep}\}_{\ep}$ are either densities of stationary or quasi-stationary distributions $\{\mu_{\ep}\}_{\ep}$ of \eqref{main-sde}. The LDP given in Theorem \ref{thm-main-result} captures the leading exponential asymptotic of $u_{\ep}$. Since $V$ is non-negative and vanishes only on $\AAa$, it implies that $\mu_{\ep}$ tends to concentrate on the attractor $\AAa$. To better understand the asymptotic of $u_{\ep}$, it is in general necessary to study the sub-exponential asymptotic, that is, the asymptotic of the prefactor $R_{\ep}:=K_{\ep}u_{\ep}e^{\frac{\ep^2}{2}V}$, where $K_{\ep}$ is the normalization constant. Straightforward calculations show that $R_{\ep}$ satisfies the following singularly perturbed linear equation
\begin{equation*}\label{eqn-prefactor}
    \begin{split}
        &\frac{\ep^{2}}{2}\sum_{i,j=1}^{d}a^{ij}\partial_{ij}^{2}R_{\ep}-\sum_{i=1}^d\left(b^i+2\sum_{j=1}^d a^{ij}\pa_j V-\ep^2\sum_{j=1}^d\pa_j a^{ij}\right)\pa_i R_{\ep}\\
        &\qquad -\left(\nabla\cdot b-\la_{\ep}+\sum_{i,j=1}^d(a^{ij}\pa^2_{ij}V+2\pa_i a^{ij}\pa_j V-\frac{\ep^2}{2}\pa^2_{ij}a^{ij})\right)R_{\ep}=0\quad \text{in}\quad G,
    \end{split}
\end{equation*}
where $G$ is given in Theorem \ref{thm-global-regularity}. Note that the regularity of $V$ gives rise to the regularity of coefficients, laying the solid foundation for studying the asymptotic of $R_{\ep}$.

When the attractor $\AAa$ is a linearly stable equilibrium and $\{u_{\ep}\}_{\ep}$ are stationary distributions in the whole space $\R^{d}$ instead of a general open and connected $\UU\subset\R^{d}$, it is shown in \cite[Theorem 3]{Day87} that there is a positive function $R_{0}\in C^{1}(G)$ such that $\lim_{\ep\to0}R_{\ep}=R_{0}$ holds locally uniformly in $G$. When it comes to a more general attractor, the situation is much more complicated since the asymptotic of $R_{\ep}$, when restricted to $\AAa$, must reflect invariant measures of $\vp^{t}$, which may be singular with respect to the volume measure on $\AAa$. Therefore, possible complex dynamics of $\vp^{t}$ on $\AAa$ must play a crucial role determining the asymptotic of $R_{\ep}$.




The sub-exponential LDP of $u_{\ep}$ has many applications that extend far beyond the LDP. We here mention a few. (i) When the attractor $\AAa$ is singleton set, it is used in \cite{BR2016} to formally derive the Eyring-Kramers formula in a randomly perturbed bistable system and in \cite[Theorem 4]{Day87} to prove that determining the asymptotic of the exit measure (associated with the exit problem from a smooth bounded domain with a non-characteristic boundary) is equivalent to the asymptotic evaluation of a Laplace integral on the boundary. These results are expected in a more general setting. (ii) The computation of a stationary or quasi-stationary distribution in the small noise regime on a domain of large size has to be done on a relatively smaller one. The sub-exponential LDPs (for the original stationary or quasi-stationary distribution and its restriction on a smaller domain with appropriate boundary conditions) would provide a theoretical foundation for this. (iii) As mentioned above, the asymptotic of $R_{\ep}$ must reflect invariant measures of $\vp^{t}$, and hence, will have the stability of invariant measures of $\vp^{t}$ under noise perturbations as consequences.

\appendix

\section{\bf Maximal attractors}\label{appendix-attractor}

Let $\UU\subset \R^d$ be open and connected and $b:\UU\to \R^d$ be locally Lipschitz continuous. Denote by $\vp^t$ be the local flow generated by the ODE \eqref{main-ode}. 

\begin{defn}[Maximal attractor]\label{defn-attractor}
Let $\Om\subset\UU$ be open, connected, and positively invariant under $\vp^{t}$. A compact $\vp^{t}$-invariant set $\AAa\subset\Om$ is called the \emph{maximal attractor} in $\Om$ of $\vp^{t}$ if $\lim_{t\to\infty}\dist_{H}(\vp^{t}(\OO),\AAa)=0$ for all $\OO\stst\Om$, where $\dist_{H}$ is the Hausdorff semi-distance. 
     
     When $\Om=\UU$, $\AAa$ is called the \emph{global attractor}. Otherwise, $\AAa$ is called a \emph{local attractor}. 
\end{defn}

\begin{prop}[{\cite[Proposition 5.2]{HJLY15-JDDE}}]
    Let $\Om\subset\UU$ be open, connected, and positively invariant under $\vp^{t}$. $\vp^{t}$ admits the maximal attractor in $\Om$ if and only if it is dissipative in $\Om$, that is, there is a compact set $K\subset\Om$ such that for any $x\in\Om$, there is $t_{x}\geq0$ such that $\vp^{t}(x)\in K$ for all $t\geq t_{x}$.
\end{prop}

\begin{defn}[Lyapunov function]
    A function $U\in C^{1}(\Om)$ is called a \emph{Lyapunov function} of $\vp^{t}$ in $\Om$ if $\lim_{x\to\partial\Om}U=\sup_{\Om}U$ and there exist a compact set $K\subset\Om$ and a constant $\gamma>0$ such that $b\cdot\nabla U\leq-\ga$ in $\Om\setminus K$.
\end{defn}

\begin{prop}[{\cite[Proposition 5.2]{HJLY15-JDDE}}]\label{thm-existence-ma-attractor}
Let $\Om\subset\UU$ be open and connected. Suppose $\vp^{t}$ admits a Lyapunov function in $\Om$. Then, $\Om$ is positively invariant under $\vp^{t}$ and $\vp^{t}$ admits the maximal attractor in $\Om$.    
\end{prop}

In Definition \ref{defn-attractor}, no ``indecomposability condition" on a maximal attractor is imposed, and therefore, it may contain smaller maximal attractors. The next result discusses some indecomposability conditions and their relationships.

\begin{prop}\label{app-prop-indecomposable-attractor}
    Let $\Om\subset\UU$ be open, connected, and positively invariant under $\vp^{t}$, and $\AAa$ be the maximal attractor in $\Om$. Consider the following statements. 
    \begin{itemize}
        \item[(1)] $\AAa$ is an equivalence class in the sense of {\bf(H3)}. 

        \item[(2)] $\AAa$ is chain-transitive, that is, for any $x,y\in\AAa$, $\ep>0$, and $T>0$, there is a finite sequence $\left\{x_1=x,x_2,\dots,x_{n},x_{n+1}=y;\,\,t_{1},\dots,t_{n}\right\}$ with $x_{i}\in\Om$ and $t_{i}\geq T$ such that $\left|\vp^{t_{i}}(x_{i})-x_{i+1}\right|<\ep$ for all $i$.

        \item[(3)] $\AAa$ does not contain a smaller maximal attractor. 
    \end{itemize}
Then, (1) implies (2), which is equivalent to (3). 
\end{prop}
\begin{proof}
    The equivalence between (2) and (3) is a consequence of a celebrated result of C. Conley (see \cite{Conley1978,Hurley1995}).  For (1)$\implies$(2), we point out that the problem can be reduced to the discrete-time case (see \cite[Proposition 5.5]{Kifer88} and \cite[Theorem 5]{Hurley1995}), which then follows from \cite[Lemma 3.1]{Kifer89}.
\end{proof}

\begin{rem}\label{app-rem-indecomposable-attractor}
    (2) does not imply (1) in general, although it does not seem easy to construct a counterexample. Following \cite[Proposition 4.1]{Kifer88}, it is not hard to formulate a condition under which (2) implies (1).
\end{rem}


\section{\bf Invariant manifold theory}\label{app-Fenichel-theory}

We recall and adapt some classical results from \cite{Feni73,Feni77,HPS77} regarding the invariant manifold theory of dynamical systems that are used in Subsection \ref{subsec-Hamiltonian-system} to construct the local unstable manifold. 

Let us consider the ODE \eqref{main-ode} with $b\in C^k(\UU)$ for $k\geq2$. Denote by $x\cdot t$ the local flow, namely, $x\cdot t=\vp^t(x)$. Let $\Psi(x,t)$ be the principal fundamental matrix solution at initial time $0$ of the linearization of \eqref{main-ode} along $x\cdot t$:
\begin{equation*}\label{app-linear-ode}
    \dot{y}=\nabla b(x\cdot t) y,\quad y\in \R^d,
\end{equation*}
where $\nabla b$ denotes the Jacobian of $b$.  

\begin{itemize}
    \item[\bf (H)] We assume that \eqref{main-ode} admits a $C^k$ compact, connected, and invariant submanifold $\MM$  with dimension $d_{\MM}<d$ and there are constants $C>0$ and $0<\be_0<\be_*$ satisfying $\frac{\be_0}{\be_*-\be_0}<\frac{1}{k'}$ for some $1\leq k'\leq k-1$ such that for each $x\in\MM$,
    \begin{itemize}
    \item $T_x\R^d=E(x)\oplus N(x)$,
    \item $\Psi(x,t)E(x)=E(x\cdot t)$, $\Psi(x,t)N(x)=N(x\cdot t)$, $t\in \R$,
    \item $T_{x}\MM\subset N(x)$,
    \item $\|\Psi(x,t)|_{E(x)}\|\leq Ce^{\be_* t}$, $t\leq 0$, 
     \item $\|\Psi(x,t)|_{N(x)}\|\leq Ce^{\be_0 |t|}$, $t\in\R$.
\end{itemize}
\end{itemize}


\begin{thm}\label{app-thm-Fenichel}
Assume {\bf (H)}. Then, there exists a $C^{k'}$ local unstable manifold $\WW^{u}\subset \UU$, namely $\vp^t \WW^{u}\subset \WW^{u}$ for $t\leq 0$, with the following properties:
\begin{enumerate}
    \item $\MM\subset \WW^{u}$ and $T_x \WW^{u}=E(x)\oplus T_x \MM$ for each $x\in \MM$,
    
    \item $\WW^{u}=\cup_{x\in\MM}\WW_{x}^{u}$, where 
    \begin{itemize}
        \item [(a)] $\WW_x^{u}$ is a $C^{k'}$ manifold containing $x$,
        \item [(b)] $\WW_x^{u}$ is diffeomorphic to a ball in $E(x)$ and $T_x\WW_x^{u}=E(x)$,
        \item [(c)] $\WW_x^{u}\cap \WW_{x'}^{u}=\emptyset$ if $x\neq x'$,
        \item [(d)] $\vp^t \WW_x^{u}\subset \WW_{x\cdot t}^{u}$ for $t\leq 0$,
    \end{itemize}

    \item the projection $\pi: \WW^{u}\to \MM$ is $C^{k'}$ and for any $0<\be'<\be_*$, there exists $C>0$ such that 
    $$
        \dist(\vp^t(x),\pi(x)\cdot t)\leq C e^{\be' t},\quad\forall x\in \WW^{u}\andd t\leq 0.
    $$
\end{enumerate}
\end{thm}


\begin{rem}


If $\MM$ is just a set instead of a manifold, then all the results in  Theorem \ref{app-thm-Fenichel} except the regularity of $\WW^{u}$ are true (see \cite{Feni73,HPS77}). The regularity of $\WW^{u}$ when $\MM$ is a manifold is established in \cite{Feni73,Feni77}.  

\end{rem}


\section{\bf Energy balance law}\label{sec-app-energy-balance-equation}

In this appendix, we present a formal derivation of the instantaneous energy balance law \eqref{intro-energy-balance-eqn} from the energy balance law, and therefore, we do not bother with matters that are involved in the rigorous proof. A rigorous justification along this line remains unavailable. See Remark \ref{rem-energy-balance-law} below for explanation.

Consider the SDE \eqref{main-sde} and denote by $X_{t}^{\ep}$ its solution. Suppose it admits a unique stationary distribution $\mu_{\ep}$ with a positive density $u_{\ep}$. For $t>0$, let $p_{\ep}(x,t)$ be the positive density of the distribution of $X_{t}^{\ep}$ with initial condition $X_{0}^{\ep}=x_0$. 

Consider
\begin{itemize}
\item relative entropy or Kullback–Leibler divergence: $F_{\ep}(t)=\int_{\UU}p_{\ep}\ln\frac{p_{\ep}}{u_{\ep}}dx$;

\item Fisher information: $I_{\ep}(t)=\frac{2}{\ep^2}\int_{\UU}\left(\frac{\ep^2}{2}A\nabla\ln\frac{p_{\ep}}{u_{\ep}}\right)\cdot A^{-1}\left(\frac{\ep^2}{2}A\nabla\ln\frac{p_{\ep}}{u_{\ep}}\right)p_{\ep}dx$;

\item entropy production rate: $e_{p}^{\ep}(t)=\frac{2}{\ep^{2}}\int_{\UU}\Ga_{\ep}\cdot A^{-1}\Ga_{\ep}p_{\ep}dx$, where $\Ga_{\ep}:=b-\frac{\ep^2}{2}\frac{\nabla\cdot (Ap_{\ep})}{p_{\ep}}$;

\item house-keeping exchange rate: $Q_{hk}^{\ep}(t)=\frac{2}{\ep^{2}}\int_{\UU}\ga_{\ep}\cdot A^{-1}\ga_{\ep}p_{\ep}dx$, where $\ga_{\ep}:=b-\frac{\ep^2}{2}\frac{\nabla\cdot (Au_{\ep})}{u_{\ep}}$ is Onsager's thermodynamic flux.
\end{itemize}
The energy balance equation \cite{VE2010,QCY2020} reads
\begin{equation}\label{energy-balance-equation}
F_{\ep}'(t)=-I_{\ep}(t)=Q_{hk}^{\ep}(t)-e_{p}^{\ep}(t),\quad t>0 
\end{equation}
The first equality in \eqref{energy-balance-equation} is known as de Bruijn's identity. Since $I_{\ep}(t)>0$, the relative entropy $F_{\ep}(t)$ decays in time. This latter property is often referred to as the Boltzmann's $H$-theorem. Moreover, since $Q_{hk}^{\ep}(t)\geq0$ and $e_{p}^{\ep}(t)>0$, they can be interpreted as the source and sink for the ``generalized free energy" $F_{\ep}(t)$. Note that the source $Q_{hk}^{\ep}(t)$ equals to $0$ if and only if $\ga_{\ep}\equiv0$ (that is, the system is an equilibrium system).

Assume that $u_{\ep}$ and $p_{\ep}$ admit sub-exponential LDPs (or zero-th order WKB expansions) as follows:
\begin{equation}\label{app-WKB}
u_{\ep}(x)=\frac{R(x)+O(\ep^2)}{K_{\ep}}e^{-\frac{2}{\ep^{2}}V(x)}\quad\text{and}\quad p_{\ep}(x,t)=\frac{\RR(x,t)+O(\ep^2)}{\KK_{\ep}}e^{-\frac{2}{\ep^{2}}\VV(x,t)},\quad t>0,   
\end{equation}
where $\VV(x(t),t)=0<\VV(x,t)$ for $x\in\UU\setminus\{x(t)\}$ and $\text{\rm Hess}(\VV)(x(t),t)$ is positive definite, where $x(t)$ is the unique solution of \eqref{app-ode} with initial condition $x(0)=x_0$. Applying Laplace's method, the following leading asymptotic can be derived:
\begin{equation*}
    \begin{split}
        F_{\ep}'(t)&=\frac{2}{\ep^2}\frac{d}{dt}V(x(t))+o\left(\frac{2}{\ep^2}\right),\\
        I_{\ep}(t)&=\frac{2}{\ep^2}\nabla V(x(t))\cdot A(x(t))\nabla V(x(t))+o\left(\frac{2}{\ep^2}\right),\\
        e_{p}^{\ep}(t)&=\frac{2}{\ep^2}b(x(t))\cdot A^{-1}(x(t))b(x(t))+o\left(\frac{2}{\ep^2}\right),\\
        Q_{hk}^{\ep}(t)&=\frac{2}{\ep^2}\ga(x(t))\cdot A^{-1}(x(t))\ga(x(t))+o\left(\frac{2}{\ep^2}\right),
    \end{split}
\end{equation*}
where $\ga=b-A\nabla V$. These together with \eqref{energy-balance-equation} give rise to the instantaneous energy balance law \eqref{intro-energy-balance-eqn}.

\begin{rem}\label{rem-energy-balance-law}
Assume that $b$ and $\si$ are sufficiently regular and $\si\si^{\top}$ is uniformly positive definite on $\ol{\UU}$. Then, the energy balance equation \eqref{energy-balance-equation} can be readily established if $\UU$ is a compact manifold or a bounded domain with reflecting boundary condition imposed on. Otherwise, dissipative conditions are needed to ensure that $p_{\ep}(x,t)$ and $u_{\ep}(x)$ have good decaying properties as $x\to\partial\UU$. 

The sub-exponential LDP for $u_{\ep}$ in \eqref{app-WKB} is only known when $b$ admits a non-degenerate and globally asymptotically stable equilibrium (see Subsection \ref{sub-application} for relevant discussions). That for $p_{\ep}$ remains unknown, except that the limit $\VV(x,t):=-\lim_{\ep\to0}\frac{\ep^{2}}{2}\ln p_{\ep}(x,t)$ has been established in \cite{Friedman197576,Sheu84}. 

\end{rem}


\bibliographystyle{amsplain}

\end{document}